\documentclass[letterpaper,11pt]{article}

\usepackage{amssymb,amsbsy,amsmath,amscd,amsthm,amsfonts,bbm}
\usepackage{enumerate,xspace}
\usepackage{cite}
\usepackage{dsfont}
\usepackage{graphicx}
\usepackage{longtable}
\usepackage{tablefootnote}
\usepackage[margin=1in]{geometry}  % set the margins to 1in on all sides
\usepackage{fullpage}

\usepackage[%dvips,
            CJKbookmarks=true,
            bookmarksnumbered=true,
            bookmarksopen=true,
            colorlinks=true,
            citecolor=red,
            linkcolor=blue,
            anchorcolor=red,
            urlcolor=blue
            ]{hyperref}

\newtheorem{definition}{Definition}
\newtheorem{theorem}{Theorem}

\newtheorem{lemma}{Lemma}
\newtheorem{remark}{Remark}

\newtheorem{corollary}{Corollary}

\newcommand{\Binom}{\mathrm{Bin}}

\usepackage{tikz}
\usetikzlibrary{matrix,arrows,calc,shapes,backgrounds}
\usetikzlibrary{shapes.callouts,decorations.text}

\usetikzlibrary{shapes.misc}

\tikzset{cross/.style={cross out, draw=black, minimum size=2*(#1-\pgflinewidth), inner sep=0pt, outer sep=0pt},
cross/.default={1pt}}

\tikzstyle{int}=[draw, fill=blue!20, minimum size=2em]
\tikzstyle{dot}=[circle, draw, fill=blue!20, minimum size=2em]
\tikzstyle{dotred}=[circle, draw, fill=red!20, minimum size=2em]
\tikzstyle{init} = [pin edge={to-,thin,black}]
\tikzstyle{initred} = [pin edge={to-,thin,red}]
\tikzstyle{plan}=[draw, fill=blue!20, minimum size=2em, text width=5em, rounded corners,align=center]
\tikzstyle{planwide}=[draw, fill=blue!20, minimum size=2em, text width=8em, rounded corners,align=center]
\tikzstyle{vertexdot}=[circle, draw, fill=white, minimum size=3,inner sep=0pt]

\newcommand{\Vertex}{\tikz[scale=0.5]{\draw (0,0) node (zero) [vertexdot] {};}}
\newcommand{\Edge}{\tikz[scale=0.5]{\draw (0,0) node (zero) [vertexdot] {} -- (0.5,0) node (zero) [vertexdot] {};}}
\newcommand{\PathThree}{\tikz[scale=0.5]{\draw (0,0) node (zero) [vertexdot] {} -- (0.5,0) node (zero) [vertexdot] {} -- (1.0,0) node (zero) [vertexdot] {};}}
\newcommand{\PathFour}{\tikz[scale=0.5]{\draw (0,0) node (zero) [vertexdot] {} -- (0.5,0) node (zero) [vertexdot] {} -- (1.0,0) node (zero) [vertexdot] {}-- (1.5,0) node (zero) [vertexdot] {};}}
\newcommand{\DoubleEdge}{\tikz[scale=0.5,baseline=(zero.base)]{\draw (0,0) node (zero) [vertexdot] {} -- (0.5,0) node[vertexdot] {};\draw (0,0.5) node[vertexdot] {} -- (0.5,0.5) node[vertexdot] {};}}
\newcommand{\Square}{\tikz[scale=0.5,baseline=(zero.base)]{\draw (0,0) node (zero) [vertexdot] {} -- (0.5,0) node[vertexdot] {} -- (0.5,0.5) node[vertexdot] {} -- (0,0.5) node[vertexdot] {} -- cycle;}}

\newcommand{\Triangle}{\tikz[scale=0.5,baseline=(zero.base)]{\draw (0,0) node (zero) [vertexdot] {} -- (0.5,0) node[vertexdot] {} -- (0.25,{0.5*sin(60)}) node[vertexdot] {} -- cycle;}}

\newcommand{\Paw}{\tikz[scale=0.5,baseline=(zero.base)]{\draw (0,0) node (zero) [vertexdot] {} -- (0.5,0) node[vertexdot] {} -- (0.25,{0.5*sin(60)}) node[vertexdot] {} -- cycle; \draw (0.5,0) node[vertexdot] {} -- (1,0) node[vertexdot] {};}}

\newcommand{\BrokenTriangle}{\tikz[scale=0.5,baseline=(zero.base)]{\draw (0,0) node (zero) [vertexdot] {} -- (0.25,{0.5*sin(60)}) node[vertexdot] {} -- (0.5,0) node[vertexdot] {};}}
\newcommand{\Kfour}{\tikz[scale=0.5,baseline=(zero.base)]{\draw (0,0)--(0.5,0.5); \draw (0.5,0)--(0,0.5); \draw (0,0) node (zero) [vertexdot] {} -- (0.5,0) node[vertexdot] {} -- (0.5,0.5) node[vertexdot] {} -- (0,0.5) node[vertexdot] {} -- cycle; }}

\renewcommand{\Diamond}{\tikz[scale=0.5,baseline=(zero.base)]{\draw (0.5,0)--(0,0.5); \draw (0,0) node (zero) [vertexdot] {} -- (0.5,0) node[vertexdot] {} -- (0.5,0.5) node[vertexdot] {} -- (0,0.5) node[vertexdot] {} -- cycle; }}

\renewcommand{\tilde}{\widetilde}

\newcommand{\stepa}[1]{\overset{\rm (a)}{#1}}
\newcommand{\stepb}[1]{\overset{\rm (b)}{#1}}

\newcommand{\floor}[1]{{\left\lfloor {#1} \right \rfloor}}

\newcommand{\naturals}{\mathbb{N}}

\newcommand{\Expect}{\mathbb{E}}
\newcommand{\expect}[1]{\mathbb{E}\left[#1\right]}
\newcommand{\Prob}{\mathbb{P}}
\newcommand{\prob}[1]{\mathbb{P}\left[#1\right]}
\newcommand{\pprob}[1]{\mathbb{P}[#1]}

\newcommand{\TV}{{\rm TV}}

\newcommand{\iid}{i.i.d.\xspace}

\newcommand{\hellinger}{H}

\newcommand{\pth}[1]{\left( #1 \right)}
\newcommand{\qth}[1]{\left[ #1 \right]}
\newcommand{\sth}[1]{\left\{ #1 \right\}}

\newcommand{\eqdistr}{{\stackrel{\rm law}{=}}}
\newcommand{\iiddistr}{{\stackrel{\text{\iid}}{\sim}}}
\newcommand{\sdistr}{{\stackrel{\text{ind.}}{\sim}}}
\newcommand{\Var}{\mathsf{Var}}
\newcommand{\Cov}{\mathsf{Cov}}

\newcommand{\Bern}{\mathrm{Bern}}
\newcommand{\Poisson}{\mathrm{Poisson}}

\newcommand{\Indc}{\mathbbm{1}}
\newcommand{\indc}[1]{\Indc\left\{{#1}\right\}}

\newcommand{\tG}{{\widetilde{G}}}
\newcommand{\tH}{{\widetilde{H}}}

\newcommand{\sfa}{{\mathsf{a}}}

\newcommand{\sfc}{{\mathsf{c}}}

\newcommand{\sfe}{{\mathsf{e}}}

\newcommand{\sfg}{{\mathsf{g}}}

\newcommand{\sfv}{{\mathsf{v}}}

\newcommand{\calC}{{\mathcal{C}}}

\newcommand{\calF}{{\mathcal{F}}}
\newcommand{\calG}{{\mathcal{G}}}
\newcommand{\calH}{{\mathcal{H}}}

\newcommand{\Th}{{^{\rm th}}}

\usepackage{prettyref}
\newrefformat{eq}{(\ref{#1})}
\newrefformat{thm}{Theorem~\ref{#1}}
\newrefformat{sec}{Section~\ref{#1}}
\newrefformat{algo}{Algorithm~\ref{#1}}
\newrefformat{fig}{Fig.~\ref{#1}}
\newrefformat{tab}{Table~\ref{#1}}
\newrefformat{rmk}{Remark~\ref{#1}}
\newrefformat{def}{Definition~\ref{#1}}
\newrefformat{cor}{Corollary~\ref{#1}}
\newrefformat{lmm}{Lemma~\ref{#1}}
\newrefformat{prop}{Proposition~\ref{#1}}
\newrefformat{app}{Appendix~\ref{#1}}
\newrefformat{ex}{Example~\ref{#1}}

\newcommand{\s}{\mathsf{s}}

\newcommand{\n}{\mathsf{n}}
\newcommand{\ind}{\s}

\newcommand{\cc}{\mathsf{cc}}
\newcommand{\cchat}{\widehat{\mathsf{cc}}}
\newcommand{\cctilde}{\widetilde{\mathsf{cc}}}

\newcommand{\de}{d}
\newcommand{\vertex}{v}

\usepackage{float}
\usepackage{caption}
\usepackage{subcaption}
\graphicspath {{graphs/}{experiments/}}

\usepackage{threeparttable}

\title{Estimating the Number of Connected Components in a Graph via Subgraph Sampling}
\date{\today}
\author{Jason M. Klusowski\thanks{Department of Statistics, Rutgers University -- New Brunswick, Piscataway, NJ, 8019, email: \url{jason.klusowski@rutgers.edu}.} \and Yihong Wu\thanks{Department of Statistics and Data Science, Yale University, New Haven, CT, 06511, email: \url{yihong.wu@yale.edu}. This research was supported in part by the NSF Grant IIS-1447879, CCF-1527105, an NSF CAREER award CCF-1651588, and an Alfred Sloan fellowship.}}

\begin{document}

\maketitle

\begin{abstract}

Learning properties of large graphs from samples has been an important problem in statistical network analysis since the early work of Goodman \cite{Goodman1949} and Frank \cite{Frank1978}. We revisit a problem formulated by Frank \cite{Frank1978} of estimating the number of connected components in a large graph based on the subgraph sampling model, in which we randomly sample a subset of the vertices and observe the induced subgraph. The key question is whether accurate estimation is achievable in the \emph{sublinear} regime where only a vanishing fraction of the vertices are sampled. We show that it is impossible if the parent graph is allowed to contain high-degree vertices or long induced cycles. For the class of chordal graphs, where induced cycles of length four or above are forbidden, we characterize the optimal sample complexity within constant factors and construct linear-time estimators that provably achieve these bounds. This significantly expands the scope of previous results which have focused on unbiased estimators and special classes of graphs such as forests or cliques.

Both the construction and the analysis of the proposed methodology rely on combinatorial properties of chordal graphs and identities of induced subgraph counts. They, in turn, also play a key role in proving minimax lower bounds based on construction of random instances of graphs with matching structures of small subgraphs. 
\end{abstract}

%each vertex independently with probability $p$ and observe the induced subgraph induced by the sampled vertices. 

%We show that it is possible by accessing only sublinear number of samples if the graph does not contain high-degree vertices or long induced cycles; otherwise it is impossible. We obtain optimal sample complexity bounds for several classes of graphs including forests, cliques, and more generally, chordal graphs, achieved by linear-time estimators.

\tableofcontents

\section{Introduction}
\label{sec:intro}

Counting the number of features in a graph -- ranging from basic local structures like motifs or graphlets (e.g., edges, triangles, wedges, stars, cycles, cliques) to more global features like the number of connected components -- is an important task in network analysis. For example, the global clustering coefficient of a graph (i.e.~the fraction of closed triangles) is a measure of the tendency for nodes to cluster together and a key quantity used to study cohesion in various networks \cite{Luce1949}. 
To learn these graph properties, applied researchers typically collect data from a random sample of nodes to construct a representation of the true network.
%This setting is largely due to resource and time constraints (e.g., in-person interviews conducted in remote locations) or an inability to access the full population (e.g., historical or observational data). 
We refer to these problems collectively as \emph{statistical inference on sampled networks}, where the goal is to infer properties of the parent network (population) from a subsampled version. Below we mention a few examples that arise in various fields of study.

\begin{itemize}
\item Sociology: Social networks of the Hadza hunter-gatherers of Tanzania were studied in \cite{Apicella2012} by surveying 205 individuals in 17 Hadza camps (from a population of 517). Another study \cite{Conley2010} of farmers in Ghana used network data from a survey of 180 households in three villages from a population of 550 households.
\item Economics and business: Low sampling ratios have been used in applied economics (such as 30\% in \cite{Fafchamps2003}), particularly for large scale studies \cite{Bandiera2006, Feigenberg2010}. A good overview of various experiments in applied economics and their corresponding sampling ratios can be found in \cite[Appendix F, p. 11]{Chandrasekhar2011}. Word of mouth marketing in consumer referral networks was studied in \cite{Reingen1986} using 158 respondents from a potential subject pool of 238.
\item Genomics: The authors of \cite{Stumpf2008} use protein-protein interaction data and demonstrate that it is possible to arrive at a reliable statistical estimate for the number of interactions (edges) from a sample containing approximately 1500 vertices. 
%Genomics has motivated the study of and defined network terminologies like motifs \cite{Alon2002} and graphlets \cite{Prvzulj2004, Prvzulj2007}. 

\item World Wide Web and Internet: Informed random IP address probing was used in \cite{Govindan2000} in an attempt to obtain a router-level map of the Internet.
\end{itemize}

%Furthermore, if the parent graph is so large that it is prohibitively expensive to store or computationally intensive to query each node, it may be more efficient to work with a sparse representation of the network from a partial sample. Thus, there are various statistical and computational issues associated with these counting tasks.

As mentioned earlier, a primary concern of these studies is how well the data represent the true network and how to reconstruct the relevant properties of the parent graphs from samples.
%For example, if too little of the graph is observed, it is likely that much of the parent structure will be lost. 
These issues and how they are addressed broadly arise from two perspectives: 

\begin{itemize}
\item The full network is unknown due to the lack of data, which could arise from the underlying experimental design and data collection procedure, e.g., 
%in-person interviews conducted in remote locations, 
historical or observational data. In this case, one needs to construct statistical estimators (i.e., functions of the sampled graph) to conduct sound inference. These estimators must be designed to account for the fact that the sampled network is only a partial observation of the true network, and thus subject to certain inherent biases and variability.

\item The full network is either too large to scan or too expensive to store. In this case, approximation algorithms can overcome such computational or storage issues that would otherwise be unwieldy. For example, for massive social networks, it is generally impossible to enumerate the whole population. Rather than reading the entire graph, query-based algorithms randomly (or deterministically) sample parts of the graph or adaptively explore the graph through a random walk \cite{Peres2017}. Some popular instances of traversal based procedures are snowball sampling \cite{Goodman1961} and respondent-driven sampling \cite{Salganik2004}. Indeed, sampling (based on edge and degree queries) is a commonly used primitive to speed up computation, which leads to various sublinear-time algorithms for testing or estimating graph properties such as the average degree 
%Feige2006
\cite{Goldreich2008}, triangle and more general subgraph counts 
%Gonen2011
\cite{Eden2015, aliakbarpour2017sublinear}, expansion properties \cite{GR11}; we refer the readers to the monograph \cite{Goldreich17}.
\end{itemize}

Learning properties of graphs from samples has been an important problem in statistical network analysis since the early work of Goodman \cite{Goodman1949} and Frank \cite{Frank1978}. Estimation of various properties such as graph totals \cite{Frank1977} and connectivity \cite{capobianco1972estimating,Frank1978} has been studied in a variety of sample models. However, most of the analysis has been confined to obtaining unbiased estimators for certain classes of graphs and little is known about their optimality.
The purpose of this paper is to initiate a systematic study of statistical inference on sampled networks, with the goal of determining their statistical limits in terms of minimax risks and sample complexity, achieved by computationally efficient procedures.

As a first step towards this goal, in this paper we focus on a representative problem introduced in \cite{Frank1978}, namely, estimating the number of connected components in a graph from a partial sample of the population network. In fact, the techniques developed in this paper are also useful for estimating other graph statistics such as motif counts, which were studied in the companion paper \cite{KlusowskiWu2017-motif}.

Before we proceed, let us emphasize that the objective of this paper is \emph{not} testing whether the graph is connected, which is a property too fragile to test on the basis of a small sampled graph; indeed, missing a single edge can destroy the connectivity.
Instead, our goal is to estimate the number of connected components with an optimal additive accuracy. Thus, naturally, it is applicable to graphs with a large number of components.

We study the problem of estimating the number of connected components for two reasons. First, it encapsulates many challenging aspects of statistical inference on sampled graphs, and we believe the mathematical framework and machinery developed in this paper will prove useful for estimating other graph properties as well.
Second, the number of connected components is a useful graph property that 
quantifies the connectivity of a network. In addition, it finds use in data-analytic applications related to determining the number of classes in a population \cite{Goodman1949}. Another example is the recent work \cite{ChenShrivastavaSteorts2017}, which studies the estimation of the number of documented deaths in the Syrian Civil War from a subgraph induced by a set of vertices obtained from an adaptive sampling process (similar to subgraph sampling). There, the goal is to estimate the number of unique individuals in a population, which roughly corresponds to the number of connected components in a network of duplicate records connected by shared attributes.

Next we discuss the sampling model, which determines how reflective the data is of the population graph and therefore the quality of the estimation procedure. There are many ways to sample from a graph (see \cite{Duffield2014,Leskovec2006} for a list of techniques and \cite{Kolaczyk2017, Kolaczyk2009, Handcock2010} for comprehensive reviews). 
 %but due to feasibility and practicality of implementation, we will focus on the simplest model which is subgraph sampling. 
For simplicity, this paper focuses on the simplest sampling model, namely, \emph{subgraph sampling}, where we randomly sample a subset of the vertices and observe their induced subgraph; in other words, only the edges between the sampled vertices are revealed. 
Results on estimating motif counts for the related neighborhood sampling model can be found in the companion paper \cite{KlusowskiWu2017-motif}.
 One of the earliest works that adopts the subgraph sampling model is by Frank \cite{Frank1978}, which is the basis for the theory developed in this paper. 
Drawing from previous work on estimating population total using vertex sampling \cite{Frank1977}, Frank obtained unbiased estimators of the number of connected components and performance guarantees (variance calculations) for graphs whose connected components are either all trees or all cliques.
%that are either forests or disjoint unions of cliques. 
Extensions to more general graphs are briefly discussed, although no unbiased estimators are proposed. 
This generality is desirable since it is more realistic to assume that the objects in each class (component) are in between being weakly and strongly connected to each other, corresponding to having the level of connectivity between a tree and clique. 
%As such, a flexible model is required to adapt to such cases. 
While the results of Frank are interesting, questions of their generality and optimality remain open and we therefore address these matters in the sequel.
Specifically, the main goals of this paper are as follows:
\begin{itemize}
	\item Characterize the sample complexity, i.e., the minimal sample size to achieve a given accuracy, as a function of graph parameters.
	%When is \emph{sublinear sample complexity} is achievable? That is, is it possible to sample a vanishing fraction of the vertices, such that even if it is impossible to reconstruct the entire graph, it is nevertheless possible to accurately estimate the desired graph property?
	\item  Devise computationally efficient estimators that provably achieve the optimal sample complexity bound.
	%, shown optimal by matching minimax lower bounds.
	%\item Develop a mathematical framework for addressing the above points in estimating other graph properties, i.e., subgraph counts.
\end{itemize}
Of particular interest is the \emph{sublinear regime}, where only a vanishing fraction of the vertices are sampled. In this case, it is impossible to reconstruct the entire graph, but it might still be possible to accurately estimate the desired graph property.

The problem of estimating the number of connected components in a large graph has also been studied in the computer science literature, where the goal is to 
design randomized algorithms with sublinear (in the size of the graph) time complexity. The celebrated work \cite{Chazelle2005} proposed a randomized algorithm 
to estimate the number of connected components in a general graph (motivated by computing the weight of the minimum spanning tree) within an additive error of $ \epsilon N$  for graphs with $ N $ vertices and average degree $ \de_{\text{avg}} $, with runtime $ O(\frac{\de_{\text{avg}}}{\epsilon^{2}} \log \frac{\de_{\text{avg}}}{\epsilon}) $. Their method relies on data obtained from a random sample of vertices and then performing a breadth first search on each vertex which ends according to a random stopping criterion. The algorithm requires knowledge of the average degree $ \de_{\text{avg}} $ and must therefore be known or estimated a priori. 
The runtime was further improved to $ O(\epsilon^{-2} \log \frac{1}{\epsilon}) $ by modifying the stopping criterion \cite{Berenbrink2014}.
%Subsequently, \cite{Berenbrink2014} noted that by modifying the stopping criterion, one could improve the runtime to $ \epsilon^{-2} \log \frac{1}{\epsilon} $. 
%Neither work shows the optimality of the algorithms among the class of simple graphs (i.e., graphs without multiple edges or self-loops). 
In these algorithms, the breadth first search may visit many of the edges and explore a larger fraction of the graph at each round. From an applied perspective, such traversal based procedures can be impractical or impossible to implement in many statistical applications due to limitations inherent in the experimental design and it is more realistic to treat the network data as a random sample from a parent graph.

Finally, let us compare, conceptually, the framework in the present paper with the work on \emph{model-based} network analysis, where networks are modeled as random graphs drawn from specific generative models, such as the stochastic block model \cite{holland1983stochastic}, graphons \cite{GLZ15}, or exponential random graph models \cite{HL81} (cf.~the recent survey \cite{Kolaczyk2017}), and performance analysis of statistical procedures for parameter estimation or clustering are carried out for these models.
In contrast, in network sampling we adopt a \emph{design-based} framework \cite{Handcock2010}, where the graph is assumed to be deterministic and the randomness comes from the sampling process.

\paragraph{Organization}
The paper is organized as follows. In \prettyref{sec:model}, we formally define the estimation problem, the subgraph sampling model, and describe what classes of graphs we will be focusing on. 
To motivate our attention on specific classes of graphs (chordal graphs with maximum degree constraints), we show that in the absence of such structural assumptions, sublinear sample complexity is impossible in the sense that at least a constant faction of the vertices need to be sampled. \prettyref{sec:main} introduces the definition of chordal graphs and states our main results in terms of the minimax risk and sample complexity.
 %the estimation problem is impossible unless one samples a constant faction of the vertices. 
In \prettyref{sec:algorithms}, after introducing the relevant combinatorial properties of chordal graphs, we define the estimator of the number of connect components and provide its statistical guarantees. We also propose a heuristic for constructing an estimator on non-chordal graphs.
In \prettyref{sec:lower}, we develop a general strategy for proving minimax lower bound for estimating graph properties and particularize it to obtain matching lower bounds for the estimator constructed in \prettyref{sec:algorithms}. 

Some of the technical proofs, additional results for the uniform sampling model and for forests and graphs with long cycles, and a numerical study of the proposed estimators on simulated data for various graphs are deferred till \prettyref{app:proofs}, \prettyref{app:results}, and \prettyref{app:experiments}, respectively.
%chordal and non-chordal graphs.

\paragraph{Notations}
%\label{sec:notations}

We use standard big-$O$ notations,
% as well as their counterparts in probability, 
%e.g., for any positive sequences $\{a_n\}$ and $\{b_n\}$, we write $a_n\gtrsim b_n$ or $b_n\lesssim a_n$ when $a_n\geq cb_n$ for some absolute constant $c$.    Finally, we write $a_n \asymp b_n$ when both $a_n\gtrsim b_n$ and $a_n\lesssim b_n$ hold.
e.g., for any positive sequences $\{a_n\}$ and $\{b_n\}$, $a_n=O(b_n)$ or $a_n \lesssim b_n$ if $a_n \leq C b_n$ for some absolute constant $C>0$,
$a_n=o(b_n)$ or $a_n \ll b_n$ or if $\lim a_n/b_n = 0$.
Furthermore, the subscript in $a_n=O_{r}(b_n)$ means $a_n \leq C_r b_n$ for some constant $C_r$ depending on the parameter $r$ only. For positive integer $ k $, let $ [k] = \{1, \dots, k \} $.
Let $\Bern(p)$ denote the Bernoulli distribution with mean $p$ 
and $\Binom(N,p)$ the binomial distribution with $N$ trials and success probability $p$.
%In order to extract non-asymptotic statements from asymptotic ones, we pay extra attention to $o(1)$ terms. Specifically, we write $o_{\delta}(1)$  as $\delta\to0$ to indicate convergence that is uniform in all other parameters.

Next we introduce some graph-theoretic notations that will be used throughout the paper. Let $ G = (V, E) $ be a simple undirected graph. Let $ \sfe = \sfe(G) = |E(G)| $ denote the number of edges, $ \sfv = \sfv(G) = |V(G)| $ denote the number of vertices, and $ \cc = \cc(G) $ be the number of connected components in $G$. The neighborhood of a vertex $ u $ is denoted by $ N_G(u) = \{ v \in V(G) : \{u, v\} \in E(G) \}  $.  
%Two vertices $ u $ and $ v $ are said to be adjacent to each other, denoted by $ u \sim v $, if $ \{u, v\} \in E(G) $.

Two graphs $ G $ and $ G' $ are isomorphic, denoted by $ G \simeq G' $, if there exists a bijection between the vertex sets of $ G $ and $ G' $ that preserves adjacency, i.e., if there exists a bijective function $ g: V(G) \to V(G') $ such that $ \{g(u), g(v)\} \in E(G') $ if and only if $ \{u, v\} \in E(G) $. The disjoint union of two graphs $G$ and $G'$, denoted $ G + G' $, is the graph whose vertex (resp.~edge) set is the disjoint union of the vertex (resp.~edge) sets of $G$ and of $G'$. For brevity, we denote by $ k G $ to the disjoint union of $ k $ copies of $ G $.
%For graphs $ G $ and $ G' $, we use the the notation $ G + G' $ to denote the disjoint union $ (V(H)\cup V(H'), E(H) \cup E(H')) $, where $H\simeq G$ and $H'\simeq G'$ and $V(H)$ and $V(H')$ are disjoint.

We use the notation $ K_{n} $, $ P_{n} $, and $ C_{n} $ to denote the complete graph, path graph, and cycle graph on $ n $ vertices, respectively. Let $ K_{n,n'} $ denote the complete bipartite graph with $ n n' $ edges and $ n + n' $ vertices. Let $ S_{n} $ denote the star graph $ K_{1, n} $ on $ n+1 $ vertices.

%We need two types of graph homomorphism numbers: 
%Denote by $\n(H,G)$ the number of injective homomorphisms from $H$ to $G$ and $ \s(H, G) $) the number of injective homomorphisms that also preserve non-adjacency.
We need two types of subgraph counts: 
Denote by $\s(H,G)$ (resp.~$ \n(H, G) $) the number of vertex (resp.~edge) induced subgraphs of $G$ that are isomorphic to $H$.\footnote{
The subgraph counts are directly related to the graph homomorphism numbers \cite[Sec 5.2]{Lovasz12}. Denote by $\mathsf{inj}(H, G) $ the number of injective homomorphisms from $H$ to $G$ and $\mathsf{ind}(H,G)$ the number of injective homomorphisms that also preserve non-adjacency. Then 
$\mathsf{ind}(H,G)=\s(H,G) \mathsf{aut}(H)$ and $\mathsf{inj}(H,G)=\n(H,G) \mathsf{aut}(H)$,
%\nbr{double check}
where $\mathsf{aut}(H)$ denotes the number of automorphisms (i.e. isomorphisms to itself) for $H$. }
For example, $\s(\PathThree,\; \Diamond) = 2$ and 
$\n(\PathThree,\; \Diamond) = 8$.
Let $\omega(G)$ denote the clique number, i.e., the size of the largest clique in $G$.

\section{Model} \label{sec:model}

\subsection{Subgraph sampling model}
	\label{sec:subgraph}
	
To fix notations, let $ G = (V, E) $ be a simple, undirected graph on $ N $ vertices. 
In the subgraph sampling model, we sample a set of vertices denoted by $ S \subset V $, and observe their induced subgraph, denoted by $G[S]=(S,E[S])$, where the edge set is defined as $E[S]=\{\{i,j\} \in E: i,j \in S\}$.
See \prettyref{fig:subgraph} for an illustration.
To simplify notations, we abbreviate the sampled graph $G[S]$ as $\tG$.
%If $ S $ represents a collection of vertices that are randomly generated according to a sampling mechanism, we denote $ G[S] $ by $ \tG $.
 %where $ n $ out of $ N $ vertices are chosen and we observe the subgraph induced by these sampled vertices. 

\begin{figure} [ht]
\centering
\begin{subfigure}[t]{0.45\textwidth}
  \centering
  \includegraphics[width=0.85\linewidth]{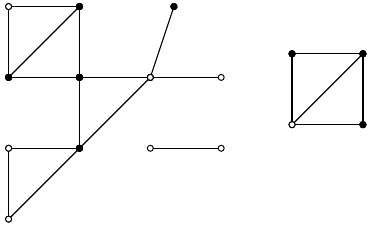}
  \caption{Parent graph $G$ with the set of sampled vertices $S$ shown in black.}
  \label{fig:fig1}
\end{subfigure}%
\hspace{0.5cm}
\begin{subfigure}[t]{0.45\textwidth}
  \centering
  \includegraphics[width=0.85\linewidth]{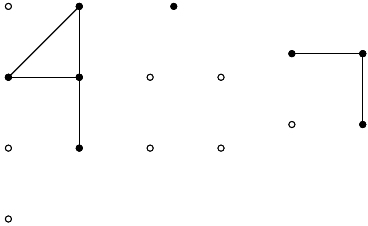}
  \caption{Subgraph induced by sampled vertices  $ \tG = G[S] $. Non-sampled vertices are shown as isolated vertices.}
  \label{fig:fig2}
\end{subfigure}
\caption{Subgraph sampling.}
\label{fig:subgraph}
\end{figure}

According to how the set $S$ of sampled vertices is generated, there are two variations of the subgraph sampling model \cite{Frank1978}:
\begin{itemize}
\item \emph{Uniform sampling}: Exactly $ n $ vertices are chosen uniformly at random without replacement from the vertex set $ V $. 
%In this case, the probability that a specific subset $ V $ of vertices is chosen is equal to $ \binom{N- |V|}{n-|V|} / \binom{N}{n} $. 
In this case, the probability of observing a subgraph isomorphic\footnote{Note that it is sufficient to describe the sampled graph up to isomorphism since the property $\cc$ we want to estimate is invariant under graph isomorphisms.} to $ H $ with $ \sfv(H) = n $ is equal to 
\begin{equation}
\pprob{ \tG \simeq H } = \frac{\s(H, G)}{\binom{N}{n}}  .
\label{eq:pmf-uniform}
\end{equation}

\item \emph{Bernoulli sampling}: Each vertex is sampled independently with probability $p$, where $ p $ is called the \emph{sampling ratio}.
Thus, the sample size $ |S|$ is distributed as $\Binom(N, p) $, and 
 %when $ S $ is generated from the outcome of $ N $ Bernoulli trials with success probability $ p $, corresponding to deciding whether a vertex $ \vertex $ belongs to the sample from the outcome of a coin $ b_{\vertex} \sim \Bern(p) $, where $ p $ is the \emph{sampling rate} or \emph{sampling probability}. Thus, $ S = \{ \vertex : b_{\vertex} = 1 \} $, $ |S| \sim \Binom(N, p) $. The probability that a specific subset $ V $ of vertices is chosen is equal to $ p^{|V|}q^{N-|V|} $. Therefore, 
the probability of observing a subgraph isomorphic to $ H $ is equal to 
\begin{equation}
\pprob{ \tG \simeq H } = \s(H, G)p^{\sfv(H)}(1-p)^{\sfv(G)-\sfv(H)} .
\label{eq:pmf-bern}
\end{equation}
\end{itemize}
The relation between these two models is analogous to that between sampling without replacements and sampling with replacements.
In the sublinear sampling regime where $n \ll N$, they are nearly equivalent. For technical simplicity, we focus on the Bernoulli sampling model and we refer to $n \triangleq pN$ as the \emph{effective sample size}. Extensions to the uniform sampling model will be discussed in \prettyref{sec:uniform} of \prettyref{app:results}. 

A number of previous work on subgraph sampling is closely related with the theory of graph limits \cite{Borgs2008}, which is motivated by the so-called property testing problems in graphs \cite{Goldreich17}. According to \cite[Definition 2.11]{Borgs2008}, a graph parameter $ f $ is ``testable" if for any $ \epsilon > 0 $, there exists a sample size $ n $ such that for any graph $ G $ with at least $ n $ vertices, there is an estimator $ \widehat{f} = \widehat{f}(\tG) $ such that $ \pprob{ | f(G) - \widehat{f} | > \epsilon} < \epsilon $. 
In other words, testable properties can be estimated with sample complexity that is \emph{independent} of the size of the graph.
Examples of testable properties include the edge density $\sfe(G)/\binom{\sfv(G)}{2}$ and the density of maximum cuts $ \frac{\mathsf{MaxCut}(G)}{\sfv(G)^2} $, where $ \mathsf{MaxCut}(G) $ is the size of the maximum edge cut-set in $ G $ \cite{Goldreich1998};
however, the number of connected components $\cc(G)$ or its normalized version $ \frac{\cc(G)}{\sfv(G)} $
 %and maximum degree density $ f(G) = \frac{\max_{v\inV(G)}|N_G(v)|}{\sfv(G)} $ 
are not testable.\footnote{To see this, recall from \cite[Theorem 6.1(b)]{Borgs2008} an equivalent characterization of $ f $ being testable is that for any $ \epsilon > 0 $, there exists a sample size $ n $ such that for any graph $ G $ with at least $ n $ vertices, $ |f(G) - \mathbb{E}f(\tG)| < \epsilon $. This is violated for star graphs $ G = S_{N} $ as $ N \rightarrow \infty $}
Instead, our focus is to understand the dependency of sample complexity of estimating $\cc(G)$ on the graph size $N$ as well as other graph parameters. 
It turns out for certain classes of graphs, the sample complexity grows \emph{sublinearly} in $N$, which is the most interesting regime.
%This shows that it is necessary to restrict our attention to specific classes of graphs by imposing conditions on certain graph parameters, which we elaborate next.

%This shows that it is necessary to restrict our attention to specific classes of graphs by imposing conditions on certain graph %parameters, which we elaborate next.

%Even if this dependence was allowed, the property $ \cc(G) $ would still not testable, since there are graphs (i.e., those with %high degree or long induced cycles) that make it impossible for such small error probability statements to be made.
%This shows that it is necessary to restrict our attention to specific classes of graphs by imposing conditions on certain graph %parameters, which we elaborate next.

%: the maximum degree (which also bounds the clique number, i.e.,  the size of the largest vertex induced clique) and the length of the maximal induced cycle.

%If $ f $ is a function defined on a collection of graphs $ \{ G \}_{G\in\calG} $ on $ N $ vertices and $ \widehat{f} $ is %an estimator of $ f $, we define the minimax mean square error as
%\begin{equation*}
%\inf_{\widehat{f}}\sup_{G\in\calG}\Expect_G|\widehat{f}-f(G)|^2.
%\end{equation*}

\subsection{Classes of graphs}
%\subsection{Which classes of graphs are hard }
%\subsection{Negative results: high-degree vertices and long induced cycles}

Before introducing the classes of graphs we consider in this paper, we note that, unless further structures are assumed about the parent graph, estimating many graph properties, including the number of connected components, has very high sample complexity that scales linearly with the size of the graph.
%is a trivial problem in the sense that almost all of the vertices must be sampled.
Indeed, there are two main obstacles in estimating the number of connected components in graphs, namely, \emph{high-degree vertices} and \emph{long induced cycles}. If either is allowed to be present, we will show that 
even if we sample a constant faction of the vertices,
any estimator of $\cc(G)$ has a worst-case additive error that is almost linear in the network size $N$.
%Specifically, for any sampling ratio $ p$ bounded away from $1$,
Specifically,
\begin{itemize}
\item 
For any sampling ratio  $ p$ bounded away from $1 $, as long as the maximum degree is allowed to scale as $\Omega(N)$, 
even if we restrict the parent graph to be acyclic, 
the worst-case estimation error for any estimator is $\Omega(N)$.
\item 
For any sampling ratio $ p$ bounded away from $ 1/2 $, as long as the length of the induced cycles is allowed to be $\Omega(\log N)$, 
even if we restrict the parent graph to have maximum degree 2, 
the worst-case estimation error for any estimator is $\Omega(\frac{N}{\log N})$.
\end{itemize}
The precise statements follow from the minimax lower bounds in \prettyref{thm:forestlower} and \prettyref{thm:cyclelower} of \prettyref{app:results}. 
Below we provide an intuitive explanation for each scenario.

For the first claim involving large degree, consider a pair of acyclic graphs $G$ and $G'$, where $ G $ is the star graph on $ N $ vertices and $ G' $ consisting of $ N $ isolated vertices. 
%Both graphs are graphs with maximum degree $ N-1 $. 
Note that as long as the center vertex in $ G $ is not sampled, the sampling distributions of $ G $ and $ G' $ are identical. This implies that the total variation between the sampled graph under $ G $ and $ G' $ is at most $ p $. Since the numbers of connected components in $ G $ and $ G' $ differ by $ N-1 $, this leads to a minimax lower bound for the estimation error of $\Omega(N)$ whenever $p$ is bounded away from one.

%\begin{figure} [ht]
%\centering
%\begin{subfigure}[t]{0.15\textwidth}
  %\centering
  %\includegraphics[width=1\linewidth]{Hstar.pdf}
  %\caption{$ G \simeq S_6. $}
  %\label{fig:fig1}
%\end{subfigure}%
%\hspace{5cm}
%\begin{subfigure}[t]{0.15\textwidth}
  %\centering
  %\includegraphics[width=1\linewidth]{Hpstar.pdf}
  %\caption{$ G \simeq 7K_1. $}
  %\label{fig:fig2}
%\end{subfigure}
%\caption{The two graphs are isomorphic if the center vertex (enclosed in the dotted region) is not sampled and all incident edges are removed.}
%\label{fig:star}
%\end{figure}

The effect of long induced cycles is subtler. 
The key observation is that a cycle and a path (or a cycle versus two cycles) locally look exactly the same. Indeed, let $ G $ (resp.~$G'$) consists of 
$ N/(2r) $ disjoint copies of the smaller graph $H$ (resp.~$H'$), where $H$ is a cycle of length $ 2r $ and 
$H'$ consists of two disjoint cycles of length $ r $.
 %(see \prettyref{fig:cycle}). 
%denote the graph consisting of $ N/(2\cyl) $ copies of a cycle of length $ 2\cyl $. Let $ G' $ denote the graph consisting of $ N/\cyl $ copies of cycles of length $ \cyl $ (see \prettyref{fig:cycle}). 
Both $ G $ and $ G' $ have maximum degree $ 2 $ and contain induced cycles of length at most $ 2r $. 
The local structure of $G$ and $G'$ is the same (e.g., each connected subgraph with at most $ r-1 $ vertices appears exactly $ N $ times in each graph) and the sampled versions of $H$ and $H'$ are identically distributed provided at most $ r-1 $ vertices are sampled.
 %(this statement will be made more precise in \prettyref{cor:counts} in \prettyref{sec:lower}). See \prettyref{fig:cycle} for an illustration of this property when $ \cyl = 3 $. 
Thus, we must sample at least $ r $ vertices (which occurs with probability at most $ e^{-r(1-2p)^2} $) for the distributions to be different. By a union bound, it can be shown that the total variation between the sampled graphs $ \tG $ and $ \tG' $ is $ O((N/r)e^{-r(1-2p)^2}) $. Thus, whenever the sampling ratio $p$ is bounded away from $ 1/2 $, choosing $r=\Theta(\log N)$ leads to a near-linear lower bound $\Omega(\frac{N}{\log N})$.
%, which is strictly less than one when $ \cyl \gtrsim (\log N) / (\log \frac{1}{p}) $. With $ \cyl $ a multiple of this quantity, the gap between the number of connected components in $ G $ and $ G' $ is $ N/(2\cyl) \asymp N\left(\frac{\log \frac{1}{p}}{\log N}\right) $. Again, the closeness of the distributions of $ \tG $ and $ \widetilde{G}' $ together with the separation in $ \cc $ leads to a minimax lower bound and the worst-case estimation error of $\Omega(\frac{N}{\log N})$.

%\begin{figure} [ht]
%\centering
%\begin{subfigure}[t]{0.4\textwidth}
  %\centering
  %\includegraphics[width=0.45\linewidth]{Hcyclemain.pdf}
  %\caption{$ H=C_6 $.}
  %\label{fig:fig1}
%\end{subfigure}%
%\hspace{0.5cm}
%\begin{subfigure}[t]{0.4\textwidth}
  %\centering
  %\includegraphics[width=0.8\linewidth]{Hpcyclemain.pdf}
  %\caption{$H' = C_3+C_3$.}
  %\label{fig:fig2}
%\end{subfigure}
%\caption{Examples of $G$ (resp.~$G'$) consisting multiple copies of $H$ (resp.~$H'$) with $r=3$. Both graphs have 6 vertices and 6 edges.}
%%Each connected subgraph with fewer than 3 vertices (viz., vertices and edges) appears exactly $ 6 $ times in each graph.}
%\label{fig:cycle}
%\end{figure}

%In view of the difficulties caused by high-degree vertices and long induced cycles, we must place some assumptions on these graph parameters. This motivates us to consider
The difficulties caused by high-degree vertices and long induced cycles motivate us to consider classes of graphs defined by two key parameters, namely, the maximum degree $d$ and the length of the longest induced cycles $c$. 
The case of $c=2$ corresponds to forests (acyclic graphs), which have been considered by Frank \cite{Frank1978}. 
%corresponding to maximal length of induced cycle equal to zero, 
The case of $c=3$ corresponds to  \emph{chordal graphs}, i.e., graphs without induced cycle of length four or above, which is the focus of this paper.
 %(see \prettyref{sec:chordal} for the formal definition).
It is well-known that various computation tasks that are intractable in the worst case, such as maximal clique and graph coloring, are easy for chordal graphs; it turns out that the chordality structure also aids in both the design and the analysis of computationally efficient estimators which provably attain the optimal sample complexity. 
%class of graphs provides enough structure that enables us to construct computationally efficient estimators which provably attain the optimal sample complexity. 
%their performance via variance bounds.

\section{Main results} \label{sec:main}

This section summarizes  our main results 
in terms of the minimax risk of estimating the number of connected components over various class of graphs. As mentioned before, for ease of exposition, we focus on the Bernoulli sampling model, where each vertex is sampled independently with probability $p$. 
%Even though our results are stated for this sampling model, 
Similar conclusions can be obtained for the uniform sampling model upon identifying $p=n/N$, as given in \prettyref{sec:uniform}.

%As mentioned before, we focus on the subgraph sampling model, that is, a subset of vertices is sampled at random and the %induced subgraph is observed. More specifically, we consider the Bernoulli sampling model, where each vertex is sampled %independently with probability $p$. 

When $ p $ grows from $ 0 $ to $ 1 $, an increasing fraction of the graph is observed and intuitively the estimation problem becomes easier. Indeed, all forthcoming minimax rates are inversely proportional to powers of $ p $. Of particular interest is whether accurate estimation in the sublinear sampling regime, i.e., $p=o(1)$.
%We are also interested in knowing how small $ p $ can be to ensure accurate estimation (i.e., so that the worst-case estimation error for any estimator is not $ \Omega(N) $).
 The forthcoming theory will give explicit conditions on $ p $ for this to hold true. 

As mentioned in the previous section, the main class of graphs we study is the so-called \emph{chordal graphs} (see \prettyref{fig:chordal-example} for an example):
%\nbr{-- this is not well written. Do you want to say something like:
%``In addition to constructing estimators that adapt to larger collections of graphs (for which forests and unions of cliques are special cases), we also provide theoretical analysis and optimality guarantees?''
%}
%There are a few extensions of his work that we provide. First, none of the analysis involves optimality and, second, we provide estimators that adapt to larger collections of graphs (for which forests and unions of cliques are special cases).
\begin{definition}
A graph $ G $ is chordal if it does not contain induced cycles of length four or above, i.e., $\s(C_k,G)=0$ for $k \geq 4$.
\end{definition}
\begin{figure} [ht]
\centering
\begin{subfigure}[t]{0.45\textwidth}
  \centering
  \includegraphics[width=0.8\linewidth]{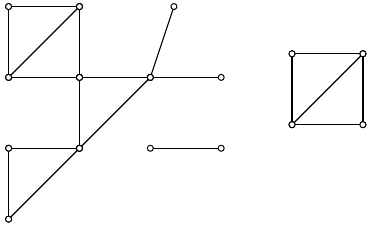}
  \caption{Chordal graph.}
  \label{fig:fig1}
\end{subfigure}%
\hspace{1cm}
\begin{subfigure}[t]{0.45\textwidth}
  \centering
  \includegraphics[width=0.8\linewidth]{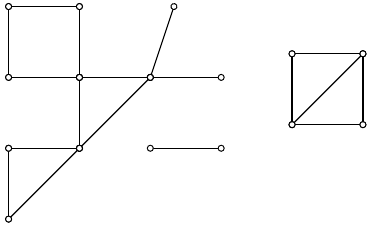}
  \caption{Non-chordal graph (containing an induced $C_4$).}
  \label{fig:fig2}
\end{subfigure}
\caption{Examples of chordal and non-chordal graphs both with three connected components.}
\label{fig:chordal-example}
\end{figure}

We emphasize that chordal graphs are allowed to have arbitrarily long cycles but no induced cycles longer than three.
The class of chordal graphs encompasses \emph{forests} and \emph{disjoint union of cliques} as special cases, the two models that were studied in Frank's original paper \cite{Frank1978}. In addition to constructing estimators that adapt to larger collections of graphs (for which forests and unions of cliques are special cases), we also provide theoretical analysis and optimality guarantees -- elements that were not considered in past work. 

Next, we characterize the rate of the minimax mean-squared error for estimating the number of connected components in a chordal graph, which turns out to depend on the number of vertices, the maximum degree, and the clique number.
The upper and lower bounds differ by at most a multiplicative factor depending only on the clique number.
To simplify the notation, henceforth we denote $q=1-p$.
 %Minimax upper bounds are also given for graph classes with additional structure (union of cliques and forests) and more general graph classes (maximum degree that grows with $ N $). In particular, the upper bounds show that the minimax rates have different behaviors depending on the size of the maximum degree of the parent graph.
\begin{theorem}[Chordal graphs]
\label{thm:chordalmain}
Let $ \calG(N,\de,\omega) $ denote the collection of all chordal graphs on $N$ vertices with maximum degree and clique number at most $ \de $ and $ \omega \geq 2$, respectively. Then
\begin{equation*}
\inf_{\cchat} \sup_{ G \in \calG(N,\de,\omega) } \Expect_G|\cchat - \cc(G)|^2 = \Theta_\omega\left(\left(\frac{N}{p^{\omega}} \vee \frac{N\de}{p^{\omega-1}} \right) \wedge N^2 \right),
\end{equation*}
where the lower bound holds provided that $p\leq p_0$ for some constant $p_0 < \frac{1}{2}$ that only depends on $\omega$.
Furthermore, if $ p \geq 1/2 $, then for any $\omega$,
\begin{equation} \label{eq:main2}
\inf_{\cchat} \sup_{ G \in \calG(N,\de,\omega) } \Expect_G|\cchat - \cc(G)|^2 \leq Nq(d+1).
\end{equation}
%and if $ \de $ is a constant, then
%\begin{equation*}
%\inf_{\cchat} \sup_{ G \in \calG(N,\de,\omega) } \Expect_G|\cchat - \cc(G)|^2 \asymp \frac{N}{p^{\omega}}%\wedge N^2.
%\end{equation*}
\end{theorem}

%The lower bound on the minimax risk in \prettyref{eq:main2} follows by specializing \prettyref{eq:main1} to $ \omega = 2 %$ and using the fact that $ \calG(N,\de,2) \subset \calG(N,\de,\omega) $ for all $ \omega \geq 2 $. \nbr{This sentence %should be moved into the proof as the proof of the second claim of the theorem. It breaks the flow. In fact,%\prettyref{eq:main2} cannot be true, e.g., $p=1$. I am worried about what the correct statement might be.}

Specializing \prettyref{thm:chordalmain} to $\omega=2$ yields the minimax rates for estimating the number of trees in forests for small sampling ratio $p$. 
The next theorem shows that the result holds verbatim even if $p$ is arbitrarily close to 1, and, consequently, shows minimax rate-optimality of the bound in \prettyref{eq:main2}. The lower bound component is proved in \prettyref{sec:forestlb} of \prettyref{app:results}.

\begin{theorem}[Forests] \label{thm:forests}
Let $ \mathcal{F}(N,\de) \triangleq \calG(N,\de,2) $ denote the collection of all forests on $N$ vertices with maximum degree at most $ \de $.
Then for all $ 0 \leq p \leq 1 $ and $1 \leq d \leq N$,
\begin{equation}
\inf_{\cchat} \sup_{ G \in \mathcal{F}(N,\de)  } \Expect_G|\cchat - \cc(G)|^2 \asymp \left(\frac{Nq}{p^2} \vee \frac{Nq\de}{p}\right) \wedge N^2.
\label{eq:main-forest}
\end{equation}
\end{theorem}

The upper bounds in the previous results are achieved by unbiased estimators. 
As \prettyref{eq:main2} shows, they work well even when the clique number $ \omega $ grow with $ N $, provided we sample more than half of the vertices; however, if the sample ratio $p$ is below $\frac{1}{2}$, especially in the sublinear regime of $p=o(1)$ that we are interested in, the variance is exponentially large. To deal with large $d $ and $\omega$, we must give up unbiasedness to achieve a good bias-variance tradeoff. Such biased estimators, obtained using the smoothing technique introduced in \cite{Wu2016-2}, lead to better performance as quantified in the following theorem. The proofs of these bounds are given in \prettyref{thm:chordal_smooth} and \prettyref{thm:clique_smooth}.

\begin{theorem}[Chordal graphs]
\label{thm:chordal-unbounded}
Let $ \calG(N, \de) $ denote the collection of all chordal graphs on $ N $ vertices with maximum degree at most $ \de $.
Then, for any $ p < 1/2 $,
\begin{equation*}
\inf_{\cchat} \sup_{ G \in \calG(N, \de) } \Expect_G|\cchat-\cc(G)|^2 \lesssim N^2\left(N/\de^2\right)^{-\frac{p}{2-3p}}.
\end{equation*}
\end{theorem}

Finally, for the special case of graphs consisting of disjoint union of cliques, as the following theorem shows,
there are enough structures so that we no longer need to impose any condition on the maximal degree. 
Similar to \prettyref{thm:chordal-unbounded}, the achievable scheme is a biased estimator, significantly improving the unbiased estimator in \cite{Goodman1949,Frank1978} which has exponentially large variance.
\begin{theorem}[Cliques]
\label{thm:clique}
Let $ \calC(N) $ denote the collection of all graphs on $ N $ vertices consisting of disjoint unions of cliques. Then, for any $ p < 1/2 $,
\begin{equation*}
\inf_{\cchat} \sup_{ G \in \calC(N)  } \Expect_G|\cchat - \cc(G)|^2 \leq N^2 (N/4)^{-\frac{p}{2-3p}}.
\end{equation*}
\end{theorem}

Alternatively, the above results are summarized in \prettyref{tab:sampling-complexity} in terms of the \emph{sample complexity}, i.e., the minimum sample size that allows an estimator $\cc(G)$ within an additive error of $\epsilon N$ with probability, say, at least 0.99, uniformly for all graphs in a given class. Here the sample size is understood as the average number of sampled vertices $n = p N$. We have the following characterization:

\begin{table}[h]
\centering
\bgroup
\def\arraystretch{2}
\begin{threeparttable}
\begin{tabular}{c| c c c}
\hline
\hline
Graph & Sample complexity $n$ & \\ [0.5ex]
\hline
\hline
Chordal & $ \Theta_{\omega}\pth{\max\sth{ N^{\frac{\omega-2}{\omega-1}} {\de}^{\frac{1}{\omega-1}}\epsilon^{-\frac{2}{\omega-1}}, \; N^{\frac{\omega-1}{\omega}} \epsilon^{-\frac{2}{\omega}} } }  $ & \\
\hline
Forest & $ \Theta\pth{\max\sth{ \frac{\de}{\epsilon^2}, \; \frac{\sqrt{N}}{\epsilon} }} $ & \\
\hline
%Chordal & $ \Theta\pth{\frac{N}{\log(\frac{N}{d\omega})} \log \frac{1}{\epsilon}   }  $ & \\
%\hline
Cliques & $ \Theta\pth{\frac{N}{\log N} \log \frac{1}{\epsilon}   }  $, \quad $\epsilon \geq N^{-1/2+\Omega(1)}$ \tnote{*} & \\
\hline
\end{tabular}
\vspace{4mm}
\begin{tablenotes}\footnotesize
\item[*] The lower bound part of this statement follows from \cite[Section 3]{WY16-distinct}, which shows the optimality of \prettyref{thm:clique}.
\end{tablenotes}
\end{threeparttable}
\caption{Sample complexity for various classes of graphs}
\label{tab:sampling-complexity}
\egroup
\end{table}
%\nbr{Remark 2,3,4 appear as orphans and rather scattered. They all talk about the same thing and all follow from the %results already summarized here. I suggest they be moved here to be part of the discussion on sample complexity and %sublinearity?}

A consequence of \prettyref{thm:forests} is that if the effective sample size $ n $ scales as $ O(\max(\sqrt{N}, \de)) $, for the class of forests $ \mathcal{F}(N, \de)  $ 
the worse-case estimation error for any estimator is $ \Omega(N) $, which is within a constant factor to the trivial error bound when no samples are available. Conversely, 
if $n \gg \max(\sqrt{N}, \de)$, which is sublinear in $N$ as long as the maximal degree satisfies $d=o(N)$, then it is possible to achieve a non-trivial estimation error of $o(N)$. More generally for chordal graphs, \prettyref{thm:chordalmain} implies that if $ n = O(\max(N^{\frac{\omega-1}{\omega}}, d^{\frac{1}{\omega-1}} N^{\frac{\omega-2}{\omega-1}})) $, the worse-case estimation error in $ \calG(N, \de, \omega)  $ for any estimator is at least $ \Omega_{\omega}(N) $, 

%A consequence of \prettyref{thm:forests} is that if the effective sample size $ p $ scales as $ O(\max(1/\sqrt{N}, \de/N)) $, the worse-case estimation error in $ \mathcal{F}(N, \de)  $ for any estimator is $ \Omega(N) $. More generally,\prettyref{thm:chordalmain} implies that if the sampling ratio is too small, that is, $ p = O(\max(N^{-\frac{1}{\omega}}, (\frac{d}{N})^{\frac{1}{\omega-1}})) $, the worse-case estimation error in $ \calG(N, \de, \omega)  $ for any estimator is at least $ \Omega_{\omega}(N) $, which is within a constant factor to the trivial error bound when no samples are available.

%To conclude this part, we mention that if 

\section{Algorithms and performance guarantees} \label{sec:algorithms}

In this section we propose estimators which provably achieve the upper bounds presented in \prettyref{sec:main} for the Bernoulli sampling model. In \prettyref{sec:chordal-property}, we highlight some combinatorial properties and characterizations of chordal graphs that underpin both the construction and the analysis of the estimators in \prettyref{sec:estimators}. 
The special case of disjoint unions of cliques is treated in \prettyref{sec:cliques}, where the estimator of Frank \cite{Frank1978} is recovered and further improved.
%In \prettyref{sec:cliques}, we specialize these estimators to graphs that are unions of cliques. 
Analogous results for the uniform sampling model are given in \prettyref{sec:uniform} of \prettyref{app:results}. Finally, in \prettyref{sec:non-chordal}, we discuss a heuristic to generalize the methodology to non-chordal graphs.

\subsection{Combinatorial properties of chordal graphs}\label{sec:chordal-property}

In this subsection we discuss the relevant combinatorial properties of chordal graphs which aid in the design and analysis of our estimators.
We start by introducing a notion of vertex elimination ordering.

\begin{definition} \label{def:peo}
A perfect elimination ordering (PEO) of a graph $ G $ on $ N $ vertices is a vertex labelling $ \{ v_1, v_2, \dots, v_N \} $ such that, for each $ j $, $ N_G(v_j) \cap \{v_1,...,v_{j-1} \} $ is a clique.
\end{definition}

\begin{figure} [ht]
  \centering
  \includegraphics[width=0.4\linewidth]{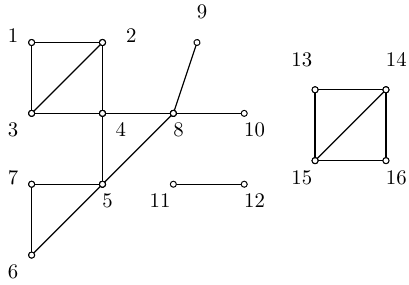}
  \caption{A chordal graph $ G $ with PEO labelled. 
	%$ (1, 2, 3, 4, 5, 6, 7, 8, 9, 10, 11, 12, 13, 14, 15, 16) $. 
	In this example, $ \cc(G) =  3 =  16 - 19 + 6 = \s(K_1, G) - \s(K_2, G) + \s(K_3, G) $.}
  \label{fig:peo}
\end{figure}

In other words, if one eliminates the vertices sequentially according to a PEO starting from the last vertex, at each step, the neighborhood of the vertex to be eliminated forms a clique; see \prettyref{fig:peo} for an example. A classical result of Dirac asserts that the existence of a PEO is in fact the defining property of chordal graphs (cf.~e.g.,~\cite[Theorem 5.3.17]{west-book}).

\begin{theorem} \label{thm:peo}
A graph is chordal if and only if it admits a PEO.
\end{theorem}

In general a PEO of a chordal graph is not unique; however, it turns out that the size of each neighborhood in the vertex elimination process is unique up to permutation, a fact that we will exploit later on.
%i.e., the collection of numbers defined by $ |N_G(\vertex_{j}) \cap \{\vertex_{1},\dots,v_{j-1}\}| $, is unique.
 The next lemma makes this claim precise. For brevity, we defer its proof to \prettyref{app:proofs}.

\begin{lemma} \label{lmm:unique}
Let $ \{v_1, \dots, v_N\} $ and $ \{v'_1, \dots, v'_N\} $ be two PEOs of a chordal graph $ G $. Let $ \sfc_j $ and $ \sfc'_j $ denote the cardinalities of $ N_G(\vertex_{j})\cap \{\vertex_{1},\dots,v_{j-1}\} $ and $ N_G(\vertex'_{j}) \cap \{\vertex'_{1},\dots,v'_{j-1}\} $, respectively. 
%Then there is a bijection between the set of numbers $ \{ \sfc_j : j \in [N] \} $ and $ \{ \sfc'_j : j \in [N] \} $.
Then there is a bijection $ \sigma : [N] \to [N] $ such that $ \sfc_{\sigma(j)} = \sfc'_j $ for all $ j $.
\end{lemma}

Recall that $ \s(K_i, G) $ denotes the number of cliques of size $ i $ in $ G $. For any chordal graph $ G $, it turns out that the number of components can be expressed as an alternating sum of clique counts (cf.~e.g., \cite[Exercise 5.3.22, p.~231]{west-book}); see \prettyref{fig:peo} for an example.
Instead of the topological proof involving properties of the clique simplex of chordal graphs \cite{McMahon2003, dohmen2013lower}, 
%\cite{Naiman1997, Edelman2000, McMahon2003, dohmen2013lower}, 
in the next lemma we provide a combinatorial proof together with a sandwich bound. The main purpose of this exposition is to explain how to enumerate cliques in chordal graphs using vertex elimination, which plays a key role in analyzing the statistical estimator developed in the next subsection.

%The fact that one can infer global properties of a graph (such as number of connected components or communities) from local %features has also be used in generative graph models (e.g., wheels \cite{Bickel2011}, triangles \cite{GaoLafferty2017}, and %cycles \cite{MosselNeemanSly2015}).
\begin{lemma}
\label{lmm:chordal}
For any chordal graph $ G $,
\begin{equation}
\label{eq:clique-id}
\cc(G) = \sum_{i \geq 1}(-1)^{i+1}\s(K_i, G).
\end{equation}
Furthermore, for any $r\geq 1$,
\begin{equation}
\label{eq:clique-ineq}
\sum_{i=1}^{2r}(-1)^{i+1}\s(K_i, G) \leq \cc(G) \leq \sum_{i=1}^{2r-1}(-1)^{i+1}\s(K_i, G).
\end{equation}
\end{lemma}

%Before proceeding to the construction of the estimator and its analysis, it is instructive to give a proof of \prettyref{lmm:chordal}, which illustrates how to count cliques in chordal graphs using vertex elimination. The same technique will be applied in bounding the variance of the estimator.
\begin{proof}%[Proof of \prettyref{lmm:chordal}]
Since $ G $ is chordal, by \prettyref{thm:peo}, it has a PEO $ \{v_1, \dots, v_N\} $. 
Define 
\begin{equation}
 C_j \triangleq N_G(\vertex_{j}) \cap \{ \vertex_{1},\ldots, \vertex_{j-1} \}, \quad \sfc_j \triangleq |C_j|.
\label{eq:Cj}
\end{equation}
Since the neighbors of $ \vertex_{j} $ among $ \vertex_{1},\dots,\vertex_{j-1} $ form a clique, we obtain $ \binom{\sfc_j}{i-1} $ new cliques of size $ i $ when we adjoin the vertex $ \vertex_{j} $ to the subgraph induced by $ \vertex_{1},\dots,\vertex_{j-1} $. Thus,
\begin{equation} \label{eq:peo-complete-count}
\s(K_i, G) = \sum_{j=1}^N \binom{\sfc_j}{i-1}.
\end{equation}
Moreover, note that $ \cc(G) = \sum_{j=1}^N \Indc\{\sfc_j = 0\} $.
Hence, it follows that
\begin{align*}
\sum_{i=1}^{2r-1}(-1)^{i+1}\s(K_i, G) & =
\sum_{i=1}^{2r-1}(-1)^{i+1}\sum_{j=1}^N \binom{\sfc_j}{i-1} = \sum_{j=1}^N\sum_{i=1}^{2r-1}(-1)^{i+1}\binom{\sfc_j}{i-1} \\
& = \sum_{j=1}^N\sum_{i=0}^{2(r-1)}(-1)^{i}\binom{\sfc_j}{i} = \sum_{j=1}^N\left(\binom{\sfc_j-1}{2(r-1)}\Indc\{ \sfc_j \neq 0 \} + \Indc\{ \sfc_j = 0 \}\right) \\
& \geq \sum_{j=1}^N \Indc\{\sfc_j = 0\} = \cc(G),
\end{align*}
and
\begin{align*}
\sum_{i=1}^{2r}(-1)^{i+1}\s(K_i, G) & =
\sum_{i=1}^{2r}(-1)^{i+1}\sum_{j=1}^N \binom{\sfc_j}{i-1} = \sum_{j=1}^N\sum_{i=1}^{2r}(-1)^{i+1}\binom{\sfc_j}{i-1} \\
& = \sum_{j=1}^N\sum_{i=0}^{2r-1}(-1)^{i}\binom{\sfc_j}{i} = \sum_{j=1}^N\left(-\binom{\sfc_j-1}{2r-1}\Indc\{ \sfc_j \neq 0 \} + \Indc\{ \sfc_j = 0 \}\right) \\
& \leq \sum_{j=1}^N \Indc\{\sfc_j = 0\} = \cc(G). \qedhere
\end{align*}
\end{proof}

\subsection{Estimators for chordal graphs} \label{sec:estimators}
\subsubsection{Bounded clique number: unbiased estimators}
In this subsection, we consider unbiased estimation of the number of connected components in chordal graphs. As we will see, unbiased estimators turn out to be minimax rate-optimal for chordal graphs with bounded clique size. The subgraph count identity \prettyref{eq:clique-id} suggests the following unbiased estimator
\begin{equation} \label{eq:chordal-estimator}
\cchat = -\sum_{i\geq 1} \pth{-\frac{1}{p}}^{i}\s(K_i, \tG).
\end{equation}
Indeed, since the probability of observing any given clique of size $i$ is $p^i$, 
\prettyref{eq:chordal-estimator} is clearly unbiased in the same spirit of the Horvitz-Thompson estimator \cite{HT52}. In the case where the parent graph $G$ is a forest, \prettyref{eq:chordal-estimator} reduces to the estimator $\cchat = \sfv(\tG)/p - \sfe(\tG)/p^2 $, as proposed by Frank \cite{Frank1978}.

A few comments about the estimator \prettyref{eq:chordal-estimator} are in order. First, it is completely adaptive to the parameters $ \omega $,  $ \de $ and $ N $, since the sum in \prettyref{eq:chordal-estimator} terminates at the clique number of the subsampled graph. Second, it can be evaluated in time that is linear in $\sfv(\widetilde G)+\sfe(\widetilde G)$. Indeed, the next lemma gives a simple formula for computing \prettyref{eq:chordal-estimator} using the PEO.
Since a PEO of a chordal graph $ G $ can be found in $ O(\sfv(G) + \sfe(G)) $ time 
\cite{Rose1976, Tarjan1984} and any induced subgraph of a chordal graph remains chordal, 
the estimator \prettyref{eq:chordal-estimator} can be evaluated in linear time.\footnote{The algorithm in \cite{Tarjan1984} is implemented in $\texttt{R}$ using the $\texttt{max\char`_cardinality()}$ function in the graph package $\texttt{igraph}$.}
%For the next set of results, we write $ q = 1-p $ for brevity.
Recall that $q=1-p$.
\begin{lemma}
\label{lmm:ccpeo}	
%There are positive integers $ \widetilde{\sfc}_1, \widetilde{\sfc}_2, \dots $ that can be calculated in linear time from $ \tG $ such that
Let $ \{\widetilde{v}_1, \dots, \widetilde{v}_m \} $, $ m = |S| $, be a PEO of $ \tG $. Then
	\begin{equation} \label{eq:chordal-estimator-peo}
\cchat = \frac{1}{p}\sum_{j = 1}^m \pth{-\frac{q}{p}}^{\widetilde{\sfc}_j},
\end{equation}
%\nbr{there seems to be some $b$ missing, or the summation is over $j=1,...,m$ with $m=|S|$.}
where $ \widetilde{\sfc}_j \triangleq |N_{\tG}(\widetilde{v}_j ) \cap \{ \widetilde{v}_1, \dots, \widetilde{v}_{j-1} \}| $
can be calculated from $ \tG $ in linear time.
\end{lemma}
\begin{proof}

%Because the subsampled graph $ \tG $ is also chordal, it too has a PEO. Thus, to calculate $ \s(K_i, \tG) $, we can first find a PEO $ \{v_1, \dots, v_m\} $ of $ \tG $ in linear time and then use the formula $ \s(K_i, \tG) = \sum_{j=1}^N \binom{\widetilde{\sfc}_j}{i-1} $.
%Thus, \prettyref{eq:chordal-estimator} can also be written as
Because the subsampled graph $ \tG $ is also chordal, by \prettyref{eq:peo-complete-count}, we have $ \s(K_i, \tG) = \sum_{j=1}^m \binom{\widetilde{\sfc}_j}{i-1} $.
Thus, \prettyref{eq:chordal-estimator} can also be written as
\begin{align*}
\cchat & = -\sum_{i=1}^m\left(-\frac{1}{p}\right)^{i}\s(K_i, \tG) = -\sum_{i=1}^m\left(-\frac{1}{p}\right)^{i}\sum_{j=1}^{m}\binom{\widetilde{\sfc}_j}{i-1} \\
& = -\sum_{j=1}^{m}\sum_{i=1}^m\left(-\frac{1}{p}\right)^{i}\binom{\widetilde{\sfc}_j}{i-1} = \frac{1}{p}\sum_{j=1}^{m}\sum_{i=0}^{m-1}\left(-\frac{1}{p}\right)^{i}\binom{\widetilde{\sfc}_j}{i} \\
& = \frac{1}{p}\sum_{j= 1}^m \left(-\frac{q}{p}\right)^{\widetilde{\sfc}_j}. \qedhere
\end{align*}
%\nbr{This is the same as \prettyref{eq:chordal-estimator-peo1}, so one of them is redundant. 
%I suggest here we add a lemma explaining that the estimator can be computed in linear time according to
%\prettyref{eq:chordal-estimator-peo}, and give a quick proof here, which is essentially the equations around \prettyref{eq:chordal-estimator-peo1}. \nb{I added the lemma.}
%Then, in the beginning of the proof of \prettyref{thm:chordal}, at which point you have already introduced $f$, you can say that 
%$\cchat$ is equivalently written as \prettyref{eq:chordal-estimator-peo1} with the $b_v$s.
%I have taken a stab in doing so. Please finish.
%}
\end{proof}

In addition to the aforementioned computational advantages of using \prettyref{eq:chordal-estimator-peo} over \prettyref{eq:chordal-estimator}, let us also describe why \prettyref{eq:chordal-estimator-peo} is more numerically stable. Both estimators are equal to an alternating sum of the form $\sum_i a_i(-1/p)^{b_i}$. In \prettyref{eq:chordal-estimator-peo}, $a_i = q^{b_i}/p$, whereas $a_i = -\ind(K_{b_i}, \tG)$ in \prettyref{eq:chordal-estimator}, which can be as large as $O(N2^{\omega})$ in magnitude. Thus, when $ \tG $ is sufficiently dense, computation of \prettyref{eq:chordal-estimator} involves adding and subtracting extremely large numbers -- making it prone to integer overflow and suffer from loss of numerical precision. For example, double-precision floating-point arithmetic (e.g., used in $\texttt{R}$) gives from $15$ to $17$ significant decimal digits precision. In our experience, this tends to be insufficient for most mid-sized, real-world networks (see \prettyref{app:experiments}) and the estimator \prettyref{eq:chordal-estimator} outputs wildly imprecise numbers.

Using elementary enumerative combinatorics, in particular, the vertex elimination structure of chordal graphs, the next theorem provides a performance guarantee for the estimator \prettyref{eq:chordal-estimator} in terms of a variance bound and a high-probability bound, which, in particular, settles the upper bound of the minimax mean squared error in \prettyref{thm:chordalmain} and \prettyref{thm:forests}.

%The following result bounds the performance of \prettyref{eq:chordal-estimator} through its mean squared error and, consequently, proves the upper bound of the minimax mean squared error in \prettyref{thm:chordalmain}.

\begin{theorem} \label{thm:chordal}
Let $ G $ be a chordal graph on $ N $ vertices with maximum degree and clique number at most $ \de $ and $ \omega \geq 2 $, respectively.
Suppose $ \tG $ is generated by the $ \Bern(p) $ sampling model. Then
$\cchat$ defined in \prettyref{eq:chordal-estimator} is an unbiased estimator of $ \cc(G) $. Furthermore,
\begin{equation}
\label{eq:chordal-var}
\Var[\cchat] \leq  N\left(\frac{q}{p} + \de \right) \left(\left( \frac{q}{p} \right)^{\omega-1} \vee \frac{q}{p} \right)\leq \frac{N}{p^{\omega}} + \frac{N\de}{p^{\omega-1}},
\end{equation}
and for all $ t \geq 0 $,
\begin{equation} \label{eq:chordal-concentration}
\prob{ |\cchat -  \cc(G) | \geq t } \leq 2\exp \left\{-\frac{8p^{\omega}t^2}{ 25(\de\omega+1)(N + t/3 ) } \right\}.
\end{equation}
\end{theorem}

To prove \prettyref{thm:chordal} we start by presenting a useful lemma. Note that \prettyref{lmm:ccpeo} states that $ \cchat $ is a linear combination of $ (-q/p)^{\widetilde{\sfc}_j} $; here $\widetilde{\sfc}_j$ is computed using a PEO of the sampled graph, which itself is random. 
The next result allows us rewrite the same estimator as a linear combination of $ (-q/p)^{\widehat{\sfc}_j} $, where $ \widehat{\sfc}_j $ depends on the PEO of the parent graph (which is deterministic). Note that this is only used in the course of analysis since the population level PEO is not observed.
This representation is extremely useful in analyzing the performance of $ \cchat $ and its biased variant in \prettyref{sec:smooth}. More generally, we have the following result, which we prove in \prettyref{app:proofs}.

\begin{lemma} \label{lmm:permutation}
Let $ \{ v_1, \dots, v_N \} $ be a PEO of $ G $ and let $ \{ \widetilde{v}_1, \dots, \widetilde{v}_m \} $, $ m = |S| $, be a PEO of $ \tG $. Furthermore, let $ \widehat{\sfc}_j = |N_{\tG}( v_j ) \cap \{ v_1, \dots, v_{j-1} \}| $ and $ \widetilde{\sfc}_j = |N_{\tG}(\widetilde{v}_j ) \cap \{ \widetilde{v}_1, \dots, \widetilde{v}_{j-1} \}| $. Let $ \widehat{\sfg} = \widehat{\sfg}(\tG) $ be a linear estimator of the form
\begin{equation} \label{eq:linear-est}
\widehat{\sfg} = \sum_{j=1}^m g(\widetilde{\sfc}_j).
\end{equation}
Then
\begin{equation*}
\widehat{\sfg} = \sum_{j=1}^Nb_j g(\widehat{\sfc}_j),
\end{equation*}
where $b_j \triangleq \indc{v_j \in S}$.
\end{lemma}

We also need a couple of ancillary results whose proofs are also given in \prettyref{app:proofs}:
\begin{lemma}[Orthogonality] \label{lmm:orthogonal}
Let\footnote{In fact, the function $f(N_{S})=(-\frac{q}{p})^{N_{S}}$ is the (unnormalized) orthogonal basis for the binomial measure that is used in the analysis of Boolean functions \cite[Definition 8.40]{odonnell}.}
\begin{equation}
f(k) = \pth{-\frac{q}{p}}^k, \quad k \geq 0.
\label{eq:f}
\end{equation}
Let $ \{b_v: v\in V\}$ be independent $ \Bern(p) $ random variables. For any $S\subset V$, define $ N_{S} = \sum_{\vertex\in S}b_{\vertex} $. 
Then
\begin{equation*}
\Expect[f(N_{S})f(N_{T})] = \Indc\{S = T\}(q/p)^{|{S}|}.
\end{equation*}
In particular, $\Expect[f(N_{S})] = 0$ for any $S \neq \emptyset$.
%Let $ S $ be a subset of a collection $ \mathcal{S} $ of non-empty subsets of $V(G)$. Define $ N_{S} = \sum_{\vertex\in S}b_{\vertex} $, where $ \{b_{\vertex}\}_{\vertex\in V(G)} $ is a sequence of independent $ \Bern(p) $ random variables. Then, for any $ S,  T \in \mathcal{S} $,
%\begin{equation*}
%\Expect[f(N_{S})f(N_{ T})] = \Indc\{S =  T\}(q/p)^{|{S}|},
%\end{equation*}
%and any sequence of numbers $ \{\alpha_{S}\}_{S\in\mathcal{S}} $,
%\begin{equation*}
%\Var\left[\sum_{S\in\mathcal{S}}\alpha_{S}f(N_{S})\right] = \sum_{S\in\mathcal{S}}\alpha^2_{S}\pth{\frac{q}{p}}^{|{S}|}.
%\end{equation*}
\end{lemma}

%\begin{remark}
%\label{rmk:boolean}	
%In fact, the function $f(N_{S})=(-\frac{q}{p})^{N_{S}}$ is the (unnormalized) orthogonal basis for the binomial measure that is used in the analysis of Boolean functions \cite[Definition 8.40]{odonnell}.
%\end{remark}

\begin{lemma} \label{lmm:independent}
Let $ \{ v_1,\ldots,v_N \} $ be a PEO of a chordal graph $ G $ on $ N $ vertices with maximum degree and clique number at most $ \de $ and $ \omega $, respectively. Let $C_j \triangleq N_G(\vertex_{j}) \cap \{ \vertex_{1},\ldots, \vertex_{j-1} \}$. Then\footnote{The bound in \prettyref{eq:Cjnumber} is almost optimal, since the left-hand side is equal to $ N(d-2) $ when $G$ consists of $ N/(d+1) $ copies of stars $ S_d $.}
%\footnote{The cardinality of the set in \prettyref{eq:Cjnumber} is equal to $ N(d-1)\left(\frac{d-2}{d+1}\right) $ when $G$ consists of $ N/(d+1) $ stars $ S_d $.}
\begin{equation}
 | \{ (i,j): i \neq j,\; C_j = C_i \neq \emptyset \} | \leq N(\de-1).
\label{eq:Cjnumber}
\end{equation}
Furthermore, let 
\begin{equation}
%A_j = \{ \vertex_{j} \} \cup (N_G(\vertex_{j})\cap \{\vertex_{1},\dots, \vertex_{j-1}\} ).
A_j = \{ \vertex_j \} \cup C_j.
%\cup \pth{ N_G(\vertex_j)\cap \{\vertex_1,\dots, \vertex_{j-1}\} }.
\label{eq:Aj}
\end{equation}
Then for each $j \in [N]$,
\begin{equation} 
\label{eq:maxdegA}
|\{ i \in [N]: i \neq j, \; A_i \cap A_j \neq \emptyset \}| \leq \de \omega.
\end{equation}
%Consequently, if $ \Gamma $ is the graph with vertex set $ [N] $ and edge set $ \{ \{i, j\} : A_i \cap A_j \neq \emptyset \} $, then $ \Gamma $ has maximum degree at most $ 2N\de\omega $.
\end{lemma}

\begin{proof}[Proof of \prettyref{thm:chordal}]

For a chordal graph $ G $ on $ N $ vertices, let $ \{v_1, \dots, v_N\} $ be a PEO of $ G $. Recall from \prettyref{eq:Cj} that $ C_j $ denote the set of neighbors of $ \vertex_{j} $ among $ \vertex_{1}, \dots, \vertex_{j-1} $ and $ \sfc_j $ denotes its cardinality. That is, 
\[
\sfc_j = |N_G(\vertex_{j}) \cap \{ \vertex_{1},\ldots, \vertex_{j-1} \}| = \sum_{k=1}^{j-1}\Indc\{ \vertex_{k} \sim \vertex_{j}  \} .
\]
As in \prettyref{lmm:permutation}, let $ \widehat{\sfc}_j $ denote the sample version, i.e.,
\begin{equation*}
\widehat{\sfc}_j 
\triangleq |N_{\tG}(\vertex_{j}) \cap \{ \vertex_{1},\ldots, \vertex_{j-1} \}|  
= b_j\sum_{k=1}^{j-1} b_{k} \Indc\{ \vertex_{k} \sim \vertex_{j} \},
\end{equation*}
%\nbr{note that the second = is incorrect without the $b_{j}$ in front. It will be OK if you drop the first equality and just use the second term as the definition. 
%However, then it will be inconsistent with \prettyref{eq:cjhat}. 
%There, since you need to invoke \prettyref{lmm:unique}, it is crucial you use the first term as definition.
%So, all in all, it seems wiser to add the $b_{j}$ back and in the later calculation you deal with it.
%}
where $ b_k \triangleq \Indc\{ \vertex_k \in S \} \iiddistr \Bern(p)$. 
%\nbr{Say something like $\{v_1, \dots, v_N\}$ is also a PEO for the subgraph $\tG$. Then... }
By \prettyref{lmm:ccpeo} and \prettyref{lmm:permutation}, $ \cchat $ can be written as
\begin{equation} \label{eq:chordal-estimator-peo1}
\cchat = \frac{1}{p}\sum_{j=1}^{m}f(\widetilde{\sfc}_j) = \frac{1}{p}\sum_{j=1}^{N} b_{j}f(\widehat{\sfc}_j),
\end{equation}
where the function $f$ is defined in \prettyref{eq:f}. 

To show the variance bound \prettyref{eq:chordal-var}, we note that
\begin{equation} \label{eq:var0}
\Var[\cchat] = \frac{1}{p^2}\sum_{j=1}^{N}\Var[b_{j}f(\widehat{\sfc}_j)] + \frac{1}{p^2}\sum_{j \neq i}\Cov[b_{j}f(\widehat{\sfc}_j), b_{i}f(\widehat{c}_i)].
\end{equation}
Note that $ \widehat{\sfc}_j \mid \{b_{j}=1\} \sim \Binom(\sfc_j, p) $.
Using \prettyref{lmm:orthogonal}, it is straightforward to verify that
\begin{equation}
\Var[b_{j}f(\widehat{\sfc}_j)] =
\begin{cases}
    p\pth{\frac{q}{p}}^{\sfc_j} & \text{if } \sfc_j > 0 \\
    pq & \text{if } \sfc_j = 0
  \end{cases}.
	\label{eq:var1}
\end{equation}
Since $ \sfc_j \leq \omega-1 $, it follows that 
\begin{equation}
\Var[b_{j}f(\widehat{\sfc}_j)] \leq p \qth{\pth{\frac{q}{p}}^{\omega-1} \vee \frac{q}{p}}.
\label{eq:var11}
\end{equation}
The covariance terms are less obvious to bound; but thanks to the orthogonality property in \prettyref{lmm:orthogonal}, many of them are zero or negative. 
Let $N_C \triangleq \sum b_j \indc{v_j \in C}$.
For any $j$, since $v_j \not\in C_j$ by definition, 
applying \prettyref{lmm:orthogonal} yields
\begin{equation}
\Expect[b_j f(\widehat{\sfc}_j)] = p \Expect[f(N_{C_j})] = p \indc{C_j =\emptyset}.
\label{eq:dec}
\end{equation}
Without loss of generality, assume $ j < i $. By the definition of $ C_j $, we have $ v_i \notin C_j $.
Next, we consider two cases separately: 
%\begin{enumerate}[(I)]
%\item $ v_j \not \sim v_i $, i.e., $ v_j \notin C_i $; \label{eq:cond1}
%\item $ v_j \sim C_i $, i.e., $ v_j \in C_i $; \label{eq:cond2}
%\end{enumerate}
\paragraph{Case I: $v_j \notin C_i $.}
If either $ C_j $ or $ C_i $ is nonempty, \prettyref{lmm:orthogonal} yields
\begin{equation*}
\Cov[b_{j}f(\widehat{\sfc}_j), b_{i}f(\widehat{\sfc}_i)]
\overset{\prettyref{eq:dec}}{=}\Expect[b_{i}b_{j}f(\widehat{\sfc}_j) f(\widehat{\sfc}_i)]
 = p^2\Expect[f(N_{C_j}) f(N_{C_i})] = p^2\Indc\{C_j = C_i\}\pth{\frac{q}{p}}^{\sfc_j}.
\end{equation*}
If $ C_j = C_i = \emptyset $, then $ \Cov[b_{j}f(\widehat{\sfc}_j), b_{i}f(\widehat{\sfc}_i)] = \Cov[b_{j}, b_{i}] = 0 $.

\paragraph{Case II: $v_j \in C_i $.}
Then $\Expect[b_{i}f(\widehat{\sfc}_i)]=0$ by \prettyref{eq:dec}. Using \prettyref{lmm:orthogonal} again, we have
\begin{align*}
\Cov[b_{j}f(\widehat{\sfc}_j), b_{i}f(\widehat{\sfc}_i)] 
& = p \expect{b_j \pth{-\frac{q}{p}}^{b_j}} \Expect[f(N_{C_j}) f(N_{C_i \setminus \{ v_j\} })] \\ 
& = -pq \Expect[f(N_{C_j}) f(N_{C_i \setminus \{ v_j\} })] \\ 
& = -pq\Indc\{C_j = C_i \setminus \{ v_j\} \}\pth{\frac{q}{p}}^{\sfc_j}.
\end{align*}
To summarize, we have shown that
\begin{equation*}
\Cov[b_{j}f(\widehat{\sfc}_j), b_{i}f(\widehat{\sfc}_i)] =
\begin{cases}
    p^2\pth{\frac{q}{p}}^{\sfc_j} & \text{if } C_j = C_i \neq \emptyset \\
    -pq\pth{\frac{q}{p}}^{\sfc_j} & \text{if } C_j  = C_i\setminus\{\vertex_{j}\} \text{ and } \vertex_{j} \in C_i \\
    0 & \text{otherwise}
  \end{cases}.
\end{equation*}
Thus, 
\begin{align}
\sum_{j\neq i}\Cov[b_{j}f(\widehat{\sfc}_j), b_{i}f(\widehat{c}_i)] \leq \sum_{j \neq i: \; C_j = C_i \neq \emptyset}p^2\pth{\frac{q}{p}}^{\sfc_j} \overset{\prettyref{eq:Cjnumber}}{\leq} N(\de-1)p^2 \qth{\pth{\frac{q}{p}}^{\omega-1} \vee \frac{q}{p}}. \label{eq:var2}
\end{align}
Finally, combining \prettyref{eq:var0}, \prettyref{eq:var11} and \prettyref{eq:var2} yields the desired \prettyref{eq:chordal-var}.
%\begin{align*}
%\Var[\cchat]
%& \leq \frac{N}{p}\pth{\frac{q}{p}}^{\omega-1} + 2Nd\pth{\frac{q}{p}}^{\omega-1}
%\end{align*}
%and hence the desired \prettyref{eq:chordal-var}.

The high-probability bound \prettyref{eq:chordal-concentration} for $ \cchat $ follows from the concentration inequality in \prettyref{lmm:janson} in \prettyref{app:proofs}. 
To apply this result, note that $ \cchat$ is a sum of dependent random variables
\begin{equation}
\cchat = \sum_{j \in [N]} Y_j,
\label{eq:ccY}
\end{equation}
 where $ Y_j = \frac{1}{p}b_{j}f(\widehat{\sfc}_j) $ satisfies $\Expect[Y_j]=0$ for $ \sfc_j > 0 $ and $|Y_j| \leq b \triangleq (\frac{1}{p})^{\omega} $ almost surely. 
Also, $ S \triangleq \sum_{j \in [N]} \Var[Y_j] \leq N(\frac{1}{p})^{\omega} $ by \prettyref{eq:var1}. 
%\nbr{shouldn't it be $\frac{N}{p^{\omega-2}}$?}
To control the dependency between $ \{ Y_j \}_{j \in [N]} $, 
note that $\widehat{\sfc}_j = b_{j} \sum_{k: v_{k} \in C_j} b_{k}$. Thus $Y_j$ only depends on $\{b_k: k \in A_j\}$, where $ A_j = \{v_j\} \cup C_j $.
Define a dependency graph $ \Gamma $, where $ V(\Gamma) = [N] $ and
\begin{equation*} 
%E(\Gamma) = \{ \{i, j\} : (\{v_{i}\} \cup C_i) \cap (\{v_{j}\} \cup C_j) \neq \emptyset \}.
E(\Gamma) = \{ \{i, j\} : i \neq j,\; A_i \cap A_j \neq \emptyset \}.
\end{equation*}
%\nbr{This part was very redundant. So I moved \prettyref{lmm:independent} to after \prettyref{lmm:orthogonal}, and now you can used the $A_j$ notation. }
Then $\Gamma$ has maximum degree bounded by $ \de \omega $, by \prettyref{lmm:independent}. 
\end{proof}

\subsubsection{Unbounded clique number: smoothed estimators} \label{sec:smooth}

Up to this point, we have only considered unbiased estimators of the number of connected components. 
If the sample ratio $p$ is at least $\frac{1}{2}$, \prettyref{thm:chordalmain} implies its variance is
\[
\Var[\cchat] \leq N(d+1),
\]
regardless of the clique number $\omega$ of the parent graph. However, if the clique number $\omega$ grows with $N$, for small sampling ratio $p$ the coefficients of the unbiased estimator \prettyref{eq:chordal-estimator} are as large as $\frac{1}{p^\omega}$ which results in exponentially large variance.
Therefore, in order to deal with graphs with large cliques, we must give up unbiasedness to achieve better bias-variance tradeoff. 
%This section describes a \emph{smoothing} technique that leads to biased estimators with better risk performance than the unbiased estimators from earlier.
Using a technique known as \emph{smoothing} introduced in \cite{Wu2016-2}, next we modify the unbiased estimator to achieve a good bias-variance tradeoff.

To this end, consider a discrete random variable $ L \in \mathbb{N} $ independent of everything else. Define the following estimator by discarding those terms in \prettyref{eq:chordal-estimator-peo} for which $ \widetilde{\sfc}_j $ exceeds $L$, and then averaging over the distribution of $L$.  
%Let $ \{\widetilde{v}_1, \dots, \widetilde{v}_m \} $, $ m = |S| $, be a PEO of $ \tG $ and let $ \widetilde{\sfc}_j  $ be the cardinality of $ N_{\tG}(\widetilde{v}_j) \cap \{\widetilde{v}_1, \dots, \widetilde{v}_{j-1}\} $. 
In other words, let
\begin{equation} \label{eq:chordal_smooth_estimator}
\cchat_L \triangleq \mathbb{E}_L  \qth{\frac{1}{p}\sum_{j = 1}^m\left(-\frac{q}{p}\right)^{\widetilde{\sfc}_j}\Indc\{\widetilde{\sfc}_j \leq L\}} = \frac{1}{p}\sum_{j = 1}^m\left(-\frac{q}{p}\right)^{\widetilde{\sfc}_j}\prob{L \geq \widetilde{\sfc}_j}.
\end{equation}
%\[
%\cctilde_L \triangleq -\sum_{i\geq 1}\left(-\frac{1}{p}\right)^{i}\s(K_i, \tG)\prob{L \geq i-1}.
%\]
Effectively, smoothing acts as soft truncation by introducing a tail probability that modulates the exponential growth of the original coefficients. 
The variance can then be bounded by the maximum magnitude of the coefficients in \prettyref{eq:chordal_smooth_estimator}. Like \prettyref{eq:chordal-estimator}, \prettyref{eq:chordal_smooth_estimator} can be computed in linear time.
%, since a PEO $ \{\widetilde{v}_1, \dots, \widetilde{v}_m \} $ for $\tilde G$ can be found in linear time.

%Let $ \{v_1, \dots, v_m\} $ be a PEO of $ \tG $ and let $ \widetilde{\sfc}_j  $ be the cardinality of $ N_{\tG}(\vertex_{\widetilde{\{v_1, \dots, v_N\}}}) \cap \{\vertex_{1},\dots,v_{j}\} $. In view of the reprfesentation \prettyref{eq:chordal-estimator-peo}, an alternative is to smooth with respect to the $ \widetilde{\sfc}_j $, viz.,
%\begin{equation} \label{eq:chordal_smooth_estimator}
%\cchat_L \triangleq \frac{1}{p}\sum_{j\geq 1}\left(-\frac{q}{p}\right)^{\widetilde{\sfc}_j}\prob{L \geq \widetilde{\sfc}_j}.
%\end{equation}
%Similar to \prettyref{eq:chordal-estimator-peo}, $\cchat_L $ can be computed in linear time.
%It turns out that $ \cctilde_L $ and $ \cchat_L $ have nearly the same bounds in terms of their bias, but the variance of $ \cctilde_L $ is more difficult to control than the variance of $ \cchat_L $. Therefore, we only consider $ \cchat_L $ in what follows. 

%\nbr{Let's delete (comment out) all mentioning of $\cctilde_L$.  It adds very little to the paper. LESS IS MORE.} \nb{I removed all mention of $\cctilde_L$. As you'll see below, there's only one smoothed estimator now (even for union of cliques).}

%Since $ \cchat_L $ has the form of a sum, its variance decomposes into the sum of variance and covariance terms. The %next lemma bounds the number of such cross terms.
%\nbr{This sentence can go inside the proof}

The next theorem bounds the mean-square error of $ \cchat_L $, which implies the minimax upper bound previously announced in \prettyref{thm:chordal-unbounded}. Its proof is somewhat technical and so we defer it to \prettyref{app:proofs}.

\begin{theorem} \label{thm:chordal_smooth}
Let $ L\sim \Poisson(\lambda) $ with $ \lambda = \frac{p}{2-3p}\log\left(\frac{Np}{1+\de\omega}\right) $.
If the maximum degree and clique number of $ G $ is at most $ \de $ and $ \omega $, respectively, 
then when $ p < 1/2 $,
\begin{equation*}
\Expect_G|\cchat_L-\cc(G)|^2 \leq 2N^2\left(\frac{Np}{1+\de\omega}\right)^{-\frac{p}{2-3p}}.
\end{equation*}
\end{theorem}

%\begin{remark}
%As mentioned earlier, using arguments similar to those in bounding the bias of $ \cchat_L $, we can also bound the bias $ |\Expect[\cctilde_L-\cc(G)]| $ by $ Ne^{-\lambda} $. Thus $ \cctilde_L $ and $ \cchat_L $ have similar biases. The variance of $ \cctilde_L $, however, is more difficult to control and seems to be larger in order, viz., $ \frac{\de\omega }{p} e^{\lambda c(\sqrt{\frac{\omega}{p}} \vee \omega) }  $ for some universal constant $ c > 0 $.
%\end{remark}

\subsection{Unions of cliques} \label{sec:cliques}

If the parent graph $G$ consists of disjoint union of cliques, so does the sampled graph $\tG$. 
Counting cliques in each connected components, we can rewrite the estimator \prettyref{eq:chordal-estimator} as 
\begin{equation}
\label{eq:clique-estimator}
\cchat = \sum_{r\geq 1} \pth{1-\pth{-\frac{q}{p}}^r} \cctilde_r =  \cc(\tG) - \sum_{r \geq 1} \pth{-\frac{q}{p}}^r \cctilde_r,
\end{equation}
where $ \cctilde_r $ is the number of components in the sampled graph $\widetilde G$ that have $ r $ vertices.
This coincides with the unbiased estimator proposed by Frank \cite{Frank1978} for cliques, which is, in turn, based on the estimator of Goodman \cite{Goodman1949}.
The following theorem, whose proof is given in \prettyref{app:proofs}, provides an upper bound on its variance, recovering the previous result in \cite[Corollary 11]{Frank1978}:

%The estimator \prettyref{eq:chordal-estimator} can also be written as $ \cchat = \sum_{k=1}^{\cc(G)}[1-(-\frac{q}{p})^{\widetilde{N}_k}] $, 
%where $\tN_k $ is the number of vertices in the $ k\Th$ component in $\tG$.
%Using this, it is easy to prove the following theorem, which also provides part of the minimax bound in \prettyref{thm:clique}.

\begin{theorem} \label{thm:cc-var-clique}
Let $ G $ be a disjoint union of cliques with clique number at most $ \omega $. 
Then $ \cchat $ is an unbiased estimator of $ \cc(G) $ and
\begin{equation*}
\Expect_G|\cchat - \cc(G)|^2 = \Var[\cchat] = \sum_{r=1}^N \pth{\frac{q}{p}}^r \cc_r \leq N\left( \left(\frac{q}{p}\right)^{\omega} \wedge \frac{q}{p} \right),
\end{equation*}
where $ \cc_r $ is the number of connected components in $G$ of size $ r $.
\end{theorem}

\prettyref{thm:cc-var-clique} implies that as long as we sample at least half of the vertices, i.e., $p \geq \frac{1}{2}$, for any $G$ consisting of disjoint cliques, the unbiased estimator \prettyref{eq:clique-estimator} satisfies
\begin{equation*}
\Expect_G|\cchat - \cc(G)|^2 \leq N,
\end{equation*}
regardless of the clique size. However, if $p< 1/2$, the variance can be exponentially large in $N$. Next, we use the smoothing technique again to obtain a biased estimator with near-optimal performance. To this end, consider a discrete random variable $ L \in \mathbb{N} $ and define the following estimator by truncating \prettyref{eq:clique-estimator} at the random location $L$ and average over its distribution:
%\[
%\cctilde_L \triangleq \cc(\tG) - \Expect_L \qth{\sum_{r=1}^L \left(-\frac{q}{p}\right)^r\cctilde_r },
%\]
%which again can be simplified to the following ``smoothed'' estimator:
\begin{equation} \label{eq:smooth}
\cctilde_L 
\triangleq \cc(\tG) - \Expect_L \qth{\sum_{r=1}^L \left(-\frac{q}{p}\right)^r\cctilde_r }
= \cc(\tG) - \sum_{r \geq 1} \pth{-\frac{q}{p}}^r \prob{L \geq r} \cctilde_r.
\end{equation}

The following result, proved in \prettyref{app:proofs}, bounds the mean squared error of $ \cctilde_L $ and, consequently, bounds the minimax risk in \prettyref{thm:clique}. It turns out that the smoothed estimator \prettyref{eq:smooth} with appropriately chosen parameters is nearly optimal. In fact, \prettyref{thm:clique_smooth}, whose proof is given in \prettyref{app:proofs}, gives an upper bound on the sampling complexity (see \prettyref{tab:sampling-complexity}), which, in view of \cite[Theorem 4]{WY16-distinct}, is seen to be optimal.

\begin{theorem} \label{thm:clique_smooth}
Let $ G $ be a disjoint union of cliques. Let $ L \sim \text{Pois}(\lambda) $ with $ \lambda = \frac{p}{2-3p}\log(N/4) $. 
If $ p < 1/2 $, then
\begin{equation*}
\mathbb{E}_G |\cctilde_L - \cc(G)|^2 \leq N^2 (N/4)^{-\frac{p}{2-3p}}.
\end{equation*}
\end{theorem}

\begin{remark}
Alternatively, we could specialize the estimator $ \cchat_L $ in \prettyref{eq:chordal_smooth_estimator} that is designed for general chordal graphs to the case when $ G $ is a disjoint union of cliques; however, the analysis is less clean and the results are slightly weaker than \prettyref{thm:clique_smooth}.
\end{remark}

\subsection{Non-chordal graphs} \label{sec:non-chordal}
A general graph can always be made chordal by adding edges. Such an operation is called a \emph{chordal completion} or \emph{triangulation} of a graph, henceforth denoted by $  \mathsf{TRI} $. There are many ways to triangulate a graph and this is typically done with the goal of minimizing some objective function (e.g., number of edges or the clique number). Without loss of generality, triangulations do not affect the number of connected components, since the operation can be applied to each component. 

In view of the various estimators and their performance guarantees developed so far for chordal graphs, a natural question to ask is how one might generalize those to non-chordal graphs.
One heuristic is to first triangulate the subsampled graph and then apply the estimator such as \prettyref{eq:chordal-estimator-peo} and 
\prettyref{eq:chordal_smooth_estimator} that are designed for chordal graphs. 
Suppose a triangulation operation commutes with subgraph sampling in distribution,\footnote{By ``commute in distribution'' we mean the random graphs $ \mathsf{TRI}(\tilde{G}) $ and $ \tilde{\mathsf{TRI}(G)} $ have the same distribution. That is, the triangulated sampled graph is statistically identical to a sampled version of a triangulation of the parent graph.} then the modified estimator would inherit all the performance guarantees proved for chordal graphs; unfortunately, this does not hold in general. Thus, so far our theory does not readily extend to non-chordal graphs. Nevertheless, the empirical performance of this heuristic estimator is competitive with $ \cchat $ in both performance (see \prettyref{fig:nonchordal}) and computational efficiency. Indeed, there are polynomial time algorithms that add at most $ 8k^2 $ edges if at least $ k $ edges must be added to make the graph chordal \cite{Natanzon2000}.\footnote{An implementation of graph triangulation $\texttt{R}$ is provided by the $\texttt{is\char`_chordal()}$ function in the package $\texttt{igraph}$ \cite{igraph}.} In view of the theoretical guarantees in \prettyref{thm:chordal}, it is better to be conservative with adding edges so as the maximal degree $ d $ and the clique number $ \omega $ are kept small. 
%Therefore, it is unwise to keep adding edges to each component of $ \tilde G$ until it becomes a union of cliques.

%that is, first modify the original estimator by first triangulating the subsampled graph $ G[S] \mapsto  \mathsf{TRI}(G[S]) $ and then applying $ \cchat $ to this transformed data via $ \cchat = \cchat (\mathsf{TRI}(G[S])) $. This more robust estimator 
%seems to be competitive with $ \cchat $ in both performance (see \prettyref{fig:nonchordal}) and computational efficiency. Indeed, there are polynomial time algorithms that add at most $ 8k^2 $ edges if at least $ k $ edges must be added to make the graph chordal \cite{Natanzon2000}. According to \prettyref{thm:chordal}, it is better to be conservative with adding edges so as to keep $ d $ and $ \omega $ small. Therefore, it is unwise to keep adding edges to each component of $ G[S] $ until it becomes a union of cliques.

It should be noted that blindly applying estimators designed for chordal graphs to the subsampled non-chordal graph without triangulation leads to nonsensical estimates. Thus, preprocessing the graph appears to be necessary for producing good results. We will leave the task of rigorously establishing these heuristics for future work.

\section{Lower bounds} \label{sec:lower}

\subsection{General strategy}
	\label{sec:lb-general}

Next we give a general lower bound for estimating additive graph properties (e.g. the number of connected components, subgraph counts) under the Bernoulli sampling model.
The proof uses the method of two fuzzy hypotheses \cite[Theorem 2.15]{Tsybakov2009}, which, in the context of estimating graph properties, entails constructing a pair of random graphs whose properties have different average values, and the distributions of their subsampled versions are close in total variation, which is 
ensured by matching lower-order subgraph counts or sampling certain configurations on their vertices.
The utility of this result is to use a pair of smaller graphs (which can be found in an ad hoc manner) to construct a bigger pair of graphs on $N$ vertices and produce a lower bound that scales with $N$. The proof of \prettyref{thm:mainlb} is furnished in \prettyref{app:proofs}.
\begin{theorem} \label{thm:mainlb}
Let $ f $ be a graph parameter that is invariant under isomorphisms and \emph{additive} under disjoint union, i.e., $f(G+H)=f(G)+f(H)$ \cite[p.~41]{Lovasz12}. 
Let $ \calG $ be a class of graphs with at most $ N $ vertices.
 %that is closed under disjoint graph unions, i.e., if $ G_1, \dots, G_M \in \calG $ and $ \sfv(G_1) + \dots + \sfv(G_M) \leq N $, then $ G_1 + \dots + G_M \in \calG $. 
Let $m$ and $M = N/m$ be integers.
Let $ H $ and $ H' $ be two graphs with $ m $ vertices. 
Assume that any disjoint union of the form $G_1+\dots+G_M$ is in $\calG$ where $G_i$ is either $H$ or $H'$.
Suppose $ M \geq 300 $ and $ \TV(P,P') \leq 1/300 $,
where $P$ (resp.~$P'$) denote the distribution of the isomorphism class of the sampled graph $\tH$ (resp.~$\tH'$).
Let $\tG$ denote the sampled version of $G$ under the Bernoulli sampling model with probability $p$.
Then
\begin{equation}
\inf_{\widehat{f}}\sup_{G\in\calG}\prob{|\widehat{f}\big(\tG\big)-f(G)| \geq \Delta } \geq 0.01,
\label{eq:mainlb}
\end{equation}
where $\Delta \triangleq \frac{|f(H)-f(H')|}{8}\left({\sqrt{\frac{N}{m\TV(P,P')}}}  \wedge \frac{N}{m}\right)$.
%\begin{equation*}
%\Delta = \frac{|f(H)-f(H')|}{8}\left({\sqrt{\frac{N}{m\TV(P,P')}}}  \wedge \frac{N}{m}\right).
%\end{equation*}
\end{theorem}

\subsection{Bounding total variations between sampled graphs}

The application of \prettyref{thm:mainlb} relies on the construction of a pair of small graphs $ H $ and $ H' $ whose sampled versions are close in total variation. To this end, we provide two schemes to bound $ \TV(P_{\tH}, P_{\tH'}) $ from above.

\subsubsection{Matching subgraphs}

Since $ \cc(G) $ is invariant with respect to isomorphisms, it suffices to describe the
sampled graph $ \tG $ up to isomorphisms. It is well-known that a graph $G$ can be determined up to isomorphisms by its homomorphism numbers that count the number of ways to embed a smaller graph in $G$. 
%\begin{itemize}
% \item $\ho(H, G)$ is the number of homomorphisms from $ H $ to $ G $
% \item $\n(H, G)$ is the number of edge-induced subgraphs of $ G $ that are isomorphic to $ H $
%\item $\s(H, G)$ is the number of vertex-induced subgraphs of $ G $ that are isomorphic to $ H $
%\end{itemize}
Among various versions of graph homomorphism numbers (cf.~\cite[Sec 5.2]{Lovasz12}) the one that is most relevant to the present paper is $\s(H, G)$, which, as defined in \prettyref{sec:intro}, is 
the number of \emph{vertex-induced} subgraphs of $ G $ that are isomorphic to $ H $.
Specifically, the relevance of induced subgraph counts to the subgraph sampling model is two-fold:
\begin{itemize}
\item The list of vertex-induced subgraph counts $ \{ \s(H, G) : \vertex(H) \leq N \} $ determines $ G $ up to isomorphism and hence constitutes a sufficient statistic for $ \tG $. In fact, it is further sufficient to summarize $\tG$ into the list of numbers:\footnote{This statistic cannot be further reduced because it is known that the connected subgraphs counts do not fulfill any predetermined relations in the sense that the closure of the range of their normalized version (subgraph densities) has nonempty interior \cite{Erdos1979}.}
$\{ \s(H, \tG) : \vertex(H) \leq N, \; H \; \text{is connected} \}$,
%\begin{equation*}
%\{ \s(H, \tG) : \vertex(H) \leq N, \; H \; \text{is connected} \} ,
%\label{eq:ind-ss}
%\end{equation*}
  since the count of any disconnected subgraph is a fixed polynomial of connected subgraph counts. 
This is a well-known result in the theory of graph reconstruction \cite{Whitney1932,Erdos1979,Kocay1982}.
%\cite{Biggs1978}, 
For example, for any graph $ G $, we have $\s(\Vertex\;\Vertex,\; G) = \binom{\s(\Vertex,\; G)}{2} - \s(\Edge,\; G)$
%\begin{equation*}
%\s(\Vertex\;\Vertex,\; G) = \binom{\s(\Vertex,\; G)}{2} - \s(\Edge,\; G)
%\end{equation*}
and
\begin{align*}
\s(\DoubleEdge,\; G) & = \binom{\s(\Edge,\; G)}{2} - \s(\BrokenTriangle,\; G) - 3\s(\Triangle,\; G) - \s(\PathFour,\; G) \\ & \qquad -2\s(\Square,\; G) - \s(\Paw,\; G) - 2\s(\Diamond,\; G) - 3\s(\Kfour,\; G),
\end{align*}
which can be obtained by counting pairs of vertices or edges in two different ways, respectively.
See \cite[Section 2]{McKay1997} for more examples.

%. This is appealing from an analytical and computational perspective since counting disconnected subgraphs can be challenging.

	\item Under the Bernoulli sampling model, the probabilistic law of the isomorphism class of the sampled graph is 
a polynomial in the sampling ratio $p$, with coefficients given by the induced subgraph counts.	Indeed, recall from \prettyref{eq:pmf-bern} that $\pprob{\tG \simeq H} = \s(H,G) p^{\sfv(H)}(1-p)^{\sfv(G)-\sfv(H)}$.
Therefore two graphs with matching subgraph counts for all (connected) graphs of $n$ vertices are statistically indistinguishable unless more than $n$ vertices are sampled. 
%This is the basis of our lower bound.
\end{itemize}

We begin with a refinement of the classical result that says disconnected subgraphs counts are fixed polynomials of connected subgraph counts. Below we provide a more quantitative version by showing that only those connected subgraphs which contain no more vertices than the disconnected subgraph involved. The proofs of the next set of results are given in \prettyref{app:proofs}.

\begin{lemma} \label{lmm:kocay}
Let $ H $ be a disconnected graph of $ v $ vertices. Then for any $G$,
 $ \s(H, G) $ can be expressed as a polynomial, independent of $ G $, in $\{ \s(g, G): \text{$ g $ is connected and $ \vertex(g) \leq v $}\}$.
\end{lemma}

\begin{corollary} \label{cor:counts}
Suppose $ H $ and $ H' $ are two graphs in which $ \s(h, H) = \s(h, H') $ for all connected $ h $ with $ \sfv(h) \leq v $. Then $ \s(h, H) = \s(h, H') $ for all $ h $ with $ \sfv(h) \leq v $.
\end{corollary}

\begin{lemma} \label{lmm:matching}
Let $ H $ and $ H' $ be two graphs on $ m $ vertices. If
\begin{equation} \label{eq:matching}
\s(h, H) = \s(h, H')
\end{equation}
for all connected graphs $ h $ with at most $ k $ vertices with $k \in [m]$, then
\begin{equation}
\TV(P_{\tH}, P_{\tH'}) \leq \prob{\Binom(m, p) \geq k+1} \leq \binom{m}{k+1}p^{k+1}.
\label{eq:tvmatch}
\end{equation}
Furthermore, if $ p \leq (k+1)/m $, then
\begin{equation} \label{eq:exp-ineq}
\TV(P_{\tH}, P_{\tH'}) \leq \exp\left\{-\frac{2(k+1-pm)^2}{m}\right\}.
\end{equation}
\end{lemma}

In \prettyref{fig:cycle-lower}, we give an example of two graphs $ H $ and $ H' $ on $ 8 $ vertices that have matching counts of connected subgraphs with at most $ 4 $ vertices. Thus, by \prettyref{lmm:matching}, they also have matching counts of \emph{all} subgraphs with at most $ 4 $ vertices, and if $ p \leq 5/8 $, then $ \TV(P_{\tH}, P_{\tH'}) \leq e^{-\frac{25}{4}(1-\frac{8p}{5})^2} $.

\subsubsection{Labeling-based coupling}
It is well-known that for any probability distributions $P$ and $P'$, the total variation is given by $\TV(P,P') = \inf \prob{X \neq X'}$, where the infimum is over all couplings, i.e., joint distributions of $X$ and $X'$ that are marginally distributed as $P$ and $P'$ respectively.
There is a natural coupling between the sampled graphs $ \tH $ and $ \tH' $ when we define the parent graph $H$ and $H'$ on the same set of labelled vertices. In some of the applications of \prettyref{thm:mainlb}, the constructions of $H$ and $H'$ are such that if certain configurations of the vertices are included or excluded in the sample, the resulting graphs are isomorphic. This property allows us to bound the total variation between the sampled graphs as follows.

\begin{lemma} \label{lmm:coupling}
Let $ H $ and $ H' $ be graphs defined on the same set of vertices $ V $. Let $U$ be a subset of $V$ and suppose that for any $ u \in U $, we have $ H[V \setminus \{u\} ] \simeq H'[V \setminus \{u\}] $. 
Then, the total variation $ \TV(P_{\tH}, P_{\tH'}) $ can be bounded by the probability that every vertex in $ U $ is sampled, viz.,
\begin{equation*}
\TV(P_{\tH}, P_{\tH'}) \leq 1-\prob{\tH \simeq \tH'} \leq p^{|U|}.
% \prob{\bigcap_{u \in U}\{ b_u = 1\} \cap \bigcup_{v \in V\setminus U} \{b_v = 1\}}
\end{equation*}
If, in addition, $ H[U] \simeq H'[U] $, then the total variation $ \TV(P_{\tH}, P_{\tH'}) $ can be bounded by the probability that every vertex in $ U $ is sampled and at least one vertex in $ V \setminus U $ is sampled, viz.,
\begin{equation*}
\TV(P_{\tH}, P_{\tH'}) \leq p^{|U|}(1-(1-p))^{|V|-|U|}.
\end{equation*}
\end{lemma}

In \prettyref{fig:main1}, we give an example of two graphs $ H $ and $ H' $ satisfying the assumption of \prettyref{lmm:coupling}. In this example, $ |U| = 2 $, and $ |V| = 8 $. Note that if any of the vertices in $ U $ are removed along with all their incident edges, then the resulting graphs are isomorphic. Also, since $ H[U] \simeq H'[U] $, \prettyref{lmm:coupling} implies that $ \TV(P_{\tH}, P_{\tH'}) \leq p^2(1-(1-p)^6) $.

\begin{figure} [ht]
\centering
\begin{subfigure}[t]{0.2\textwidth}
  \centering
  \includegraphics[width=1.1\linewidth]{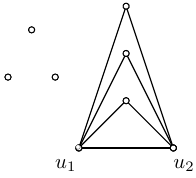}
  \caption{The graph $H$.}
  \label{fig:fig1}
\end{subfigure}%
\hspace{1cm}
\begin{subfigure}[t]{0.3\textwidth}
  \centering
  \includegraphics[width=0.9\linewidth]{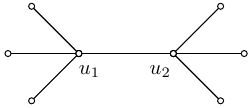}
  \caption{The graph $H'$.}
  \label{fig:fig2}
\end{subfigure}
\hspace{1cm}
\begin{subfigure}[t]{0.3\textwidth}
  \centering
  \includegraphics[width=0.9\linewidth]{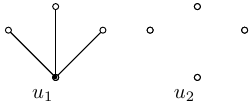}
  \caption{The resulting graph when $ u_1 $ is sampled but not $ u_ 2 $. %\nbr{can you move the two parts side by side to save some space?}
  }
  \label{fig:iso3}
\end{subfigure}
\caption{Example where $ U = \{u_1, u_2 \} $ is an edge. If any of these vertices are not sampled and all incident edges are removed, the resulting graphs are isomorphic.}
\label{fig:main1}
\end{figure}

%In \prettyref{sec:cycle} and \prettyref{sec:chordal}, 
In the remainder of the section,
we apply \prettyref{thm:mainlb}, \prettyref{lmm:matching}, and \prettyref{lmm:coupling} to derive lower bounds on the minimax risk for graphs that contain cycles and general chordal graphs, respectively. The main task is to handcraft a pair of graphs $ H $ and $ H' $ that either have matching counts of small subgraphs \emph{or} for which certain configurations of their vertices induce subgraphs that are isomorphic.

\subsection{Lower bound for chordal graphs} \label{sec:chordal}

\begin{theorem}[Chordal graphs] \label{thm:chordallower}
Let $ \calG(N,\de,\omega) $ denote the collection of all chordal graphs on $N$ vertices with maximum degree and clique number at most $ \de $ and $ \omega \geq 2 $, respectively. Assume that $ p < \frac{1}{2^{\omega}100} $. Then
\begin{equation*}
\inf_{\cchat}\sup_{G\in\calG(N,\de,\omega)} \Expect_G|\cchat-\cc(G)|^2  = \Theta_{\omega}\left(\left(\frac{N}{p^{\omega}} \vee \frac{N\de}{p^{\omega-1}} \right) \wedge N^2 \right).
\end{equation*}
%Consequently, if $ \de $ is a constant, then
%\begin{equation*}
%\inf_{\cchat}\sup_{G\in\calG(N,\de,\omega)} \Expect_G|\cchat-\cc(G)|^2 \gtrsim  \frac{N}{p^{\omega}} \wedge %N^2 .
%\end{equation*}
\end{theorem}

\begin{proof}
There are two different constructions we give, according to whether $ d \geq 2^{\omega} $ or $ d < 2^{\omega} $.
\paragraph{Case I: $ d \geq 2^{\omega} $.} For every $\omega \geq 2$ and $m \in \naturals$, we construct a pair of graphs $H$ and $H'$, such that
\begin{align}
\vertex(H)= & ~ \vertex(H')= \omega-1+ m 2^{\omega-2}  \label{eq:chordalHH1}\\
\de_{\max}(H)= & ~ \de_{\max}(H')= m 2^{\omega-3} + \omega-2 \label{eq:chordalHH2}, \qquad \omega \geq 3\\
\de_{\max}(H) = & ~ 0, \quad \de_{\max}(H') = m, \qquad \omega = 2\\
\cc(H)= & ~m+1, \quad \cc(H')= 1  \label{eq:chordalHH3}\\
|\s(K_\omega,H) - \s(K_\omega,H')| = & ~m \label{eq:chordalHH4}
\end{align}
Fix a set of $ \omega - 1$ vertices $U$ that forms a clique.
We first construct $H$. 
For every subset $S \subset U$ such that $|S|$ is even, let $V_S$ be a set of $ m $ distinct vertices such that the neighborhood of every $\vertex \in V_S$ is given by $\partial \vertex = S $.
Let the vertex set $V(H)$ be the union of $U$ and all $V_S$ such that $|S|$ is even. In particular, because of the presence of $ S = \emptyset $, $ H $ always has exactly $ m $ isolated vertices (unless $ \omega = 2 $, in which case $ H $ consists of $ m+1 $ isolated vertices).
Repeat the same construction for $H'$ with $|S|$ being odd.
Then both $H$ are $H'$ are chordal and have the same number of vertices as in \prettyref{eq:chordalHH1}, since
\[
\vertex(H)= \omega-1+ m \sum_{0\leq i \leq \omega-1,~i \text{ even}}  \binom{\omega-1}{i} = \vertex(H')= ~ \omega-1+m \sum_{0\leq i \leq \omega-1,~i \text{ odd}}  \binom{\omega-1}{i} 
\]
which follows from the binomial summation formula. 
Similarly, \prettyref{eq:chordalHH2}--\prettyref{eq:chordalHH4} can be readily verified.

We also have that
\begin{align*}
\s(K_i, H) &= \binom{\omega-1}{i}+ m \sum_{0\leq j \leq \omega-1,~j \text{ even}}  \binom{\omega-1}{j} \binom{j}{i-1} = \\
\s(K_i, H') &= \binom{\omega-1}{i}+ m \sum_{0\leq j \leq \omega-1,~j \text{ odd}}  \binom{\omega-1}{j} \binom{j}{i-1} = 
\binom{\omega-1}{i} + m\binom{\omega-1}{i-1}2^{\omega-1-i},
\end{align*}
for $ i = 1, 2, \dots, \omega -1 $. This follows from the fact that
$\sum_{0\leq j \leq \omega-1}(-1)^j \binom{\omega-1}{j} \binom{j}{i-1} = 0$ and $\sum_{0\leq j \leq \omega-1}\binom{\omega-1}{j} \binom{j}{i-1} = \binom{\omega-1}{i-1}2^{\omega-i}$.
%\[
%\sum_{0\leq j \leq \omega-1}(-1)^j \binom{\omega-1}{j} \binom{j}{i-1} = 0,
%\]
%and
%\[
%\sum_{0\leq j \leq \omega-1}\binom{\omega-1}{j} \binom{j}{i-1} = \binom{\omega-1}{i-1}2^{\omega-i}.
%\]

To compute the total variation distance between the sampled graphs, we first assume that $ H $ and $ H' $ are defined on the same set of labelled vertices $ V $. The key observation is the following: by construction, $ H[U] \simeq H'[U]$ (since $ U $ induces a clique) and, furthermore, failing to sample any vertex in $U$ results in an isomorphic graph, i.e., $ H[V \setminus \{u \} ] \simeq H'[V \setminus \{ u \}] $ for any $ u \in U $.
%For any $ v \in U $ and any every nonempty subset $ S \subset U \setminus \{v\} $, let $V_S$ be a set of $ m $ distinct vertices such that the neighborhood of every $\vertex \in V_S$ is given by $\partial \vertex = S $. The rest of the graph consists of $ m+1 $ isolated vertices.
Indeed, the structure of the induced subgraph $ H[V \setminus \{u \} ] $ can be described as follows. 
%Let $ u \in U $, and suppose that $ U $ forms a clique. 
First, let $ U $ form a clique. Next, for every nonempty subset $ S \subset U \setminus \{u\} $, attach a set of $ m $ distinct vertices (denoted by $V_S$) so that the neighborhood of every $\vertex \in V_S$ is given by $\partial \vertex = S $. 
Finally, add $ m+1 $ isolated vertices.
%The rest of $ H[V \setminus \{u \} ]$ 
See \prettyref{fig:main1} ($\omega=3$) and \prettyref{fig:main2} ($\omega=4$) for illustrations of this property and the iterative nature of this construction, in the sense that the construction of $ H $ (resp. $ H' $) for $ \omega = k+1 $ can be obtained from the construction of $ H $ (resp. $ H' $) for $ \omega = k $ by adding another vertex $ u $ to $ U $ such that $ \partial u = U $ and then adjoining $ m $ distinct vertices to every even (resp. odd) cardinality set $ S \subset U $ containing $ u $.

Thus by \prettyref{lmm:coupling},
$\TV(P_{\tH}, P_{\tH'} ) \leq p^{|U|}\pth{1-(1-p)^{|V|-|U|}}= p^{\omega-1} (1-(1-p)^{m2^{\omega-2}})$.
According to \prettyref{eq:chordalHH2}, we choose $m = \floor{ (\de -\omega+2) 2^{-\omega+3}} \geq \de 2^{-\omega+2}$ if $ \omega \geq 3 $ and $ m = d $ if $ \omega = 2 $. Then we have,
$\TV(P_{\tH}, P_{\tH'} ) = p^{\omega-1} (1-(1-p)^{d}) \leq p^{\omega-1} (pd \wedge 1)$.
The condition on $ p $ ensures that $ \TV(P_{\tH}, P_{\tH'} ) \leq p < 1/300 $.
In view of \prettyref{thm:mainlb} and \prettyref{eq:chordalHH3}, we have
\begin{equation*}
\inf_{\cchat}\sup_{G\in\calG(N,\de,\omega)}\Expect_G|\cchat-\cc(G)|^2 = \Theta_{\omega} \left( \left(\frac{N}{p^{\omega}} \vee \frac{N\de}{p^{\omega-1}} \right) \wedge N^2 \right),
\end{equation*}
provided $ d \geq 2^{\omega} $.

\begin{figure} [ht]
\centering
\begin{subfigure}[t]{0.33\textwidth}
  \centering
  \includegraphics[width=1\linewidth]{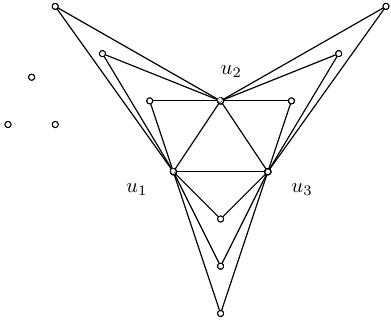}
  \caption{The graph $H$.}
  \label{fig:fig1}
\end{subfigure}%
\hspace{1cm}
\begin{subfigure}[t]{0.21\textwidth}
  \centering
  \includegraphics[width=1\linewidth]{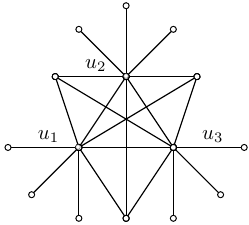}
  \caption{The graph $H'$.}
  \label{fig:fig2}
\end{subfigure}
\hspace{1cm}
\begin{subfigure}[t]{0.22\textwidth}
  \centering
  \includegraphics[width=1\linewidth]{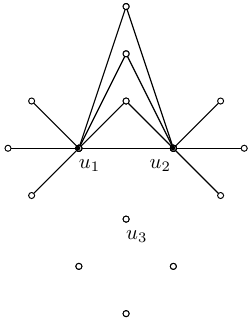}
  \caption{The resulting graph when $ u_1 $ and $ u_2 $ are sampled but not $ u_ 3 $.}
  \label{fig:iso4}
\end{subfigure}
\caption{Example for $ \omega = 4 $ and $ m = 3 $, where $ U = \{ u_1, u_2, u_3 \} $ form a triangle. If any one or two (as shown in the figure) of these vertices are not sampled and all incident edges are removed, the resulting graphs are isomorphic.}
\label{fig:main2}
\end{figure}

\paragraph{Case II: $ d \leq 2^{\omega} $.}
In this case, the previous construction is no longer feasible and we must construct another pair of graphs with a smaller maximum degree. To this end, we consider graphs $ H $ and $ H' $ consisting of disjoint cliques of size at most $ \omega \geq 2 $, such that
%\begin{align}
%\vertex(H)= & ~ \vertex(H')= \omega2^{\omega-2} \nonumber \\
%\de_{\max}(H)= & ~ \de_{\max}(H')= \omega-1 \nonumber \\
%%\cc(G)= & ~2^{\omega-1}, \quad \cc(G')= 2^{\omega-1}-1 \\
%|\cc(H) - \cc(H')| = & ~ 1. \label{eq:chordalGG1}
%\end{align}
\begin{equation}
\vertex(H)= \vertex(H')= \omega2^{\omega-2}, \quad
\de_{\max}(H)= \de_{\max}(H')= \omega-1, \quad
|\cc(H) - \cc(H')| = 1. \label{eq:chordalGG1}
\end{equation}
If $ \omega $ is odd, we set
\begin{equation}
\begin{aligned}
H = & ~  \tbinom{\omega}{\omega}K_{\omega} + \tbinom{\omega}{\omega-2}K_{\omega-2} + \cdots + \tbinom{\omega}{3}K_{3} + \tbinom{\omega}{1}K_{1} \\
H' = & ~  \tbinom{\omega}{\omega-1}K_{\omega-1} + \tbinom{\omega}{\omega-3}K_{\omega-3} + \cdots + \tbinom{\omega}{4}K_{4} + \tbinom{\omega}{2}K_{2}.
\end{aligned}
\label{eq:chordallb-odd}
\end{equation}
If $ \omega $ is even, we set
\begin{equation}
\begin{aligned}
H = & ~  \tbinom{\omega}{\omega}K_{\omega} + \tbinom{\omega}{\omega-2}K_{\omega-2} + \cdots + \tbinom{\omega}{4}K_{4} + \tbinom{\omega}{2}K_{2} \\
H' = & ~  \tbinom{\omega}{\omega-1}K_{\omega-1} + \tbinom{\omega}{\omega-3}K_{\omega-3} + \cdots + \tbinom{\omega}{3}K_{3} + \tbinom{\omega}{1}K_{1}.
\end{aligned}
\label{eq:chordallb-even}
\end{equation}
For example, 
for $\omega=3$, \prettyref{eq:chordallb-odd} becomes $H = \Triangle + 3 \times \Vertex$ and $H' = 3\times \Edge$~;
for $\omega=4$, \prettyref{eq:chordallb-even} becomes $H = \Kfour + 6 \times \Edge$ and $H' = 4 \times \Triangle + 	4 \times \Vertex$.

%See \prettyref{fig:main3} and \prettyref{fig:main4} for examples of this construction, 
%
%
%\begin{figure}[ht]
%\centering
%\begin{subfigure}[t]{0.35\textwidth}
  %\centering
  %\includegraphics[width=0.35\textwidth]{Gcomplete2.pdf}
  %\caption{The graph of $H$.}
  %\label{fig:fig1}
%\end{subfigure}%
%\hspace{3cm}
%\begin{subfigure}[t]{0.3\textwidth}
  %\centering
  %\includegraphics[width=0.35\textwidth]{Gpcomplete2.pdf}
  %\caption{The graph of $H'$.}
  %\label{fig:fig2}
%\end{subfigure}
%\caption{Illustration for the construction in \prettyref{eq:chordallb-odd} for $ \omega = 3 $. Each graph contains a matching number of cliques of size up to $ 2 $.}
%\label{fig:main3}
%\end{figure}
%
%\begin{figure}[ht]
%\centering
%\begin{subfigure}[t]{0.27\textwidth}
  %\centering
  %\includegraphics[width=1\linewidth]{Gcomplete.pdf}
  %\caption{The graph of $H$.}
  %\label{fig:fig1}
%\end{subfigure}%
%\hspace{3cm}
%\begin{subfigure}[t]{0.3\textwidth}
  %\centering
  %\includegraphics[width=1\linewidth]{Gpcomplete.pdf}
  %\caption{The graph of $H'$.}
  %\label{fig:fig2}
%\end{subfigure}
%\caption{Illustration for the construction in \prettyref{eq:chordallb-even} for $ \omega = 4 $. Each graph contains a matching number of cliques of size up to $ 3 $.}
%\label{fig:main4}
%\end{figure}

Next we verify that $H$ and $H'$ have matching subgraph counts. Indeed, for $ i = 1,2,\dots, \omega-1 $,
$\s(K_i, H)-\s(K_i, H') = \sum_{k=i}^{\omega}(-1)^k\tbinom{\omega}{k}\tbinom{k}{i} = 0$ 
and 
$\s(K_i, H)=\s(K_i, H') = \frac{1}{2}\sum_{k=i}^{\omega}\tbinom{\omega}{k}\tbinom{k}{i} = 2^{\omega-1-i}\tbinom{\omega}{i}$.
%\begin{align*}
%\s(K_i, H)-\s(K_i, H') = \sum_{k=i}^{\omega}(-1)^k\tbinom{\omega}{k}\tbinom{k}{i} = 0,
%\end{align*}
%and 
%\begin{equation*}
%\s(K_i, H)=\s(K_i, H') = \frac{1}{2}\sum_{k=i}^{\omega}\tbinom{\omega}{k}\tbinom{k}{i} = 2^{\omega-1-i}\tbinom{\omega}{i}.
%\end{equation*}
Hence $ H $ and $ H' $ contain matching number of cliques up to size $ \omega - 1 $.
Note that the only connected induced subgraphs of $ H $ and $ H' $ with at most $ \omega-1 $ vertices are cliques. Consequently, by \prettyref{eq:tvmatch}, $ \TV(P_{\tH}, P_{\tH'}) \leq \binom{\omega2^{\omega-2}}{\omega}p^{\omega} $ and together with \prettyref{thm:mainlb} and \prettyref{eq:chordalGG1}, we have
\begin{equation*}
\inf_{\cchat}\sup_{G\in\calG(N,\de,\omega)}\Expect_G|\cchat-\cc(G)|^2 \geq \Omega_{\omega} \left( \frac{N}{p^{\omega}} \wedge N^2 \right) = 
\Theta_{\omega} \left( \left(\frac{N}{p^{\omega}} \vee \frac{N\de}{p^{\omega-1}} \right) \wedge N^2 \right),
\end{equation*}
where the last inequality follows from the current assumption that $d \leq 2^\omega$. The condition on $ p $ ensures that $ \TV(P_{\tH}, P_{\tH'} ) \leq p2^{\omega-2} < 1/300 $.
%Now, if $ p < 1/d $,  $ \left( \frac{N}{p^{\omega}} \wedge N^2 \right) = \left( \left(\frac{N}{p^{\omega}} \vee \frac{N\de}{p^{\omega-1}} \right) \wedge N^2 \right) $ and if $ p \geq 1/d $ and $ d \leq 2^{\omega} $, 
%\begin{align*}
%\left( \frac{N}{p^{\omega}} \wedge N^2 \right) & = 
%\left( \frac{Npd}{p^{\omega}pd} \wedge N^2 \right) \\
%& \geq \left( \frac{Npd}{p^{\omega}d} \wedge N^2 \right) \\
%& \geq \left( \frac{Nd}{p^{\omega-1}2^{\omega}} \wedge N^2 \right) \\
%& = \Omega_{\omega} \left( \frac{N\de}{p^{\omega-1}} \wedge N^2 \right).
%\end{align*}
%Thus, we have shown that when $ d \leq 2^{\omega} $,
%\begin{equation*}
%\inf_{\cchat}\sup_{G\in\calG(N,\de,\omega)}\Expect_G|\cchat-\cc(G)|^2 = \Omega_{\omega} \left( \left(\frac{N}{p^{\omega}} \vee \frac{N\de}{p^{\omega-1}} \right) \wedge N^2 \right).
%\end{equation*}
\end{proof}

\section*{Acknowledgment}
	\label{sec:ack}
	The authors are grateful to Ben Rossman for fruitful discussions and Richard Stanley for helpful comments on \prettyref{eq:clique-id}.

\appendix

\section{Additional proofs}\label{app:proofs}

In this appendix, we give the proofs of \prettyref{lmm:unique}, \prettyref{lmm:permutation}, \prettyref{lmm:orthogonal}, \prettyref{lmm:independent}, \prettyref{thm:chordal_smooth}, \prettyref{thm:cc-var-clique}, \prettyref{thm:clique_smooth}, \prettyref{thm:mainlb}, \prettyref{lmm:kocay}, and \prettyref{lmm:matching}. We also state the concentration inequality from \prettyref{lmm:janson} that was used in the proof of \prettyref{thm:chordal}.

\begin{proof}[Proof of \prettyref{lmm:unique}]
By \cite[Theorem 5.3.26]{west-book}, the chromatic polynomial of $ G $ is
\begin{equation*}
\chi(G; x) = (x - \sfc_1)\cdots (x-\sfc_N) = (x - \sfc'_1)\cdots (x-\sfc'_N).
\end{equation*}
The conclusion follows from the uniqueness of the chromatic polynomial (and its roots). 
\end{proof}

\begin{proof}[Proof of \prettyref{lmm:permutation}]
Note that $ \{v_1, \dots, v_N\} $ is also a PEO\footnote{When we say a PEO $ \{v_1, \dots, v_N \} $ of $G$ is also a PEO of $ \tG=G[S] $, it is understood in the following sense: for any $ v_j \in S $, $ N_{\tG}(v_j) \cap \{v_i \in S: i < j \} $ is a clique in $ G[S] $. } of $ \tG $ and hence by \prettyref{lmm:unique}, there is a bijection between $ \{ \widetilde{\sfc}_{j} : j \in [m] \} $ and $ \{ \widehat{\sfc}_{j} : j \in [N] \} $. Therefore
%Since we can sum $ \widehat{\sfg} $ in any order, we have
\begin{equation*}
\widehat{\sfg} = \sum_{j=1}^m g(\widetilde{\sfc}_j) = \sum_{j=1}^Nb_j g(\widehat{\sfc}_j). \qedhere
\end{equation*}
\end{proof}

\begin{proof}[Proof of \prettyref{lmm:orthogonal}]
%The second identity follows from the first since $ f(N_{S}) $ and $ f(N_{ T}) $ have zero mean.
Note that $ N_{S} + N_{T} = N_{S \setminus T} + N_{T \setminus S}  + 2N_{S \cap T} $, where $ N_{S \setminus  T} $, $ N_{T \setminus  S} $, and $ N_{S \cap T} $ are independent binomially distributed random variables. By independence, we have
\begin{align*}
\Expect[f(N_{S})f(N_{ T})] & = \expect{\left(-\frac{q}{p}\right)^{N_{S}+N_{T}}} = \expect{\left(-\frac{q}{p}\right)^{N_{S \setminus T} + N_{T \setminus  S} + 2N_{S \cap T}}} \\
& = \expect{\left(-\frac{q}{p}\right)^{N_{S \setminus  T}}}\expect{\left(-\frac{q}{p}\right)^{N_{T \setminus  S}}}\expect{\left(-\frac{q}{p}\right)^{2N_{S \cap T}}}.
\end{align*}
Finally, note that if $ S \neq T $, then at least one of $ \Expect[(-\frac{q}{p})^{N_{S \setminus  T}}] $ or $ \Expect[(-\frac{q}{p})^{N_{T \setminus  S}}] $ is zero. If $ S =  T $, we have
\begin{align*}
\Expect[f(N_{S})^2] = \expect{\pth{-\frac{q}{p}}^{2N_{S}}} = \pth{\frac{q}{p}}^{|{S}|}.
\end{align*}
\end{proof}

\begin{proof}[Proof of \prettyref{lmm:independent}]
Let $ \sfc_j = |C_j| $. To prove \prettyref{eq:Cjnumber}, we will show that for any fixed $ j $,
\begin{equation*}
 | \{ i \in [N]: i \neq j,\; C_j = C_i \neq \emptyset \} | \leq \de-\sfc_j \leq \de-1.
\end{equation*}
By definition of the PEO, $ |N_G(v)| \geq \sfc_j $ for all $ v \in C_j $. For any $ i \in [N] $ such that $ C_j = C_i \neq \emptyset $, $ v_i \in N_G(v) $ for all $ v \in C_j $. Also, the fact that $ C_j = C_i \neq \emptyset $ makes it impossible for $ v_i \in C_j $. This shows that
\begin{equation*}
\sfc_j +  | \{ i \in [N]: i \neq j,\; C_j = C_i \neq \emptyset \} | \leq |N_G(v)| \leq d,
\end{equation*}
and hence the desired \prettyref{eq:Cjnumber}.

Next, we show \prettyref{eq:maxdegA}. Let $ \sfa_j = |A_j| $. We will prove that for any fixed $ j $,
\begin{equation} \label{eq:intersect}
| \{ i \in [N] : i \neq j, \; A_i \cap A_j \neq \emptyset \} | \leq \de\sfa_j - (\sfa_j-1)^2.
\end{equation}
This fact immediately implies \prettyref{eq:cross-terms} by noting that $ \sfa_j \leq \omega $.
To this end, note that
\begin{align*} 
| \{ i \in [N] : i \neq j, \; A_i \cap A_j \neq \emptyset \} | & =
| \{ i \in [N] : i \neq j, \; v_{i} \notin A_j, \; A_i \cap A_j \neq \emptyset \} |+ \\ & \qquad 
| \{ i \in [N] : i \neq j, \; v_{i} \in A_j \} |,
\end{align*}
where the second term is obviously at most $ \sfa_j-1$. 
Next we prove that the first term is at most $ (\de+1-\sfa_j)\sfa_j $, 
which, in view of $ (\de+1-\sfa_j)\sfa_j + (\sfa_j-1) = \de\sfa_j - (\sfa_j-1)^2 $, implies the desired \prettyref{eq:intersect}.
Suppose, for the sake of contradiction, that
\begin{equation*}
| \{ i \in [N] : i \neq j, \; v_{i} \notin A_j, \; A_i \cap A_j \neq \emptyset \} | \geq (\de+1-\sfa_j)\sfa_j+1
\end{equation*}
Then at least $ (\de+1-\sfa_j)\sfa_j+1 $ of the $ A_i $ have nonempty intersection with $ A_j $, meaning that at least $ (\de+1-\sfa_j)\sfa_j+1 $ vertices outside $ A_j $ are incident to vertices in $ A_j $. 
By the pigeonhole principle, there is at least one vertex $ u \in A_j $ which is incident to $\de+2-\sfa_j$ of those vertices outside $ A_j $.
Moreover, the vertices in $ A_j $ form a clique of size $ \sfa_j $ in $ G $ by definition of the PEO. This implies that $ |N_G(u)| \geq (\sfa_j-1) + (\de-\sfa_j+2) = \de+1 $, contradicting the maximum degree assumption and completing the proof.
\end{proof}

To prove the high-probability bound \prettyref{eq:chordal-concentration}, we used a concentration inequality for the sum of dependent random variables due to Janson \cite{Janson2004}. This result, stated next, can be distilled from \cite[Theorem 2.3]{Janson2004}. The two-sided version of the concentration inequality therein also holds; see the paragraph before \cite[Equation (2.3)]{Janson2004}.

\begin{lemma}
\label{lmm:janson}	
Let $ X = \sum_{j\in [N]} Y_j $, where $ |Y_j-\expect{Y_j}| \leq b $ almost surely. Let $ S = \sum_{j\in [N]} \Var[Y_j] $. 
Let $ \Gamma =([N],E(\Gamma))$ be a dependency graph for $ \{ Y_j \}_{j \in [N]} $ in the sense that if $ A \subset [N] $, and $ i \in [N]\backslash A $ does not belong to the neighborhood of any vertex in $ A $, then $ Y_i $ is independent of $ \{ Y_j \}_{j \in A} $. Furthermore, suppose $ \Gamma $ has maximum degree $ d_{\max} $. Then, for all $ t \geq 0 $,
\begin{equation*}
\prob{ | X - \expect{X} | \geq t } \leq 2\exp\left \{-\frac{8t^2}{25(d_{\max}+1)(S+bt/3)} \right\}.
\end{equation*}
\end{lemma}

\begin{proof}[Proof of \prettyref{thm:chordal_smooth}]
%\nbr{See Thm7.jpg in the dropbox folder for my suggestion on how to write this part.}
Let $ \{v_1, \dots, v_N\} $ be a PEO of the parent graph $ G $ and let $ \{\widetilde{v}_{1},\dots,\widetilde{v}_{m}\} $, $ m = |S| $, be a PEO of $ \tG $ and $ \widetilde{\sfc}_j = |N_{\tG}(\widetilde{v}_{j}) \cap \{\widetilde{v}_{1},\dots,\widetilde{v}_{j-1}\}| $. 
%\nbr{I don't understand. Is the estimator same as \prettyref{eq:chordal_smooth_estimator}? If so, why ``define''?} \nb{It was the same estimator and so I made it clearer.}
Let $ \widehat{\sfc}_j = |N_{\tG}(\vertex_{j}) \cap \{\vertex_{1},\dots,v_{j-1}\}| $ and $ \sfc_j  =|N_G(\vertex_{j}) \cap \{\vertex_{1},\dots,v_{j-1}\}| $. By \prettyref{lmm:permutation}, we can rewrite $ \cchat_L $ as
\begin{align*}
\cchat_L
%& = \frac{1}{p}\sum_{j = 1}^m\left(-\frac{q}{p}\right)^{\widetilde{\sfc}_j}\prob{L \geq \widetilde{\sfc}_j} \\
& = \frac{1}{p}\sum_{j\geq 1}b_{j}\left(-\frac{q}{p}\right)^{\widehat{\sfc}_j}\prob{L \geq \widehat{\sfc}_j}, 
\end{align*}
where $ \widehat{\sfc}_j  \sim \Binom(\sfc_j, p) $ conditioned on $\{ b_{j} = 1 \}$.

We compute the bias and variance of $ \cchat_L $ and then optimize over $ \lambda $.
First, 
\begin{align*}
\Expect[\cc(G)-\cchat_L] & = \frac{1}{p}\sum_{j=1}^{N}\Expect[b_{j}\left(-\frac{q}{p}\right)^{\widehat{\sfc}_j}\prob{L < \widehat{\sfc}_j}] = \sum_{j=1}^{N}\sum_{i=0}^{\sfc_j}\binom{\sfc_j}{i}p^iq^{\sfc_j-i}\left(-\frac{q}{p}\right)^{i}\prob{L < i} \\
& = \sum_{j=1}^{N}q^{\sfc_j}\sum_{i=0}^{\sfc_j}\binom{\sfc_j}{i}(-1)^{i}\prob{L < i} = \sum_{j=1}^{N}q^{\sfc_j}\sum_{i=0}^{\sfc_j}\binom{\sfc_j}{i}(-1)^{i}\sum_{\ell=0}^{i-1}\prob{L=\ell} \\
& = \sum_{j=1}^{N}q^{\sfc_j}\sum_{\ell=0}^{\sfc_j-1}\prob{L=\ell}\sum_{i=\ell+1}^{\sfc_j}\binom{\sfc_j}{i}(-1)^{i}\\
%& \stepa{=} \sum_{j=1}^{N}q^{\sfc_j}\sum_{\ell=0}^{\sfc_j-1}\prob{L=\ell}\binom{\sfc_j-1}{\ell}(-1)^{\ell+1} \\
& \stepa{=} \sum_{j=1}^{N}q^{\sfc_j}   \Expect_L\qth{ \binom{\sfc_j-1}{L}(-1)^{L+1}} \\
& \stepb{=} -e^{-\lambda}\sum_{j=1}^Nq^{\sfc_j}L_{\sfc_j-1}(\lambda),
\end{align*}
where (a) follows from the fact that $\sum_{i=\ell+1}^{k}\binom{k}{i}(-1)^{i}= \binom{k-1}{\ell}(-1)^{\ell+1}$, and 
(b) follows from 
\begin{equation} \label{eq:laguerre}
\Expect_L\qth{ \binom{k-1}{L}(-1)^{L+1}} = e^{-\lambda} L_{k-1}(\lambda),
\end{equation}
where $ L_{m} $ is the Laguerre polynomial of degree $ m $, which satisfies $ |L_m(x)| \leq e^{x/2} $ for all $ m \geq 0 $ and $ x \geq 0 $ \cite{Stegun1964}. Thus
\begin{equation} \label{eq:bias}
|\Expect[\cchat_L-\cchat]| \leq Ne^{-\lambda/2}.
\end{equation}

%Since $ \cchat_L $ has the form of a sum, its variance decomposes into the sum of variance and covariance terms. 
To bound the variance, write 
$ \cchat_L = \frac{1}{p}\sum_{j=1}^{N}W_j $, where $ W_j = b_{j}(-\frac{q}{p})^{\widehat{\sfc}_j}\prob{L \geq \widehat{\sfc}_j} $.
Thus
\begin{equation}
\Var[\cchat_L] = \frac{1}{p^2} \sum_{j\in [N]} \Var[W_j] + \frac{1}{p^2} \sum_{i \neq j} \Cov[W_i,W_j]
\label{eq:varccL}
\end{equation}
 %and $\widehat{\sfc}_j$ is defined in \prettyref{}. 
Note that $W_j$ is a function of $\{b_{\ell}: v_{\ell}\in A_j, \ell\in[N]\}$, where $A_j$ is defined in \prettyref{eq:Aj}.
Using \prettyref{lmm:independent}, we have 
\begin{equation} 
\label{eq:cross-terms}
|\{ (i,j) \in [N]^2 : i \neq j, \; A_i \cap A_j \neq \emptyset \}| \leq N\de \omega.
\end{equation}
Thus the number of cross terms in \prettyref{eq:varccL} is at most $N\de \omega$ thanks to \prettyref{eq:cross-terms}.
%\nbr{How did you get the next bound? Add details? I understand you used $\Cov(W_i,W_j) \leq \sup_{j \in [N]} \Var(W_j)$. %I just don't get the conditional variance business.} 
Thus,
\begin{equation}
\Var[\cchat_L] \leq \frac{N(1+\de\omega)}{p^2}\max_{1 \leq j\leq N}\Var[W_j].
\label{eq:varccL2}
\end{equation}
Finally, note that if $ p < 1/2 $, then
\begin{equation}
\Var[W_j] \leq p\left( \sup_{k\geq 0} \left\{\left( \frac{q}{p} \right)^k \prob{L \geq k} \right\} \right)^2 \leq p\pth{\Expect_L \qth{\pth{ \frac{q}{p}}^{L}}}^2 = p\exp \left\{2\lambda\left(\frac{q}{p}-1\right)\right \}.
\label{eq:varccL3}
\end{equation}
% and if $ p \geq 1/2 $, we have $|W_j| \leq b_j$ and hence $ \Var[W_j] \leq p $.
Combining \prettyref{eq:bias}, \prettyref{eq:varccL2}, and \prettyref{eq:varccL3}, we have
\[
\Expect_G|\cchat_L-\cc(G)|^2
\leq N^2 e^{-\lambda} + 
\frac{N(1+\de\omega)}{p} \exp \left\{2\lambda\left(\frac{q}{p}-1\right) \right\}.
\]
The choice of $ \lambda $ yields the desired bound.
\end{proof}

\begin{proof}[Proof of \prettyref{thm:cc-var-clique}]
The estimator \prettyref{eq:chordal-estimator} can also be written as $ \cchat = \sum_{k=1}^{\cc(G)}[1-(-\frac{q}{p})^{\widetilde{N}_k}] $, where 
%Observe that $  \cchat = \sum_{k=1}^{\cc(G)}[1-(-\frac{q}{p})^{\widetilde{N}_k}]  $ and 
$\widetilde{N}_k$ is the number of sampled vertices from the $k\Th$ component. 
Then $ \widetilde{N}_k \sdistr \Binom(N_k, p) $. Thus,
\begin{equation*}
\Var[\cchat]= \sum_{k=1}^{\cc(G)}\pth{\frac{q}{p}}^{N_k} = \sum_{r=1}^{N} \pth{\frac{q}{p}}^r \cc_r.
\end{equation*}
The upper bound follows from the fact that $ \cc_r = 0 $ for all $ r > \omega $ and $ \sum_{r=1}^N\cc_r = \cc(G) \leq N $.
\end{proof}

\begin{proof}[Proof of \prettyref{thm:clique_smooth}]
The bias of this estimator is seen to be
\begin{equation*}
\expect{\cc(G) -\cctilde_L} = \sum_{k=1}^{\cc(G)}\expect{\prob{L < \widetilde{N}_k}\left(-\frac{q}{p}\right)^{\widetilde{N}_k}}.
\end{equation*}
Note that
\begin{align*}
\expect{\prob{L < \widetilde{N}_k}\left(-\frac{q}{p}\right)^{\widetilde{N}_k}} & = \sum_{r=1}^N\prob{L < r}\pth{-\frac{q}{p}}^r \prob{ \widetilde{N}_k = r} \\
& = \sum_{i=0}^{N-1}\prob{L=i}\sum_{r=i+1}^N\pth{-\frac{q}{p}}^r\prob{\widetilde{N}_k = r}.
\end{align*}
Since $ \widetilde{N}_k \sim \text{Bin}(N_k,p) $, it follows that 
\begin{equation*}
\sum_{r=i+1}^N\pth{-\frac{q}{p}}^r\prob{\widetilde{N}_k = r} = q^{N_k}\sum_{r=i+1}^N\binom{N_k}{r}(-1)^r = q^{N_k}(-1)^{i+1}\binom{N_k-1}{i}.
\end{equation*}
Putting these facts together, we have
\begin{equation*}
\expect{\cc(G) -\cctilde_L} = -\sum_{k=1}^{\cc(G)}q^{N_k}P_{N_k-1}(\lambda) = \sum_{k=1}^{\cc(G)}q^{N_k}\mathbb{E}_L\left[\binom{N_k-1}{L}(-1)^{L+1}\right],
\end{equation*}
Analogous to \prettyref{eq:laguerre}, we have $ \left|\mathbb{E}_L\left[\binom{N_k-1}{L}(-1)^{L+1}\right]\right| \leq e^{-\lambda/2} $, and hence by the Cauchy-Schwarz inequality,
\begin{equation}
|\expect{\cc(G) -\cctilde_L}| \leq e^{-\lambda/2}\sqrt{N\sum_{k=1}^{\cc(G)}q^{N_k}}.
\label{eq:clique1}
\end{equation}

For the variance of $ \cctilde_L $, note that $ \cctilde_L = \sum_{k=1}^{\cc(G)} W_k $, where $ W_k \triangleq 1-\prob{L\geq  \widetilde{N}_k}\left(-\frac{q}{p}\right)^{ \widetilde{N}_k} $. The $ W_k $ are independent random variables and hence
\begin{equation*}
\Var[\cctilde_L] = \sum_{k=1}^{\cc(G)} \Var[W_k] \leq \sum_{k=1}^{\cc(G)} \mathbb{E}W^2_k.
\end{equation*}
Also,
\begin{equation*}
W^2_k \leq \max_{1 \leq r \leq N}\left\{ 1-\prob{L \geq r}\pth{-\frac{q}{p}}^r \right\}^2\Indc\{ \widetilde{N}_k \geq 1\}.
\end{equation*}
This means that
\begin{equation*}
\Var[\cctilde_L] \leq \max_{1 \leq r \leq N}\left\{ 1-\prob{L \geq r}\pth{-\frac{q}{p}}^r \right\}^2\sum_{k=1}^{\cc(G)}(1-q^{N_k}).
\end{equation*}
Since $ p < 1/2 $, we have
\begin{align*}
\prob{L \geq r}\left(\frac{q}{p}\right)^r & = \sum_{i= r}^{\infty}\prob{L=i}\left(\frac{q}{p}\right)^r \leq \sum_{i= r}^{\infty}\prob{L=i}\left(\frac{q}{p}\right)^i \\ & \leq \sum_{i=0}^{\infty}\prob{L=i}\left(\frac{q}{p}\right)^i = \Expect_L\left(\frac{q}{p}\right)^L = e^{\lambda(\frac{q}{p}-1)}.
\end{align*}
%If $ p \geq 1/2 $, we have $ \prob{L \geq r}\left(\frac{q}{p}\right)^r \leq 1 $.
Thus, it follows that 
\begin{equation}
\Var[\cctilde_L] \leq 4e^{2\lambda(\frac{q}{p}-1)}\sum_{k=1}^{\cc(G)}(1-q^{N_k}). 
\label{eq:clique2}
\end{equation}
Combining \prettyref{eq:clique1} and \prettyref{eq:clique2} yields
\begin{equation*}
\mathbb{E}|\cctilde_L - \cc(G)|^2 \leq 4e^{2\lambda(\frac{q}{p}-1)}\sum_{k=1}^{\cc(G)}(1-q^{N_k}) + N e^{-\lambda}\sum_{k=1}^{\cc(G)} q^{N_k}
\leq \cc(G) \max\sth{4e^{2\lambda(\frac{q}{p}-1)}, N e^{-\lambda}}.
\end{equation*}
Choosing $ \lambda = \frac{p}{2-3p}\log(N/4) $ leads to $ 4e^{2\lambda(\frac{q}{p}-1)} = N e^{-\lambda} $ and completes the proof.
%If $ p \geq 1/2 $, we set $4 = N e^{-\lambda} $, which determines $ \lambda = \log(N/4) $.
\end{proof}

\begin{proof}[Proof of \prettyref{thm:mainlb}]
Fix $\alpha \in (0,1)$.
Let $ M = N/m $ and $ G = G_1+ G_2 + \cdots + G_M $, where 
$ G_i \simeq H$ or $H'$ with probability $\alpha$ and $1-\alpha$, respectively.
%$ G_i \simeq \begin{cases}
    %H & \text{with probability } \alpha \\
    %H' & \text{with probability } 1-\alpha
  %\end{cases} $.
	Let $\Prob_\alpha$ denote the law of $G$ and $\Expect_\alpha$ the corresponding expectation.	
Assume without loss of generality that $ f(H) > f(H') $. Note that $ \mathbb{E}_{\alpha}f(G) = M[\alpha f(H) + (1-\alpha)f(H')] $. 
%Since we know the labels $ \{ b_{\vertex} \} $, we can decompose $ \tG_i $ via $ \tG_i \simeq h + (m-\sfv(h))K_1 $,
%where $ b_{\vertex} = 1 $ for all $ \vertex \in V(h) $ and $ b_{\vertex} = 0 $ otherwise.

Let $ \tG_i $ be the sample version of $ G_i $. Then $\tG=\tG_1+\dots+\tG_M$.
For each subgraph $ h $, by \prettyref{eq:pmf-bern}, we have
\begin{equation*}
\prob{\tG_i \simeq h \mid G_i \simeq H} = \s(h, H)p^{\sfv(h)}(1-p)^{m-\sfv(h)},
\end{equation*}
and
\begin{equation*}
\prob{\tG_i \simeq h \mid G_i \simeq H'} = \s(h, H')p^{\sfv(h)}(1-p)^{m-\sfv(h)}.
\end{equation*}
Let $ P \triangleq P_{\tH} = \mathcal{L}(\tG_i \mid G_i \simeq H) $ and $ P' \triangleq P_{\tH'} = \mathcal{L}(\tG_i \mid G_i \simeq H') $. Then the law of each $\tG_i$ is simply a mixture $ P_{\alpha} \triangleq \mathcal{L}(\tG_i) = \alpha P + (1-\alpha)P' $.
Furthermore, $ (\tG_1,\tG_2,\dots,\tG_M)' \sim P^{\otimes M}_{\alpha} $. 

To lower bound the minimax risk of estimating the functional $ f(G)$, we apply the method of two fuzzy hypotheses \cite[Theorem 2.15(i)]{Tsybakov2009}.
To this end, consider a pair of priors, that is, the distribution of $G$ with 
%$ \alpha_0 = 1/2 $ and $ \alpha_1 = 1/2+\delta $, where $ \delta \in [0, 1/2] $. 
$\alpha=\alpha_0=1/2$ and $\alpha_1=1/2+\delta$, respectively, where $ \delta \in [0, 1/2] $ is to be determined. 
To ensure that the values of $f(G)$ are separated under the two priors,
%, i.e., 
%there exist $ L$, $\Delta > 0 $ and $ \beta_0, \beta_1 \in (0,1) $ such that
%\begin{equation*}
%\Prob_{\alpha_0}[f(G) \leq L] = \prob{\Binom(M,\alpha_0) \leq L } \geq 1-\beta_0,
%\end{equation*}
%and
%\begin{equation*}
%\Prob_{\alpha_1}[f(G) \geq L+2\Delta] = \prob{\Binom(M,\alpha_1) \geq L+2\Delta} \geq 1-\beta_1.
%\end{equation*}
note that $ f(G) \eqdistr (f(H)-f(H'))\Binom(M, \alpha) + f(H')M $. 
Define $ L = f(H)(1/2+\delta/4)M + f(H')(1/2-\delta/4)M $ and 
\[
\Delta \triangleq \frac{1}{4} (\mathbb{E}_{\alpha_1}f(G)-\mathbb{E}_{\alpha_0}f(G)) = \frac{M\delta}{4} (f(H)-f(H')) .
\] 
%By the multiplicative Chernoff bound \cite[Proposition 2.4]{Angluin1979}, 
By Hoeffding's inequality, for any $ \delta \geq 0 $,
\begin{equation*}
\Prob_{\alpha_0}[f(G) \leq L] = \prob{\Binom(M,\alpha_0) \leq M\alpha_0 + M\delta/4} \geq 1-e^{-\delta^2M/8} \triangleq 1 -\beta_0.
\end{equation*}
and
\begin{equation*}
\Prob_{\alpha_1}[f(G) \geq L+2\Delta] = \prob{\Binom(M,\alpha_1) \geq M\alpha_1 - M\delta/4} \geq 1-e^{-\delta^2M/8} \triangleq 1 -\beta_1.
\end{equation*}
Invoking \cite[Theorem 2.15(i)]{Tsybakov2009}, we have
\begin{equation}
\inf_{\widehat{f}}\sup_{G\in\calG}\prob{|\widehat{f}\big(\tG\big)-f(G)| \geq \Delta } \geq \frac{1-\TV(P^{\otimes M}_{\alpha_0}, P^{\otimes M}_{\alpha_1})-\beta_0-\beta_1}{2}.
\label{eq:fuzzy}
\end{equation}
The total variation term can be bounded as follows:
\begin{align*}
\TV(P^{\otimes M}_{\alpha_0}, P^{\otimes M}_{\alpha_1})
& \stepa{\leq} 1-\frac{1}{2}\exp\{-\chi^2(P^{\otimes M}_{\alpha_0} \| P^{\otimes M}_{\alpha_1})\} \\
& = 1 - \frac{1}{2}\exp\{-(1+\chi^2(P_{\alpha_0} \| P_{\alpha_1}))^M+1\} \\
& \stepb{\leq} 1 - \frac{1}{2}\exp\{-(1+4\delta^2\TV(P, P'))^M+1\},
\end{align*}
where 
(a) follows from the inequality between the total variation and the $\chi^2$-divergence $\chi^2(P\|Q) \triangleq \int (\frac{dP}{dQ}-1)^2 dQ$ \cite[Eqn.~(2.25)]{Tsybakov2009};
(b) follows from 
\begin{align*}
\chi^2(P_{\alpha_0} \| P_{\alpha_1})
& = \chi^2\left(\frac{P+P'}{2} + \delta(P-P') \Big\| \frac{P+P'}{2}\right) \\
& = \delta^2\int \frac{(P-P')^2}{\frac{P+P'}{2}} \leq 4\delta^2\TV(P, P').
\end{align*}

Choosing $ \delta = \frac{1}{2} \wedge \sqrt{\frac{1}{4M\TV(P_{\tH}, P_{\tH'})}} $ and in view of the assumptions that $ M \geq 300 $ and $ \TV(P, P') \leq 1/300 $, 
the right-hand size of \prettyref{eq:fuzzy} is at least
\begin{equation*}
\frac{1}{4} \exp\{-(1+4\delta^2\TV(P, P'))^M+1\} - e^{-\delta^2M/8} \geq 0.01,
\end{equation*}
which proves \prettyref{eq:mainlb}.
\end{proof}

\begin{proof}[Proof of \prettyref{lmm:kocay}]
We use Kocay's Vertex Theorem \cite{Kocay1982} which says that if $ \calH $ is a collection of graphs, then
\begin{equation*}
\prod_{h\in \calH} \s(h, G) = \sum_{g}a_g \s(g,G),
\end{equation*}
where the sum runs over all graphs $ g $ such that $ \vertex(g) \leq \sum_{h\in \calH}\sfv(h) $ and $ a_g $ is the number of decompositions of $ V(g) $ into $ \cup_{h\in\calH}V(h) $ such that $ g[V(h)] \simeq h $. 

In particular, if $ \calH $ consists of the connected components of $ H $, then the only disconnected $ g $ with $ \vertex(g) = v $ satisfying the above decomposition property is $ g \simeq H $. Hence
\begin{equation*}
\s(H, G) = \frac{1}{a_H}\left[\prod_{h\in \calH} \s(h, G) - \sum_{g}a_g \s(g,G)\right],
\end{equation*}
where the sum runs over all $ g $ that are either connected and $ \vertex(g) \leq v $ or disconnected and $ \vertex(g) \leq v-1 $. This shows that $ \s(H, G) $ can be expressed as a polynomial, independent of $ G $, in $ \s(g, G) $ where either $ g $ is connected and $ \vertex(g) \leq v $ or $ g $ is disconnected and $ \vertex(g) \leq v-1 $. 

The proof proceeds by induction on $v$. The base case of $v=1$ is clearly true. 
Suppose that for any disconnected graph $ h $ with at most $ v $ vertices, $ \s(h,G) $ can be expressed as a polynomial, independent of $ G $, in $ \s(g,G) $ where $ g $ is connected and $ \vertex(g) \leq v $. By the first part of the proof, if $ H $ is a disconnected graph with $ v+1 $ vertices, then $ \s(H, G) $ can be expressed as a polynomial, independent of $ G $, in $ \s(h, G) $ where either $ h $ is connected and $ \sfv(h) \leq v+1 $ or $ h $ is disconnected and $ \sfv(h) \leq v $. By $ S(v) $, each $ \s(h, G) $ with $ h $ disconnected and $ \sfv(h) \leq v $ can be expressed as a polynomial, independent of $ G $, in $ \s(g, G) $ where $ g $ is connected and $ \vertex(g) \leq v $. Thus, we can express $ \s(H,G) $ as a polynomial, independent of $ G $, in terms of $ \s(g,G) $ where $ g $ is connected and $ \vertex(g) \leq v+1 $.
\end{proof}

\begin{proof}[Proof of \prettyref{lmm:matching}]
By \prettyref{cor:counts}, we have
\begin{equation}
\s(h, H) = \s(h, H'),
\label{eq:matchingind}
\end{equation}
for all $ h $ (not necessarily connected) with $ \sfv(h) \leq k $. 
Note that conditioned on $\ell$ vertices are sampled, 
$\tH$ is uniformly distributed over the collection of all induced subgraphs of $H$ with $\ell$ vertices. Thus
\[
\prob{\tH \simeq h \mid \sfv(\tH)=\ell}  = \frac{\s(h, H) }{\binom{m}{\ell}}.
\]
In view of \prettyref{eq:matchingind}, we conclude that the isomorphism class of $\tH$ and $\tH'$ have the same distribution provided that no more than $k$ vertices are sampled. Hence the first inequality in \prettyref{eq:tvmatch} follows, while the last inequality therein follows from the union bound $\prob{\Binom(m,p) \geq \ell} \leq \binom{m}{\ell} p^{\ell}$. The bound \prettyref{eq:exp-ineq} follows directly from Hoeffding's inequality on the binomial tail probability in \prettyref{eq:tvmatch}.
% the fact that $ \Indc\{ x \geq \ell \} \leq \binom{x}{\ell} $ and hence $\prob{\Binom(m,p) \geq \ell} \leq \expect{\binom{\Binom(m,p)}{\ell}} = \binom{m}{\ell} p^{\ell}$.
%\begin{equation*}
%\sum_{i=k+1}^m \binom{m}{i}p^iq^{m-i} \leq p^{k+1}\sum_{i=0}^m \binom{m}{i}q^{m-i} = p^{k+1}(1+q)^m \leq 2^mp^{k+1}.
%\end{equation*}
%We can calculate $ \TV(P, P') =\TV(P_{\tH}, P_{\tH'}) $ explicitly as
%
%\begin{equation*}
%\frac{1}{2}\sum_{h:\sfv(h) \leq m}|\s(h, H)-\s(h, H')|p^{\sfv(h)}(1-p)^{m-\sfv(h)}.
%\end{equation*}
%Since $ \s(h, H) = \s(h, H') $ for all $ h $ with $ \sfv(h) \leq k $, it follows that $ \TV(P_{\tH}, P_{\tH'}) $ is in fact equal to
%
%\begin{equation*}
%\frac{1}{2}\sum_{h:k+1 \leq \sfv(h)\leq m}|\s(h, H)-\s(h, H')|p^{\sfv(h)}(1-p)^{m-\sfv(h)}.
%\end{equation*}
%The triangle inequality can be used to further bound $ \TV(P_{\tH}, P_{\tH'}) $ by
%\begin{equation*}
%\frac{1}{2}\sum_{h:k+1 \leq \sfv(h)\leq m} (\s(h, H)+\s(h, H'))p^{\sfv(h)}(1-p)^{m-\sfv(h)} = \sum_{v=k+1}^{m}\binom{m}{v}p^vq^{m-v}.
%\end{equation*}
\end{proof}

\section{Additional results}\label{app:results}

In this appendix, we provide results for the uniform sampling model in \prettyref{sec:uniform}. We also discuss additional lower bound conclusions for graphs with long cycles in \prettyref{sec:cycle} and forests in \prettyref{sec:forestlb}.

\subsection{Extensions to uniform sampling model} \label{sec:uniform}

As we mentioned in \prettyref{sec:subgraph}, the uniform sampling model, where $ n $ vertices are selected uniformly at random from $ G $, is similar to Bernoulli sampling with $ p = n/N $. For this model, the unbiased estimator analogous to \prettyref{eq:chordal-estimator} is
\begin{equation} \label{eq:cc-estimator-without}
\cchat_U = \sum_{i\geq 1} \frac{\pth{-1}^{i+1}}{p_i} \s(K_i, \widetilde{G}),
\end{equation}
where $ p_i \triangleq \frac{\binom{N-i}{n-i}}{\binom{N}{n}} $.
Next we show that this unbiased estimator enjoys the same variance bound in \prettyref{thm:chordal} up to constant factors that only depend on $ \omega $. The proof of this result if given in \prettyref{app:proofs}.
%We will show in this subsection that this unbiased estimator can be used in lieu of \prettyref{eq:chordal-estimator} when the data $ \tG $ is generated from sampling vertices without replacement. Furthermore, its performance is the same, up to constants in $ \omega $ (c.f., \prettyref{thm:chordal}).

\begin{theorem} \label{thm:var-uniform}
Let $ \tG $ be generated from the uniform sampling model with $ n = pN $. Then
\begin{equation*}
\Var [\cchat_U ] = O_{\omega}\left(\frac{N}{p^{\omega}} + \frac{N\de}{p^{\omega-1}}\right).
\end{equation*}
\end{theorem}

\begin{proof}
Using $(a_1+\cdots + a_k)^2 \leq k (a_1^2 + \cdots + a_k^2)$, we have
	\begin{equation} \label{eq:var-cc-unif}
	\Var[\cchat_U]
	\leq \omega \cdot \sum_{i=1}^{\omega} \frac{\Var[\s(K_i, \tG)]}{p^{2}_i}.
	\end{equation}
Next, each variance term can be bounded as follows. 
Let $b_v = \Indc\{v \in S\} \sim \Bern(p)$.
Note that
	\begin{align}
	\Var[\s(K_i, \tG)] & = \Var\left[  \sum_{T: \; G[T] \simeq K_i} \prod_{v\in T} b_v \right] \nonumber \\
	 & =  \sum_{T: \; G[T] \simeq K_i} \Var\left[\prod_{v\in T} b_v\right] + \sum_{k=0}^{i-1}\sum_{ \substack{T \neq T' : \; |T\cap T'| = k, \\ \; G[T] \simeq K_i, \; G[T'] \simeq K_i}}\Cov\left[ \prod_{v\in T} b_v, \prod_{v'\in T'} b_{v'} \right] \nonumber \\
	& = \s(K_i, G)p_{i,i} + 2\sum_{k=0}^{i-1}\n(T_{i,k},G) p_{i, k},
	\label{eq:var-uniform}
	\end{align}
	where 
	%\nbr{I am not sure if the last step is =. It is certainly $\leq$. Are you sure there is no overcounting? It would be good to get to the bottom of it.} \nb{They are equal. For any edge induced $ T_{i,k} $, we can identify the $ k $ "overlap" vertices and associate two sets $ T $ and $ T' $ such that $ |T\cap T'| = k $ and the edge induced subgraph of $ T\cup T' $ is isomorphic to $ T_{i,k} $.}
\begin{equation*}
p_{i,k} \triangleq p_{2i-k} - p_i^2 =
\frac{\binom{N-2i+k}{n-2i+k}}{\binom{N}{n}} - \left(\frac{\binom{N-i}{n-i}}{\binom{N}{n}}\right)^2, \quad 0\leq k \leq i \leq n,
\end{equation*}
$T_{i,k}$ denotes two $K_i$'s sharing $k$ vertices, and we recall that $\n(H,G)$ notes the number of embeddings of (edge-induced subgraphs isomorphic to) $H$ in $G$.
%we recall that $ \n(H,G) $ denotes the number of edge induced copies of $ H $ in $ G $). 
It is readily seen that $ \frac{p_{i,k}}{p^2_i} \leq \frac{i!}{p^k} $ since
\begin{equation*} \label{eq:pik-bound}
\frac{p_{i,k}}{p^2_i} \leq 
\frac{p_{2i-k}}{p_i^2}
%= \frac{\frac{\binom{N-2i+k}{n-2i+k}}{\binom{N}{n}}}{\left(\frac{\binom{N-i}{n-i}}{\binom{N}{n}}\right)^2}
= \frac{\binom{N-2i+k}{n-2i+k}}{\binom{N-i}{n-i}} \frac{\binom{N}{n}}{\binom{N-i}{n-i}}
= \frac{\prod_{j=i+1}^{2i-k} \frac{n-j+1}{N-j+1}}{\prod_{j=1}^{i} \frac{n-j+1}{N-j+1}} \leq \frac{\prod_{j=i+1}^{2i-k}\frac{n}{N}}{\prod_{j=1}^{i}\frac{n}{jN}} = \frac{i!}{p^k},
\end{equation*}
where we used $p=n/N$ and the inequalities $ \frac{n}{jN} \leq \frac{n-j+1}{N-j+1} \leq \frac{n}{N} $ for $ 1 \leq j \leq (1+\frac{1}{N})n $. Furthermore, from the same steps,
for $k=0$ we have 
\[
\frac{p_{2i}}{p_i^2} = \prod_{j=1}^{i} \frac{ \frac{n-j+1-i}{N-j+1-i}}{\frac{n-j+1}{N-j+1}} \leq 1,
\]
or equivalently, $ p_{i,0} \leq 0 $, which also follows from negative association.

%One can verify that the sequence $f(i) = - \log p_i $ is strictly convex over $ i = 0, 1, \dots, n $ in the sense that $ f(i+1) - f(i) > f(i) - f(i-1) $ for $ i = 1, 2, \dots, n-1 $. Thus, in particular, $ f(i) < \frac{1}{2} f(0) + \frac{1}{2} f(2i) = \frac{1}{2} f(2i) $; or equivalently, $ p_{i,0} < 0 $. 
Substituting $ p_{i,0} \leq 0 $ and $ \frac{p_{i,k}}{p^2_i} \leq \frac{i!}{p^k} $ into \prettyref{eq:var-uniform}
yields
\begin{align}
\frac{1}{p^2_i}\Var[\s(K_i, \tG)] & = \frac{\s(K_i, G)p_{i,i}}{p^2_i} + 2\sum_{k=0}^{i-1}\n(T_{i,k},G) \frac{p_{i, k}}{p^2_i} \nonumber \\
& \leq \frac{\s(K_i, G)p_{i,i}}{p^2_i} + 2\sum_{k=1}^{i-1}\n(T_{i,k},G) \frac{p_{i, k}}{p^2_i} \nonumber \\
& \leq i!\left(\frac{\s(K_i, G)}{p^i} + 2\sum_{k=1}^{i-1}\frac{\n(T_{i,k},G)}{p^k}\right). \label{eq:var-complete-unif}
\end{align}

%To this end, note that
%\begin{align*}
%p_{i,0} & = \frac{\binom{N-2i}{n-2i}}{\binom{N}{n}} - \left(\frac{\binom{N-i}{n-i}}{\binom{N}{n}}\right)^2 = \prod_{j=1}^{2i} %\frac{n-j+1}{N-j+1} - \left(\prod_{j=1}^{i}\frac{n-j+1}{N-j+1}\right)^2 \\
%& = \prod_{j=1}^{i} \frac{n-j+1}{N-j+1}\left(\prod_{j=1}^{i} \frac{n-j-i+1}{N-j-i+1} - \prod_{j=1}^{i} \frac{n-j+1}{N-j+1} \right) \\
%& \leq 0,
%\end{align*}
%since $ \frac{n-j-i+1}{N-j-i+1} \leq \frac{n-j+1}{N-j+1} $. 

To finish the proof, we establish two combinatorial facts:
\begin{align}
\s(K_i, G)  = & ~ O_{\omega}(N), \quad i = 1, 2, \dots, \omega  \label{eq:a}\\
\n(T_{i, k},G) = & ~ O_{\omega}(Nd), \quad k = 1, 2, \dots, i-1 \label{eq:b}
\end{align}
Here \prettyref{eq:a} follows from the fact that for any chordal graph $G$ with clique number bounded by $\omega$, the number of cliques of any size is at most $O_\omega(|\sfv(G)|) = O_{\omega}(N) $. This can be seen from the PEO representation in \prettyref{eq:peo-complete-count} since $ \sfc_j \leq \omega-1 $.
To show \prettyref{eq:b}, note that to enumerate $T_{i,k}$, we can first enumerate cliques of size $i$, then for each clique, choose $i-k$ other vertices in the neighborhood of $k$ vertices of the clique. Note that for each $v\in V(G)$, the neighborhood of $v$ is also a chordal graph of at most $d$ vertices and clique number at most $\omega$. Therefore, by \prettyref{eq:a}, the number of $K_{i-k}$'s in the neighborhood of any given vertex is at most $O_{\omega}(d)$.

Finally, applying \prettyref{eq:a}--\prettyref{eq:b} to each term in \prettyref{eq:var-complete-unif}, we have
\[
\frac{1}{p^{2}_i} \Var[\s(K_i, \tG)] = O_{\omega} \left(\frac{N}{p^i} + \sum_{k=1}^{i-1} \frac{Nd}{p^{k}} \right) = O_{\omega}\left( 
\frac{N}{p^i} + \frac{Nd}{p^{i-1}} \right),
\]
which, in view of \prettyref{eq:var-cc-unif}, yields the desired result.
%The final result follows from summing the above bound from $ i = 1 $ to $ \omega $ and using the inequality 
\end{proof}

\subsection{Lower bound for graphs with long induced cycles} \label{sec:cycle}

\begin{theorem} \label{thm:cyclelower}
Let $ \calG(N, r) $ denote the collection of all graphs on $ N $ vertices with longest induced cycle at most $ r $, with $ r \geq 4 $. Suppose $ p < 1/2 $ and $ r \geq \frac{6}{(1-2p)^2} $. Then
\begin{equation*}
\inf_{\cchat}\sup_{G\in\calG(N, r)} \Expect_G|\cchat-\cc(G)|^2 \gtrsim Ne^{r(1-2p)^2} \wedge \frac{N^2}{r^2}.
\end{equation*}
In particular, if $ p < 1/2 $ and $ r = \Theta(\log N) $, then
\begin{equation*}
\inf_{\cchat}\sup_{G\in\calG(N, r)} \Expect_G|\cchat-\cc(G)| \gtrsim \frac{N}{\log N}.
\end{equation*}
\end{theorem}

\begin{proof}
We will prove the lower bound via \prettyref{thm:mainlb} with $ m = 2(r-1) $. Let $ H = C_{r}+ P_{r-2}$ and $ H' = P_{2(r-1)}$. Note that $ \s(P_i, H) = \s(P_i, H') = 2r-1-i $ for $ i = 1,2,\dots, r-1 $. For an illustration of the construction when $ r = 5 $, see \prettyref{fig:cycle-lower}. Since paths of length at most $ r-1 $ are the only connected subgraphs of $ H $ and $ H' $ with at most $ r - 1 $ vertices, \prettyref{cor:counts} implies that $ H $ and $ H' $ have matching subgraph counts up to order $ r-1 $.

\begin{figure} [ht]
\centering
\begin{subfigure}[t]{0.25\textwidth}
  \centering
  \includegraphics[width=0.5\linewidth]{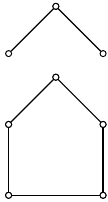}
  \caption{The graph $H$}
  \label{fig:fig1}
\end{subfigure}%
\hspace{1cm}
\begin{subfigure}[t]{0.4\textwidth}
  \centering
  \includegraphics[width=0.5\linewidth]{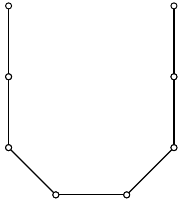}
  \caption{The graph $H'$.}
  \label{fig:fig2}
\end{subfigure}
\caption{The construction for $ r = 5 $. Each connected subgraph with $ k \leq 4 $ vertices appears exactly $ 9-k $ times in each graph.}
\label{fig:cycle-lower}
\end{figure}

In the notation of \prettyref{thm:mainlb}, $ k = r-1 $, $ m = 2(r-1) $, and $ |\cc(H)-\cc(H')| = 1 $. By \prettyref{thm:mainlb},
\begin{equation*}
\inf_{\cchat}\sup_{G\in\calG(N, r)}\prob{|\cchat-\cc(G)| \geq \Delta } \geq 0.10,
\end{equation*}
where
\begin{equation*}
\Delta \asymp |\cc(H)-\cc(H')|\left({\sqrt{\frac{N}{m\TV(P_{\tH}, P_{\tH'})}}}  \wedge \frac{N}{m} \right) = \pth{\sqrt{\frac{N}{m\TV(P_{\tH}, P_{\tH'})}}  \wedge \frac{N}{m}}.
\end{equation*}
Furthermore, by \prettyref{eq:exp-ineq}, the total variation between the sampled graphs $ \tH $ and $ \tH' $ satisfies 
\begin{equation*} 
\TV(P_{\tH}, P_{\tH'} ) \leq e^{-\frac{r^2}{r-1}(1-2p+\frac{2p}{r})^2} \leq e^{-r(1-2p)^2} < 1/300,
\end{equation*}
provided $ p < 1/2 $ and $ r \geq \frac{6}{(1-2p)^2} $.
%\begin{align*}
%\TV(P_{\tH}, P_{\tH'})
%& \leq \prob{\Binom(2(r-1), p) \geq r} \\
%& = \sum_{i= r}^{2(r-1)} \binom{2(r-1)}{i}p^i(1-p)^{2(r-1)-i} \\
%& \leq (4p)^r.
%\end{align*}
The desired lower bound on the squared error follows from Markov's inequality.
\end{proof}

\subsection{Lower bounds for forests} \label{sec:forestlb}

Particularizing \prettyref{thm:chordallower} to $ \omega = 2 $, we obtain a lower bound which shows that the estimator for forests
$\cchat = \sfv(\tG)/p-\sfe(\tG)/p^2$ proposed by Frank \cite{Frank1978}
is minimax rate-optimal. As opposed to the general construction in \prettyref{thm:chordallower}, 
\prettyref{fig:forest} illustrates a simple construction of $ H $ and $ H' $ for forests. However, we still require that $ p $ is less than some absolute constant. Through another argument, we show that this constant can be arbitrarily close to one.

\begin{figure} [ht]
\centering
\begin{subfigure}[t]{0.2\textwidth}
  \centering
  \includegraphics[width=0.8\linewidth]{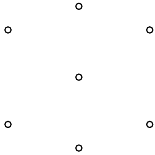}
  \caption{The graph of $H$ for $ \omega = 2 $ and $ m = 6 $.}
  \label{fig:figforest}
\end{subfigure}%
\hspace{3cm}
\begin{subfigure}[t]{0.2\textwidth}
  \centering
  \includegraphics[width=0.8\linewidth]{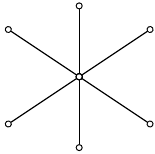}
  \caption{The graph of $H'$ for $ \omega = 2 $ and $ m = 6 $.}
  \label{fig:figforest2}
\end{subfigure}
\caption{The two graphs are isomorphic if the center vertex is not sampled and all incident edges are removed. Thus, $\TV(P_{\tH},P_{\tH'}) = p(1-q^6) $.}
\label{fig:forest}
\end{figure}

\begin{theorem}[Forests] \label{thm:forestlower}
Let $ \mathcal{F}(N,\de) = \calG(N,\de,2) $ denote the collection of all forests on $N$ vertices with maximum degree at most $ \de $. Then for all $ 0 < p < 1 $,
\begin{equation*}
\inf_{\cchat}\sup_{G\in\calF(N,\de)} \Expect_G|\cchat-\cc(G)|^2 \gtrsim \left(\frac{Nq}{p^2} \vee \frac{Nq\de}{p}\right) \wedge N^2.
\end{equation*}
In particular, if $ d = \Theta(N) $ and $ \omega \geq 2 $, then
\begin{equation*}
\inf_{\cchat}\sup_{G\in\calG(N,\de, \omega)} \Expect_G|\cchat-\cc(G)| \geq \inf_{\cchat}\sup_{G\in\calF(N,\de)} \Expect_G|\cchat-\cc(G)| \gtrsim N.
\end{equation*}
\end{theorem}
%\nbr{It seems through our new effort, we have the full result: 
%\begin{equation*}
%\inf_{\cchat}\sup_{G\in\calF(N,\de)} \Expect_G|\cchat-\cc(G)|^2 \asymp \left(\frac{Nq}{p^2} \vee \frac{Nq\de}{p}\right) \wedge N^2
%\end{equation*}
%which holds for all $0 \leq p \leq 1$ and all $1 \leq d \leq N$.
%The upper bound follows from the general result with $\omega=2$. Alternatively, we can write
%\begin{equation*}
%\inf_{\cchat}\sup_{G\in\calF(N,\de)} \Expect_G|\cchat-\cc(G)|^2 \asymp \left(\frac{Nq^2}{p^2} \vee \frac{Nq\de}{p}\right) \wedge N^2.
%\end{equation*}
%which should be equivalent.
%}

\begin{proof}
The strategy is to choose a one-parameter family of forests $ \calF_0 $ and reduce the problem to
 estimating the total number of trials in a binomial experiment with a given success probability. 
 %estimating the number of trials $ m $ in a binomial experiment with success probability $ p(1-q^d) $. 
To this end, define $ M = N/(d+1) $ and let
\begin{equation*}
\calF_0 = \{ (N-m(d+1))S_0 + mS_d : m \in \{0, 1, \dots, M\} \}.
\end{equation*}
Let $ G \in \calF_0 $. Because we do not observe the labels $ \{ b_v : v \in V(G) \} $, the distribution of $ \tG $ can be described by the vector $ (T_0, T_1, \dots, T_d) $, where $ T_j $ is the observed number of $ S_j $. Since $ T_0 = N - \sum_{j\geq 1} (j+1)T_j $, it follows that $ (T_1, \dots, T_d) $ is sufficient for $ \tG $. Next, we will show that $ T = T_1 + \cdots + T_d \sim \Binom(m, p') $, where $ p' \triangleq p(1-q^d) $ is sufficient for $ \tG $. To this end, note that conditioned on $T = n $, the probability mass function of $ (T_1, \dots, T_d)$ at $ (n_1, \dots, n_d) $ is equal to
\begin{align*}
\frac{\prob{T_1 = n_1, \dots, T_d = n_d, T = n}}{\prob{T = n}} & = \frac{\binom{m}{n}\binom{n}{n_1, \dots, n_d}p^{n_1}_1\cdots p^{n_d}_d(1-p')^{m-n}}{\binom{m}{n}(p')^n(1-p')^{m-n}} \\ & = \binom{n}{n_1, \dots, n_d}(p_1/p')^{n_1}\cdots (p_d/p')^{n_d},
\end{align*}
where $ p_j \triangleq \binom{d}{j} p^j q^{d-j} $. Thus, $ (T_1, \dots, T_d) \mid T = n \sim \text{Multinomial}(n, p_1/p', \dots, p_d/p' ) $, whose distribution is independent of $ m $.
Thus, since $ \cc(G) = N - md $, we have that
\begin{align*}
\inf_{\cchat}\sup_{G\in\calF(N,\de)} \Expect_G|\cchat-\cc(G)|^2 & \geq \inf_{\cchat}\sup_{G\in\calF_0} \Expect_G|\cchat-\cc(G)|^2 \\
& = d^2\inf_{\widehat{m}(T)}\sup_{m\in \{0, 1, \dots, M\} } \Expect_{T\sim \Binom(m, p')}|\widehat{m}(T)-m|^2 \\
% & \geq \frac{d^2q}{4} \left(\frac{Mq^2}{4p'} \wedge M^2 \right) = \frac{d^2q}{4(d+1)^2} \left(\frac{N(d+1)q^2}{4p(1-q^d)} \wedge N^2 \right) \\
& \gtrsim \left(\frac{Nq}{p^2} \vee \frac{Nq\de}{p}\right) \wedge N^2,
\end{align*}
which follows applying \prettyref{lmm:binom-lower} below with $ \alpha = p' $ and $ M = N/(d+1) $ and 
the fact that $p'= p(1-q^d) \leq p \wedge (p^2d) $.
\end{proof}

The proof of \prettyref{lmm:binom-lower} is given in \prettyref{app:proofs}.
\begin{lemma}[Binomial experiment] \label{lmm:binom-lower}
Let $ X\sim \Binom(m, \alpha) $. For all $ 0 \leq \alpha \leq 1 $ and $M \in \naturals$ known a priori,
\begin{equation*}
\inf_{\widehat{m}}\sup_{m\in \{0, 1, \dots, M\} } \Expect|\widehat{m}(X)-m|^2 \asymp \frac{(1-\alpha)M}{\alpha} \wedge M^2.
\end{equation*}
\end{lemma}

\begin{proof}
The upper bound follows from choosing $ \widehat{m} = X/\alpha $ when $ \alpha > (1-\alpha)/M $ and $ \widehat{m} = (M+1)/2 $ when $ \alpha \leq (1-\alpha)/M $.

For the lower bound, let $ \gamma > 0 $. Consider the two hypothesis $ H_1: m_1 = M $ and $ H_2: m_2 = M - \sqrt{\frac{\gamma M}{\alpha}} \wedge M $. By Le Cam's two point method \cite[Theorem 2.2(i)]{Tsybakov2009},
\begin{align*}
\inf_{\widehat{m}}\sup_{m\in \{0, 1, \dots, M\} } \Expect|\widehat{m}(X)-m|^2
& \geq \frac{1}{2}|m_1 - m_2|^2[1 - \TV( \Binom(m_1, \alpha), \Binom(m_2, \alpha))] \\
& \geq \frac{\frac{\gamma M}{\alpha} \wedge M^2}{2}[1 - \hellinger( \Binom(m_1, \alpha), \Binom(m_2, \alpha))],
\end{align*}
where we used the inequality between total variation and the Hellinger distance $ \TV \leq \hellinger$  \cite[Lemma 2.3]{Tsybakov2009}. Finally, choosing $ \gamma = (1-\alpha)/16 $ and using the bound in \cite[Lemma 21]{Polyanskiy2017} on the Hellinger distance between two binomials, we obtain 
\[
\hellinger( \Binom(m_1, \alpha), \Binom(m_2, \alpha)) \leq 1/2 
\]
as desired.
\end{proof}

\section{Numerical experiments} \label{app:experiments}

In this section, we study the empirical performance of the estimators proposed in \prettyref{sec:algorithms} using synthetic data from various random graphs. The error bars in the following plots show the variability of the relative error $ \frac{|\cchat-\cc(G)|}{\cc(G)} $ over 20 independent experiments of subgraph sampling on a fixed parent graph $ G $. The solid black horizontal line shows the sample average and the whiskers show the mean $ \pm $ the standard deviation. 
%Note that as $ p $ increases, the sampling variability decreases monotonically. The estimates perform poorly for small $ p $ (e.g., $ p < 0.2 $), but becomes increasingly better as $ p $ grows. The decay of relative error also appears to be non-linear, changing from very large to moderate improvements as $ p $ varies from $ 0 $ to $ 1 $.

\subsection{Synthetic experiment}

\paragraph{Chordal graphs}
Both \prettyref{fig:chordal} and \prettyref{fig:exp-smooth} focus on chordal graphs, where the parent graph is first generated from a random graph ensemble then triangulated by calculating a fill-in of edges to make it chordal (using a maximum cardinality search algorithm from \cite{Tarjan1984}).
In \prettyref{fig:fig1}, the parent graph $G$ is a triangulated Erd\"os-R\'enyi graph $ \calG(N, \delta) $, with $ N = 2000 $ and $ \delta = 0.0005 $ which is below the connectivity threshold $ \delta = \frac{\log N}{N} $ \cite{Erdos1960}. 
%In other words, to make $ \cc(G) $ large, we require that the edge connection probability $ \delta $ be sufficiently small.  
In \prettyref{fig:fig2}, we generate $ G $ with $N=20000$ vertices by taking the disjoint union of 200 independent copies of $ \calG(100, 0.2) $ and then apply triangulation.  In accordance with \prettyref{thm:chordal}, the better performance in \prettyref{fig:fig2} is due to moderately sized $ \de $ and $ \omega $, and large $ \cc(G) $.

In \prettyref{fig:exp-smooth} we perform a simulation study of the smoothed estimator $ \cchat_L $ from \prettyref{thm:chordal_smooth}. The parent graph is equal to a triangulated realization of $ \calG(1000, 0.0015) $ with $ \de = 88 $, $ \omega = 15 $, and $ \cc(G) = 325 $. The plots in \prettyref{fig:fig6} show that the sampling variability is significantly reduced for the smoothed estimator, particularly for small values of $ p $ (to show detail, the vertical axes are plotted on different scales). This behavior is in accordance with the upper bounds furnished in \prettyref{thm:chordal} and \prettyref{thm:chordal_smooth}. Large values of $ \omega $ inflate the variance of $ \cchat $ considerably by an exponential factor of $ 1/p^{\omega} $, whereas the effect of large $ \omega $ on the variance of $ \cchat_L $ is polynomial, viz., $ \omega^{\frac{p}{2-3p}} $. We chose the smoothing parameter $ \lambda $ to be $ p \log N $, but other values that improve the performance can be chosen through cross-validation on various known graphs.

The non-monotone behavior of the relative error in \prettyref{fig:fig5} can be explained by the tradeoff between increasing $ p $ (which improves the accuracy) and increasing probability of observing a clique (which increases the variability, particularly in this case of large $ \omega $). Such behavior is apparent for moderate values of $ p $ (e.g., $ p < 0.25 $), but less so as $ p $ increases to $ 1 $ since the mean squared error tends to zero as more of the parent graph is observed. The plots also suggest that the marginal benefit (i.e., the marginal decrease in relative error) from increasing $ p $ diminishes for moderate values of $ p $. Future research would address the selection of $ p $, if such control was available to the experimenter.

\paragraph{Non-chordal graphs}
Finally, in \prettyref{fig:nonchordal} we experiment with sampling non-chordal graphs. As proposed in \prettyref{sec:non-chordal}, one heuristic is to modify the original estimator by first triangulating the subsampled graph $ \tG$ to $\mathsf{TRI}(\tG) $ and then applying the estimator $ \cchat $ in \prettyref{eq:chordal-estimator-peo}.
 %to this transformed data via $ \cchat = \cchat(\mathsf{TRI}(G[S])) $. 
The plots in \prettyref{fig:nonchordal} show that this strategy works well; in fact the performance is competitive with the same estimator in \prettyref{fig:chordal}, where the parent graph is first triangulated and then subsampled.

\begin{figure} [H]
\centering
\begin{subfigure}[t]{0.4\textwidth}
  \centering
  \includegraphics[width=1\linewidth]{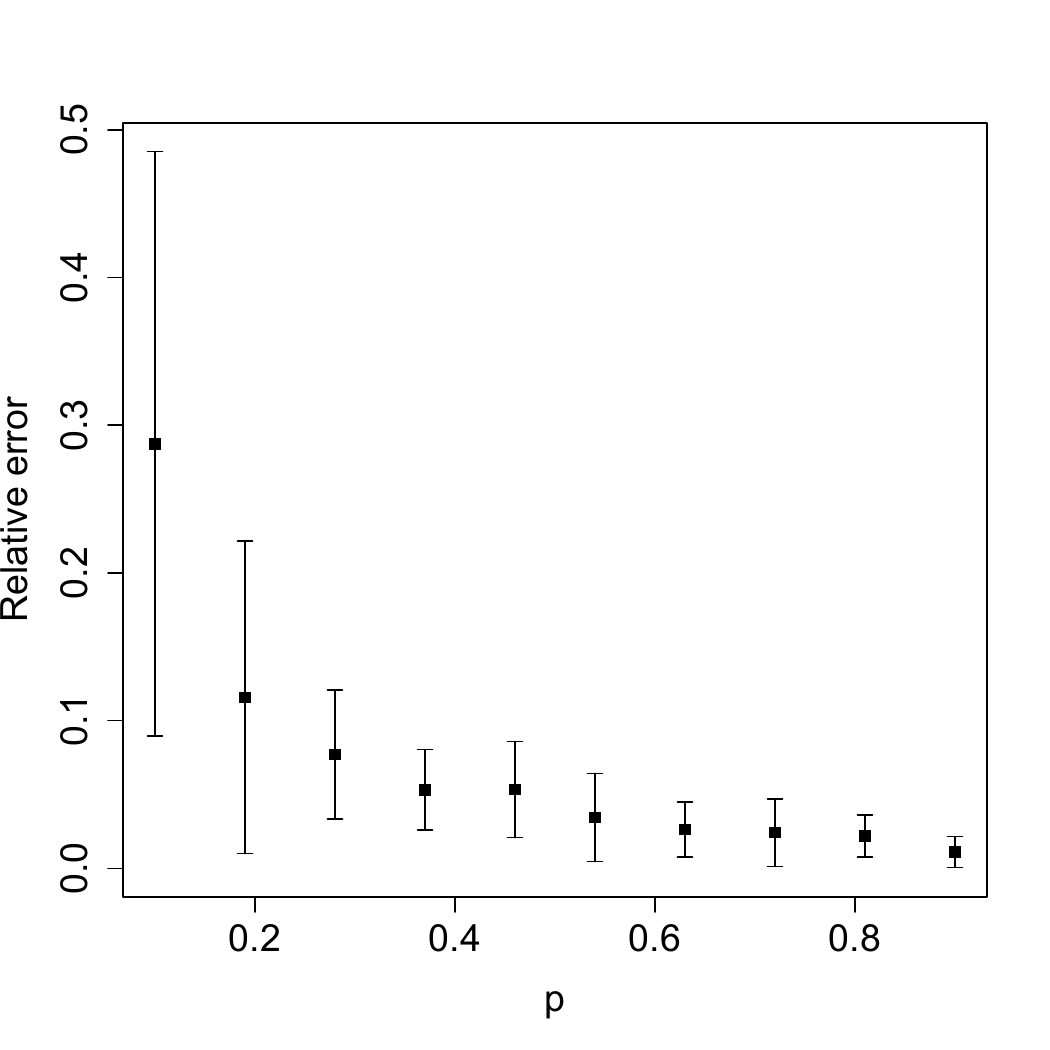}
  \caption{Parent graph equal to a triangulated realization of $ \calG(2000, 0.0005) $ with $ \de = 36 $, $ \omega = 5 $, and $ \cc(G) = 985 $.}
  \label{fig:fig1}
\end{subfigure}%
\qquad
\begin{subfigure}[t]{0.4\textwidth}
  \centering
  \includegraphics[width=1\linewidth]{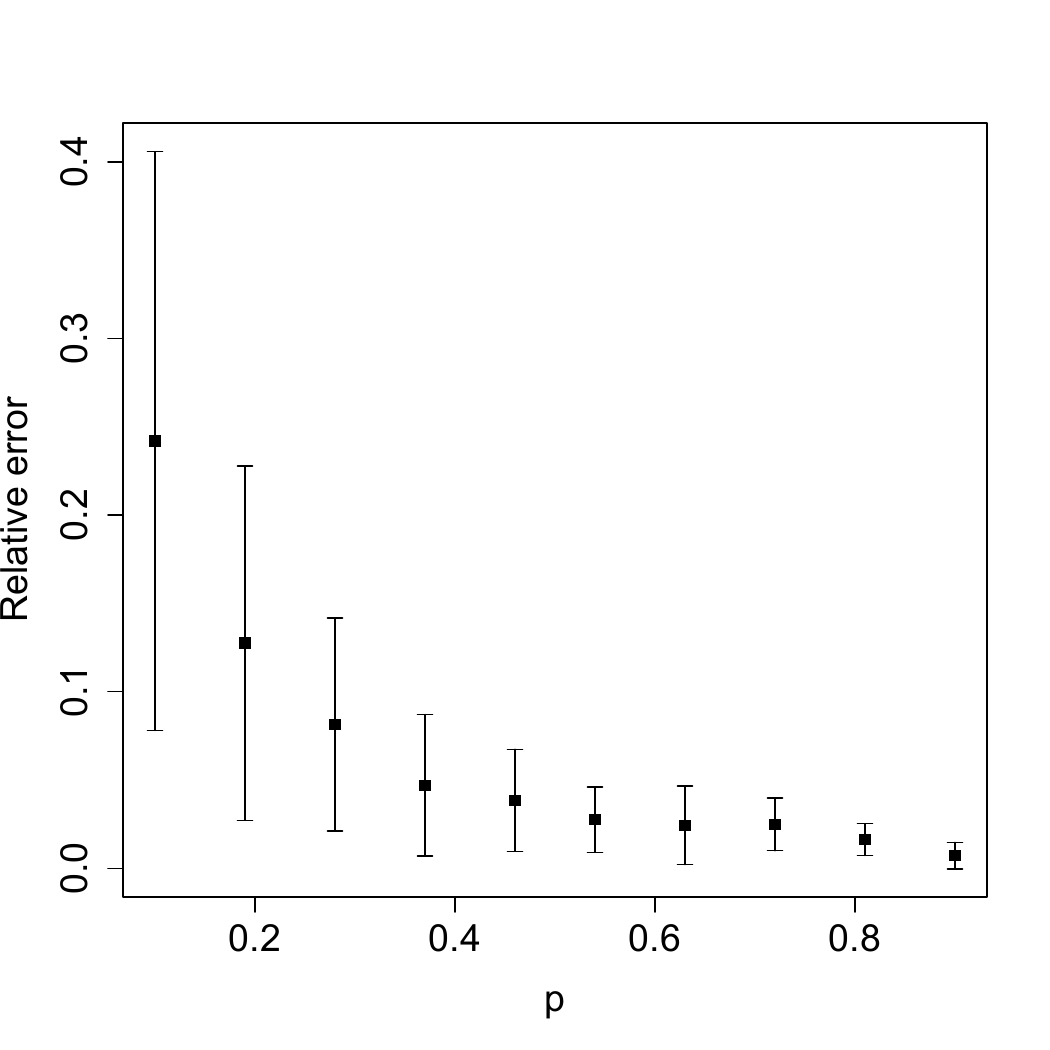}
  \caption{Parent graph equal to a triangulated realization of 200 copies of $ \calG(100, 0.2) $ with $ \de = 8 $, $ \omega = 4 $, and $ \cc(G) = 803 $.}
  \label{fig:fig2}
\end{subfigure}
\caption{The relative error of $ \cchat $ with moderate values of $ d $ and $ \omega $.}
\label{fig:chordal}
\end{figure}

\begin{figure} [H]
\centering
\begin{subfigure}[t]{0.4\textwidth}
  \centering
  \includegraphics[width=1\linewidth]{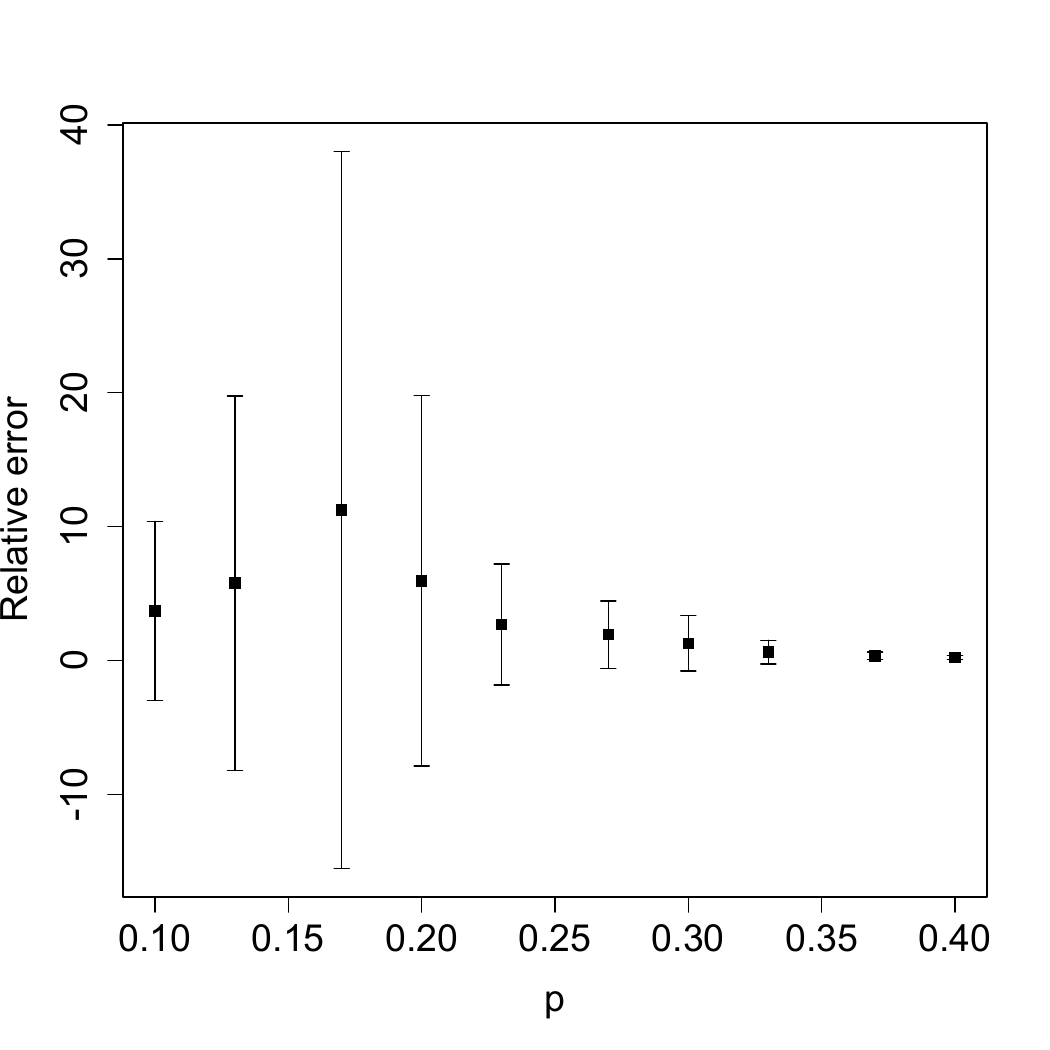}
  \caption{Non-smoothed $\cchat$.}
  \label{fig:fig5}
\end{subfigure}%
\qquad
\begin{subfigure}[t]{0.4\textwidth}
  \centering
  \includegraphics[width=1\linewidth]{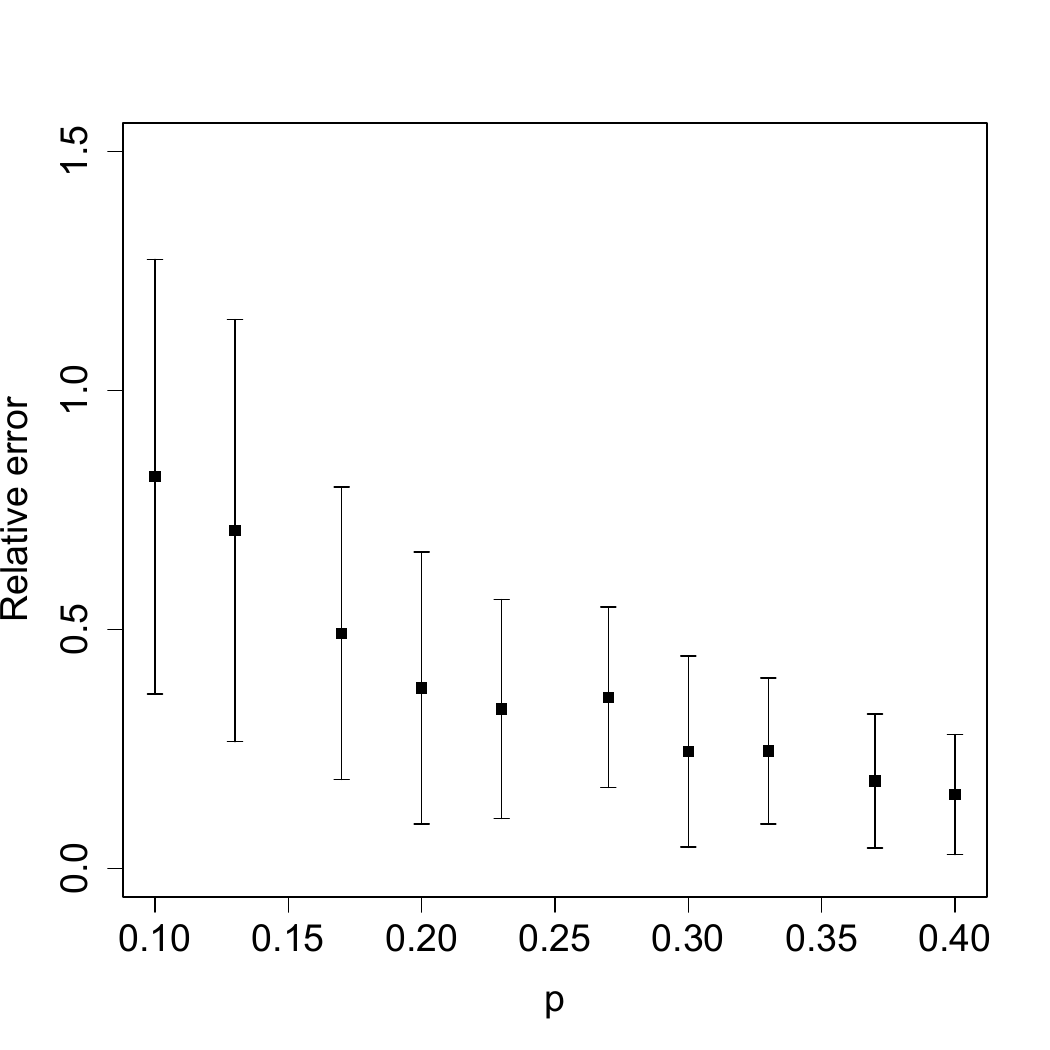}
  \caption{Smoothed $\widehat{cc}_L$.}
  \label{fig:fig6}
\end{subfigure}
\caption{A comparison of the relative error of the unbiased estimator $ \cchat $ in \prettyref{eq:chordal-estimator-peo} and its smoothed version $ \cchat_L $ in \prettyref{eq:chordal_smooth_estimator}. The parent graph is a triangulated realization of $ \calG(1000, 0.0015) $ with $ \de = 88 $, $ \omega = 15 $, and $ \cc(G) = 325 $.}
\label{fig:exp-smooth}
\end{figure}

\begin{figure} [H]
\centering
\begin{subfigure}[t]{0.4\textwidth}
  \centering
  \includegraphics[width=1\linewidth]{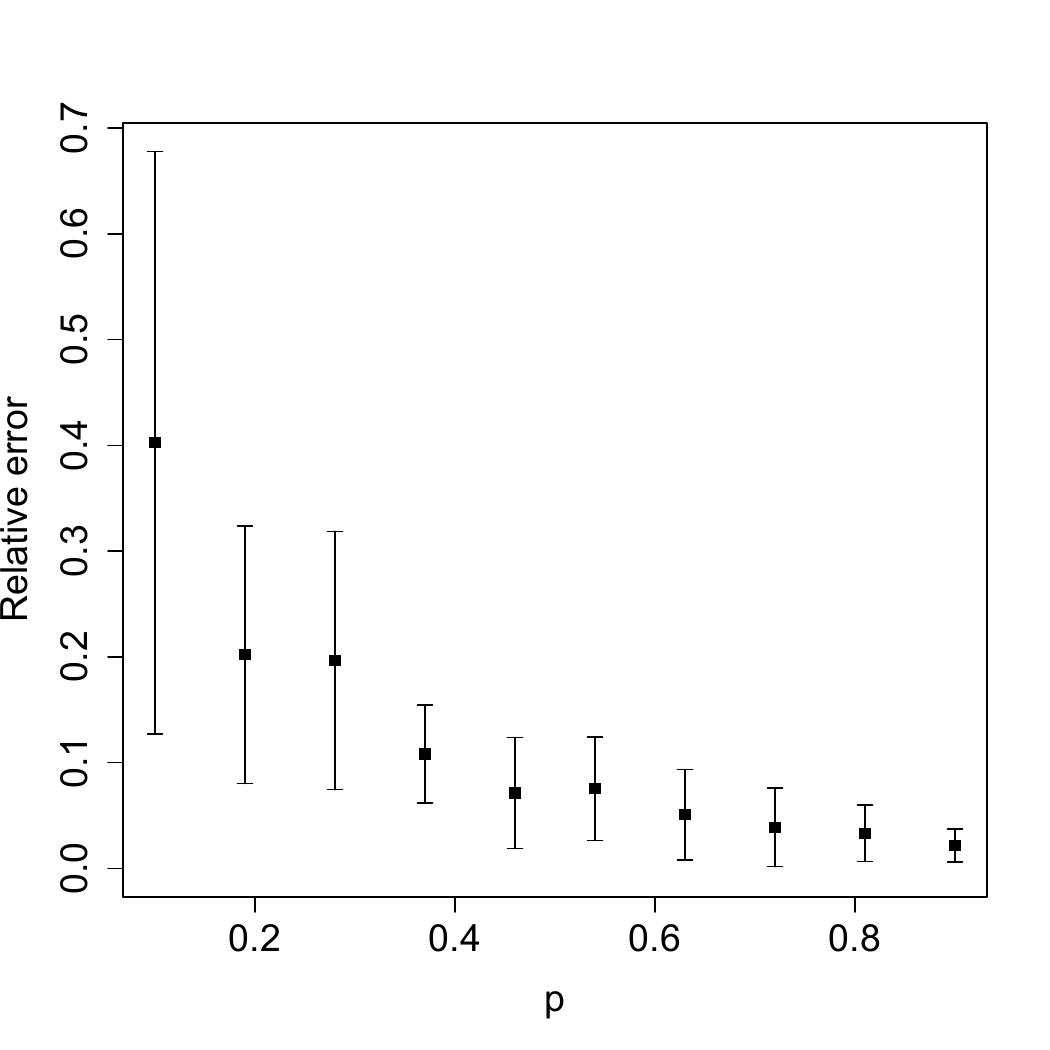}
  \caption{Parent graph equal to a realization of $ \calG(2000, 0.0005) $ with $ \de = 8 $, $ \omega = 3 $, and $ \cc(G) = 756 $.}
  \label{fig:fig3}
\end{subfigure}%
\qquad
\begin{subfigure}[t]{0.4\textwidth}
  \centering
  \includegraphics[width=1\linewidth]{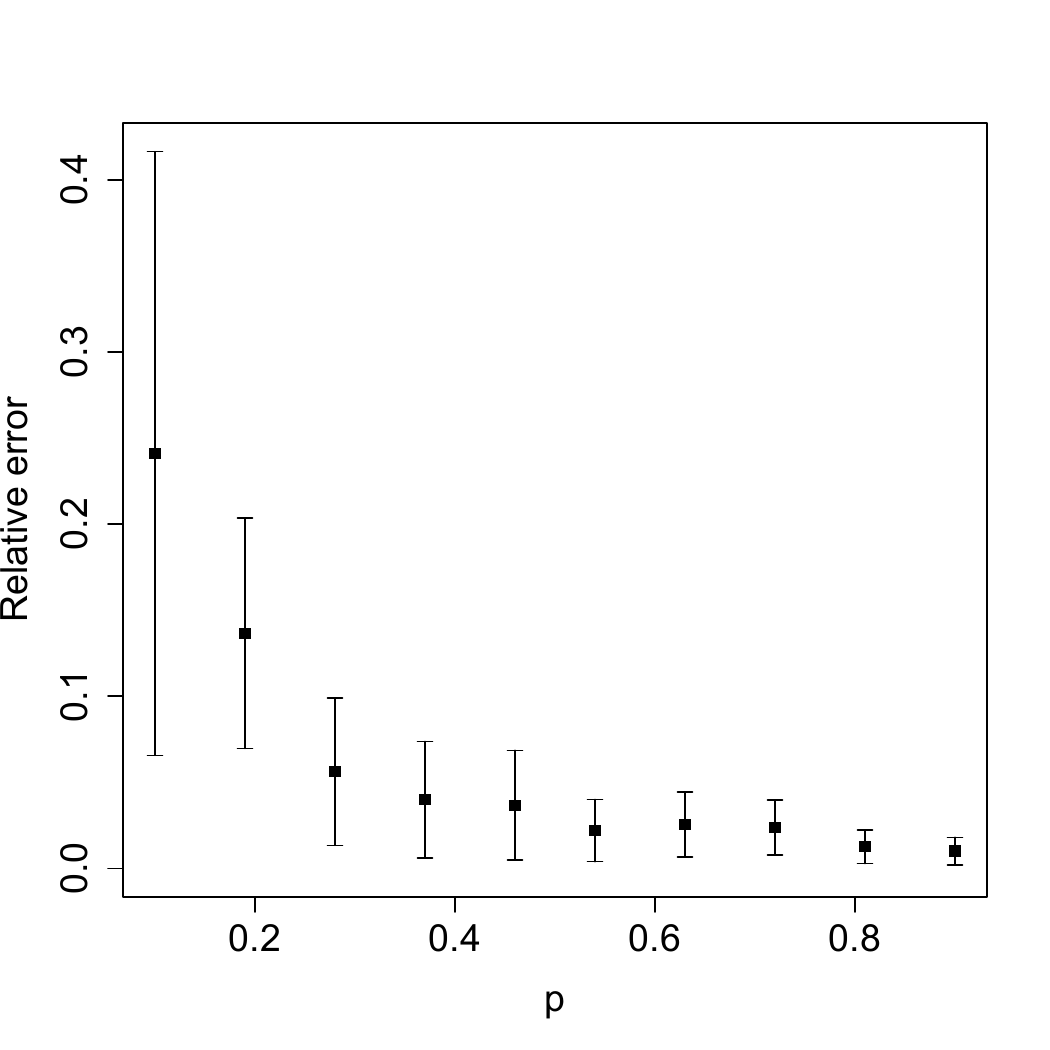}
  \caption{Parent graph equal to a realization of 200 copies of $ \calG(100, 0.2) $ with $ \de = 7 $, $ \omega = 4 $, and $ \cc(G) = 532 $.}
  \label{fig:fig4}
\end{subfigure}
\caption{The estimator  $\cchat(\mathsf{TRI}(\tG)) $ applied to non-chordal graphs. }
\label{fig:nonchordal}
\end{figure}

\subsection{Real-data experiment}

The point of developing a theory for graphs with specific structure (i.e.~chordal) is to (a) study how the graph parameters (such as maximal degree) impact the estimation difficulty and (b) motivate a heuristic for real-world graphs encountered in practice. Indeed, as with all problems in a minimax framework, a certain amount of finesse is required to define a parameter space that showcases the richness of the problem, while at the same time, enables one to provide a characterization of the fundamental limits of estimation. Chordal graphs seem to fit this purpose. Importantly, they serve as a catalyst for more general strategies that apply to a wider collection of parent graphs, including those commonly encountered in practice.

In the previous subsection, we studied the estimators $ \cchat $ and $ \cchat_L $ using synthetic data on moderately sized, chordal parent graphs. In this section, we consider real-world instances of network datasets, where the parent graph is not chordal and the number of nodes is large. More specifically, we consider two representative examples of collaboration and biological networks. We believe these examples show the usefulness of our estimators on real-world data, despite the fact that the methodology was developed for chordal parent graphs.

The first network \cite{collab} is the collaboration network of authors with arXiv ``cond-mat'' (condense matter physics) articles submitted between January 1993 and April 2003. Note that the category has been active since April 1992. An edge is attached between two researchers in the network if and only if they co-authored a paper together. 

The second network \cite[Supplementary Table S2]{Rual2005} is an initial version of a proteome-scale map of human binary protein-protein interactions (i.e., edges represent direct physical interactions between two protein molecules). Because self-loops do not affect the connectivity of the network, we removed them from the dataset.

We use the smoothed estimator $ \cchat_L $ on both networks; the standard estimator $ \cchat $ performs poorly because of high-degree vertices and large clique numbers (c.f., \prettyref{fig:exp-smooth}). To deal with the non-chordal parent graphs, we again use the heuristic proposed in \prettyref{sec:non-chordal} of first triangulating the subsampled graph $ \tG$ to $\mathsf{TRI}(\tG) $ and then applying the smoothed estimator $ \cchat_L $ in \prettyref{eq:chordal_smooth_estimator}. The results of this estimation scheme on both networks are displayed in \prettyref{fig:empirical}.

\begin{figure} [H]
\begin{subfigure}[t]{0.4\textwidth}
  \centering
  \includegraphics[width=1\linewidth]{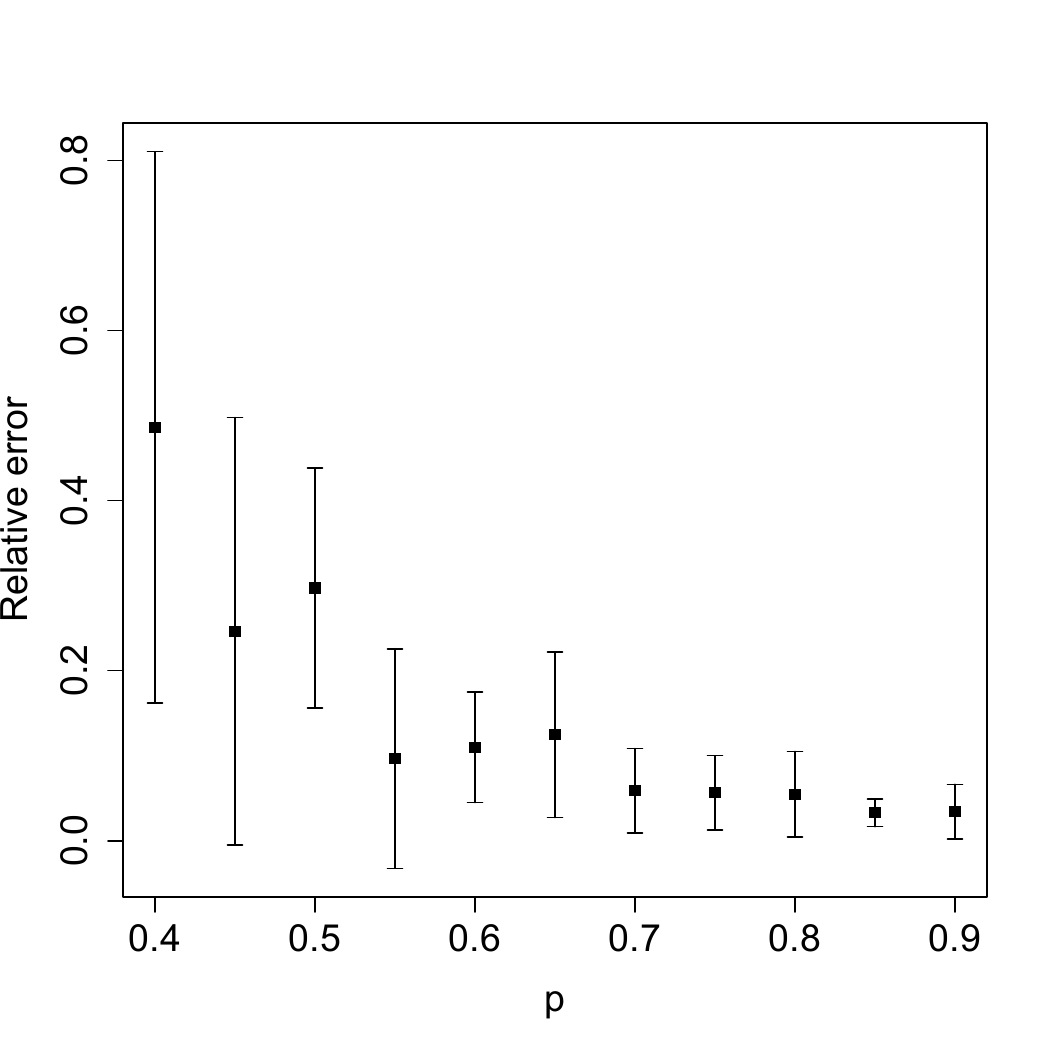}
  \caption{Collaboration network of arXiv condense matter physics: $ N = 23133 $, $ \sfe(G) = 93439 $, $ \de = 279 $, and $ \cc(G) = 567 $.}
  \label{fig:fig5}
\end{subfigure}%
\qquad
\centering
\begin{subfigure}[t]{0.4\textwidth}
  \centering
  \includegraphics[width=1\linewidth]{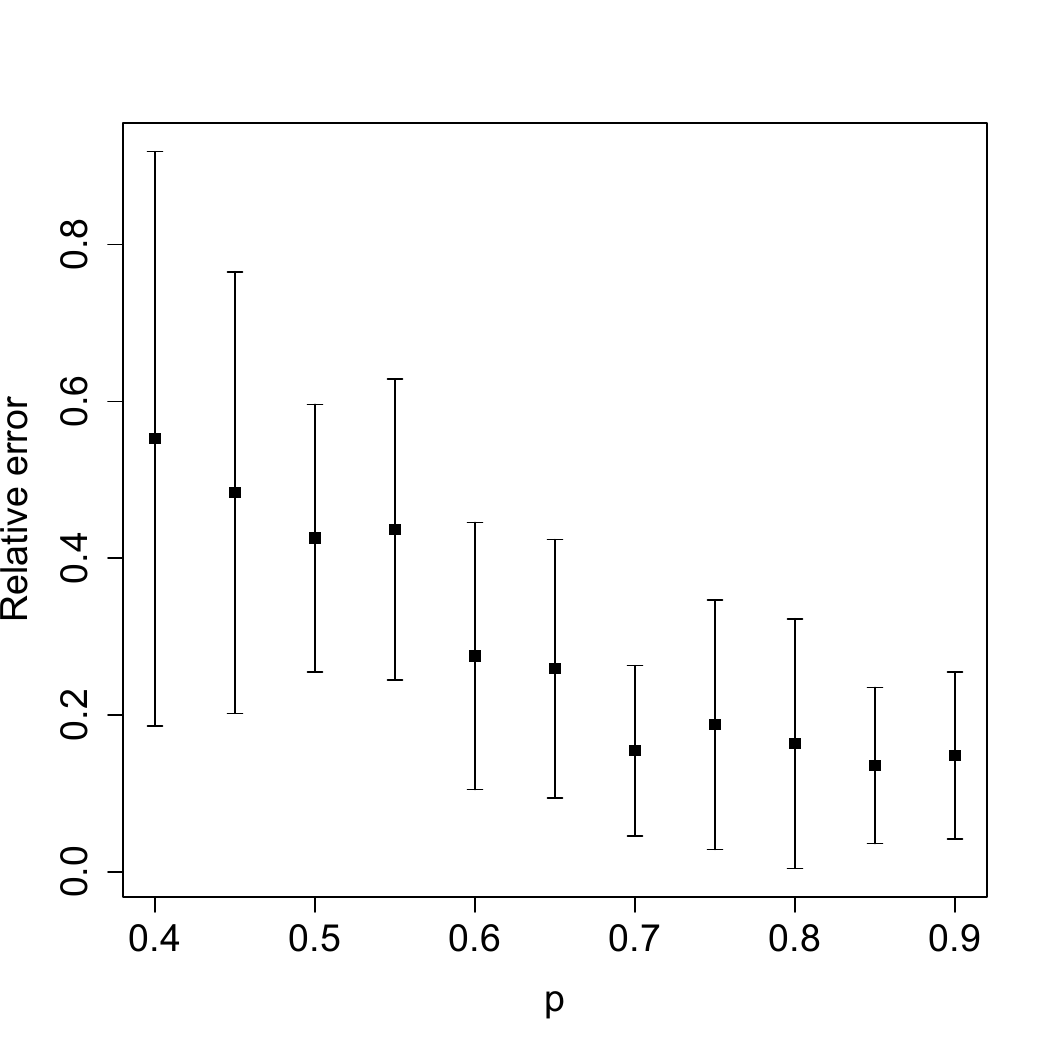}
  \caption{Human protein-protein network: $ N = 3133 $, $ \sfe(G) = 6149 $, $ \de = 129 $, and $ \cc(G) = 210 $.}
    \label{fig:fig6}
\end{subfigure}%
\caption{Smoothed estimator $\cchat_L(\mathsf{TRI}(\tG)) $ applied to a collaboration and biological network.}
\label{fig:empirical}
\end{figure}

\bibliographystyle{plain}
\bibliography{ccref}

\end{document}